\documentclass{article}
\usepackage[utf8]{inputenc}
\usepackage[margin=0.75in]{geometry}
\usepackage[title]{appendix}
\usepackage{amsthm}
\usepackage{url,color}
\usepackage{hyperref}
\usepackage{amsmath, amsfonts}
\usepackage{mathtools}
\usepackage{mathalfa}
\usepackage{amssymb}
\usepackage{mathrsfs}
\usepackage{relsize}
\usepackage[makeroom]{cancel}
\usepackage{tikz}
\usetikzlibrary{positioning}

\newcommand{\dell}{\partial}

\newcommand{\Ndell}[2]{\rdell^{#1}\angdell^{\underline{#2}}}

\newcommand{\comm}[2]{\left[#1,#2\right]}
\newcommand{\X}[1]{\mathcal{X}^{#1}}
\newcommand{\Y}[3]{\mathcal{Y}^{#1}_{#2}(#3)}
\newcommand{\angdell}{\cancel{\partial}}
\newcommand{\rdell}{\Lambda}
\newcommand{\cov}{\mathcal{M}}
\newcommand{\inv}{\mathcal{A}}
\newcommand{\jac}{\mathcal{J}}
\newcommand{\ncov}{\mathscr{M}}
\newcommand{\ninv}{\mathscr{A}}
\newcommand{\njac}{\mathscr{J}}
\newcommand{\pjac}{\mathscr{J}^{-1/\alpha}}

\newcommand{\spaceI}{\int_{\Omega}}

\newcommand{\twopartdef}[4]
   {
    \left\{
      \begin{array}{ll}
        #1 & \text{if } #2 \\
        #3 & \text{if } #4
      \end{array}
    \right.
   }

\DeclarePairedDelimiterX{\inp}[2]{\langle}{\rangle}{#1, #2}
\DeclareMathOperator{\tr}{Tr}
\DeclareMathOperator{\dist}{dist}
\DeclareMathOperator{\supp}{supp}
\DeclareMathOperator{\dive}{div}
\DeclareMathOperator{\curl}{curl}
\DeclareMathOperator{\Curl}{Curl}
\DeclareMathOperator{\grad}{\nabla}

\DeclareMathOperator{\ngrad}{\nabla_{\zeta_{\kappa}}}
\DeclareMathOperator{\ndiv}{div_{\zeta_{\kappa}}}
\DeclareMathOperator{\ncurl}{curl_{\zeta_{\kappa}}}
\DeclareMathOperator{\nCurl}{Curl_{\zeta_{\kappa}}}
\DeclarePairedDelimiter\ceil{\lceil}{\rceil}

\newtheorem{definition}{Definition}[section]

\newtheorem{theorem}[definition]{Theorem}

\newtheorem{proposition}[definition]{Proposition}

\newtheorem{lemma}[definition]{Lemma}

\newtheorem{remark}[definition]{Remark}

\def\bcr{\begin{color}{red}}
\def\bcb{\begin{color}{blue}}
\def\ec{\end{color}}

\title{Global Existence for the $N$ Body Euler-Poisson System}
\author{Shrish Parmeshwar\footnote{Department of Mathematics, King's College London, Strand, London WC2R 2LS, UK}}
\date{March 2019}

\begin{document}

\maketitle

\begin{abstract}
	In this paper we investigate the problem of multiple expanding Newtonian stars that interact via their gravitational effect on each other. It is clear physically that if two stars at rest are separated initially, and start expanding as well as moving according to the laws of Newtonian gravity, they may eventually collide. Thus, one can ask whether each star can be given an initial position and velocity such that they can keep expanding without touching. We show that even with gravitational interaction between the bodies, a large class of initial positions and velocities give global-in-time solutions to the $N$ Body Euler-Poisson system. To do this we use a scaling mechanism present in the compressible Euler system shown in \cite{PaHaJa} and a careful analysis of how the gravitational interaction between stars affects their dynamics.
\end{abstract}

\section{Introduction}
The compressible Euler-Poisson system for an inviscid, isentropic, ideal gas, acting under the influence of its own gravity, in its most basic form, is given by
\begin{align}
\dell_{t}\rho+\nabla\cdot \left(\rho u\right)&=0,\label{general-euler-1}\\
\rho\left(\dell_{t}+u\cdot\nabla\right)u+\nabla p+\rho\nabla\phi&=0,\label{general-euler-2}\\
\Delta\phi=4\pi\rho,\ \ \lim\limits_{|x|\rightarrow\infty}\phi(t,x)&=0,\label{general-euler-3}
\end{align}
where $\rho, u$, $p$, and $\phi$ are the fluid density, velocity, pressure, and gravitational potential respectively. Newton's gravitational constant $G$ will be set to $1$ here and hereafter in this work.
\begin{remark}
	Investigating the properties of stars modelled as bodies of fluids is a classical problem in astrophysics, with a long history. In the case of one body of incompressible fluid moving under its own gravity, equilibrium states have been studied by the likes of Newton, Maclaurin, Jacobi, Poincar\'{e}, and others. See, for example, $\cite{Ch3}$ for a summary of these results, and further context regarding this problem.
\end{remark}
\begin{remark}\label{plasma-case-remark}
	Note that if instead we had $\Delta\phi=-4\pi\rho$, $\phi$ would be referred to as the electrostatic potential, and the above would  be a toy model of a plasma instead of a Newtonian star. See $\cite{Guo1,GuoPau}$ for more details on the mathematical theory of plasmas.
\end{remark}
In addition to the basic system $\eqref{general-euler-1}$--$\eqref{general-euler-3}$, we will always work with a polytropic equation of state between the pressure $p$ and the density $\rho$:
\begin{align}
p(\rho)=\rho^{\gamma},\label{equation-of-state}
\end{align}
for some $\gamma>1$, the adiabatic index. The system $\eqref{general-euler-1}$--
$\eqref{equation-of-state}$ is one of the most fundamental and commonly used models of a Newtonian star (\cite{Ch2,ZelNov,BiTr}). In addition, the equation of state guarantees the system is not underdetermined. A famous class of solutions to $\eqref{general-euler-1}$--$\eqref{equation-of-state}$ are the so called Lane-Emden stars, obtained by looking for radially symmetric, time independent density profiles, and vanishing velocity:
\begin{align}
(\rho(t,x),u(t,x))=(\bar{\rho}(r),0).\label{lane-emden-solutions}
\end{align}
In the range $6/5 < \gamma < 2$, it is known that there always exists a solution with finite mass and compact support, whilst for $\gamma=6/5$ there exists a solution with finite mass and infinite support (\cite{Ch2,ZelNov,BiTr,Jang2014}). The fundamental problem of establishing linear or nonlinear stability for these solutions has also been investigated. Lin \cite{Lin} established linear instability for $\gamma\in(1,4/3)$, and linear stability for $\gamma\in[4/3,2)$. It is also known that at $\gamma=4/3$, the Lane-Emden star solution is nonlinearly unstable (\cite{GolWeb,MakPer,DengLiuYangYao,Rein1,Jang2014,HaJa2}). Jang \cite{Jang2008,Jang2014} showed that for $\gamma\in[6/5,4/3)$, the Lane-Emden stars are nonlinearly unstable, whilst stability results for the range $(4/3,2)$ conditional on the existence of solutions close to the stationary ones have been shown in \cite{Rein1,LuoSmol2}.

In the mass critical case of $\gamma=4/3$, Goldreich and Weber \cite{GolWeb} exhibited another special class of spherically symmetric solutions that could either expand indefinitely, or collapse in finite time. These solutions are self similar with respect to the same rescaling under which $4/3$ is the mass critical exponent, and have $0$ energy. In further works for the same value of $\gamma$, Makino \cite{Mak2}, and Fu and Lin \cite{FuLin} constructed related solutions that also expand or collapse. Had\v{z}i\'{c} and Jang \cite{HaJa2} showed the nonlinear stability of these solutions in the expanding case under spherically symmetric perturbations. Corresponding solutions for the noninsentropic case, and spatial dimensions $\geq3$, have been constructed by Deng, Xiang, and Yang \cite{DengXiangYang}.

Another important family of solutions, that of axisymmetric rotating stars in equilibrium, has long been a subject of interest, with efforts to construct such stars beginning in the astrophysics community in the early 20th century, see \cite{Milne,Von,Ch1,Lich}. Auchmuty and Beals \cite{AuchBeals1,AuchBeals2} rigorously constructed solutions using a variational method. Given a prescribed angular velocity, or angular momentum per unit mass, they found axisymmetric densities $\rho$ that satisified the time independent version of the Euler-Poisson system $\eqref{general-euler-1}$--$\eqref{general-euler-3}$, for $\gamma>4/3$ (amongst other, non polytropic equations of state). Auchmuty \cite{Auch} showed the existence of equilibrium rotating solutions in the incompressible case using these methods. Li \cite{Li} also used these variational techniques to find related solutions with constant angular velocity, a case which was excluded by the assumptions made in the work of Auchmuty and Beals. McCann \cite{McCann} has shown that one can construct rotating binary star solutions in equilibrium using variational principles. Results have also been established on the shape of the support of these solutions, see \cite{CafFrie,FrieTurk1,FrieTurk2,ChanLi,ChanWeiss}. For further works in the variational framework, see \cite{Lion,FrieTurk3,LuoSmol1,LuoSmol2,Wu1,Wu2}, and references within, as well as work by Federbush, Luo, and Smoller \cite{FedLuoSmol}, in the case of magnetic stars.

Jang and Makino \cite{JangMak1,JangMak2}, and Strauss and Wu \cite{StraussWu1,StraussWu2}, who built on techniques first employed by Lichtenstein \cite{Lich} and Heilig \cite{Hei}, have also constructed axisymmetric rotating solutions, by making use of the Implicit Function Theorem. This method allows for solutions with $\gamma>4/3$, and in the works by Strauss and Yu, they in fact construct a continuous family of solutions parameterised by the prescribed angular velocity. In addition, Jang, Strauss, and Wu \cite{JangStraussWu} used these techniques to construct a rotating magnetic star.

\subsection{The Euler-Poisson System for $N$ Stars}

The model we work with describes $N$ Newtonian stars moving under both the influence of their own gravity, and that of the other stars, in the \emph{free boundary settting}. Let $N\geq2$ be a natural number. For each $\kappa=1,2,\dots,N$, we have
\begin{align}
\dell_{t}\rho_{\kappa}+\nabla\cdot \left(\rho_{\kappa}u_{\kappa}\right)=0\ \ \ \ &\text{in}\ \Omega_{\kappa}(t),\label{pre-euler-1}\\
\rho_{\kappa}\left(\dell_{t}+u_{\kappa}\cdot\nabla\right)u_{\kappa}+\nabla p_{\kappa}+\rho_{\kappa}\nabla\phi=0\ \ \ \ &\text{in}\ \Omega_{\kappa}(t),\label{pre-euler-2}\\
p_{\kappa}=0\ \ \ \ &\text{on}\ \dell\Omega_{\kappa}(t),\label{pre-euler-3}\\
\Delta\phi=4\pi\rho,\ \ \lim\limits_{|x|\rightarrow\infty}\phi(t,x)=0\ \ \ \ &\text{on}\ \mathbb{R}^{3},\label{pre-euler-4}\\
\mathcal{V}\left(\dell\Omega_{\kappa}(t)\right)=u_{\kappa}\cdot n_{\kappa}\ \ \ \ &\text{on}\ \dell\Omega_{\kappa}(t),\label{pre-euler-5}\\
\left(\rho_{\kappa},u_{\kappa}\right)=\left(\mathring{\rho}_{\kappa},\mathring{u}_{\kappa}\right)\ \ \ \ &\text{in}\ \Omega_{\kappa}(0).\label{pre-euler-6}
\end{align}

\begin{remark}\label{kappa-indexing-only-remark}
	Here $\kappa$ is introduced as an index variable ranging from $1$ to $N$. Throughout this work, indexing the star bodies will be the one and only use of $\kappa$, and conversely, indexing of the separate stars will only be done with $\kappa$. In particular there will be no summation convention over $\kappa$.
\end{remark}

Analogously to $\eqref{general-euler-1}$--$\eqref{general-euler-3}$, the $\rho_{\kappa}$, $u_{\kappa}$, and $p_{\kappa}$ are the densities, velocities, and pressures for each separate star. As mentioned in $\eqref{equation-of-state}$, we have a polytropic equation of state for each $\kappa$ given by
\begin{align}
p_{\kappa}(\rho_{\kappa})=\rho_{\kappa}^{\gamma},\label{equation-of-state-kappa}
\end{align}
where $\gamma>1$ is constant, and the same for each $\kappa$.

\begin{remark}\label{all-gamma-equal} It is not crucial for each $\gamma$ to be the same for each equation of state, i.e. $\gamma$ could depend on $\kappa$. However, it simplifies the analysis at certain points.
\end{remark}

Comparing $\eqref{general-euler-1}$--$\eqref{equation-of-state}$ to $\eqref{pre-euler-1}$--$\eqref{equation-of-state-kappa}$, we see that in the latter system, for each fixed $t$, the continuity and momentum equations for each $\kappa$ are only required to hold on the support of $\rho_{\kappa}(t,\cdot)$, namely $\Omega_{\kappa}(t)$, instead of on all of $\mathbb{R}^{3}$ as in the former system. In addition, the support, and therefore its boundary is now a dynamic object whose evolution we track in our framework, making this model a \emph{vacuum free boundary problem}.

The $\rho_{\kappa}$ are non-negative functions $\rho_{\kappa}:[0,\infty)\times\mathbb{R}^{3}\rightarrow\mathbb{R}_{\geq0}$, with supports $\supp{\rho_{\kappa}(t,\cdot)}=\Omega_{\kappa}(t)$. The boundary of $\Omega_{\kappa}(t)$ is given by $\dell\Omega_{\kappa}(t)$. We also have the initial domains and their boundaries
\begin{align}
\Omega_{\kappa}&\coloneqq\Omega_{\kappa}(0),\label{initial-domains}\\
\dell\Omega_{\kappa}&\coloneqq\dell\Omega_{\kappa}(0).\label{initial-domains-boundary}
\end{align}
Since we are looking at stars that are separated from each other initially, physically we must have that the $\Omega_{\kappa}$ are \emph{mutually disjoint} and \emph{compact}. The cumulative density function $\rho:[0,\infty)\times\mathbb{R}^{3}\rightarrow\mathbb{R}_{\geq0}$ is defined by
\begin{align}
\rho=\sum_{\kappa=1}^{N}\rho_{\kappa},\label{definition-cumulative-density}
\end{align}
and $\phi$ is the cumulative gravitational potential. For each fixed $t$, the domain of the velocity field at time $t$ is given by the support of the density at time $t$:
\begin{align*}
u_{\kappa}(t,\cdot):\Omega_{\kappa}(t)&\rightarrow\mathbb{R}^{3},\\
x&\mapsto u_{\kappa}(t,x).
\end{align*}
Finally, the quantity $\mathcal{V}(\dell\Omega_{\kappa}(t))$ is the outward normal velocity of $\dell\Omega_{\kappa}(t)$ and $n_{\kappa}$ is the outward unit normal on $\dell\Omega_{\kappa}(t)$.

To construct a robust well-posedness theory for the moving vacuum boundary Euler-Poisson system, it is crucial to include the {\em physical vacuum boundary condition}.
For each $\kappa$, the {\em speed of sound} $c_{\kappa}$ is given by
\begin{align}
c_{\kappa}^{2}=\frac{d p_{\kappa}}{d\rho_{\kappa}}=\gamma\rho_{\kappa}^{\gamma-1},\ \ \ \mathring{c}_{\kappa}^{2}(x)=c_{\kappa}^{2}(0,x).\label{speed-of-sound}
\end{align}
Then the physical vacuum boundary condition reads
\begin{align}
-\infty <\frac{\dell \mathring{c}_{\kappa}^{2}}{\dell \mathring{n}_{\kappa}}\bigg\rvert_{\dell\Omega_{\kappa}} < 0,\label{vacuum-boundary-condition}
\end{align}
where $\mathring{n}_{\kappa}$ is the outward unit normal on $\dell\Omega_{\kappa}$. Condition $\eqref{vacuum-boundary-condition}$ implies that for some constant $C$
\begin{align}
\frac1C \text{dist}(x,\partial\Omega_{\kappa}) \le \mathring{c}_{\kappa}^{2}(x) \le C  \text{dist}(x,\partial\Omega_{\kappa}),\label{speed-of-sound-distance function}
\end{align}
which in turn implies that $\mathring{c}_{\kappa}$ is only H\"{o}lder continuous of exponent $1/2$ at the boundary $\Omega_{\kappa}$. We note that as well as being a key component for the local theory of our system, the physical vacuum condition also naturally occurs in Lane-Emden stars \cite{Jang2014}. Liu \cite{Liu} gave an argument suggesting that the physical vacuum condition was natural in the context of the Euler system with damping. In general, the physical vacuum condition, and the role it plays in understanding the interaction of fluids with vacuum regions has been a subject of great interest, see \cite{LiuSmol,LiuYang1,LiuYang2,CouLindShk,CoSh2011,CoSh2012,JaMa2009,JaMa2011,JaMa2012,JaMa2015,LuoXinZeng,Serre2,HaJa3,HaJa4}.

When we view the problem on $\mathbb{R}^{3}$, any sufficiently regular nontrivial spherically symmetric solution with compact initial data has a finite time of existence \cite{MakUkai,MakPer}, with the same holding for any nontrivial solution for the Euler system without gravitation \cite{MakUkaiKawa}. Results on singularity formulation for the Euler system have also been established by Sideris \cite{Sid1}. The solutions in \cite{MakUkai,MakPer} cannot satisfy the physical vacuum boundary condition $\eqref{vacuum-boundary-condition}$, as the initial speed of sound $\mathring{c}$ is continuously differentiable on the boundary.

In the case of the physical vacuum free boundary, well-posedness theories for the compressible Euler system were developed by Coutand and Skholler \cite{CoSh2012}, and Jang and Masmoudi \cite{JaMa2015} independently. In our work, we adapt the techniques of Jang and Masmoudi. This theory is readily adapted to the Euler-Poisson system for the case $N=1$ as the potential term $\rho\nabla\phi$ is lower order, with respect to derivative count, to the top order pressure term $\nabla p$. We shall see in Appendix $\ref{local-well-posedness}$ that the local-in-time theory for $N\geq2$ does not require much adjustment in the theory either.

It is natural to look for expanding solutions to the Euler-Poisson system. Indeed, it has been shown that global-in-time solutions to the Euler system must have supports whose diameters grow at least linearly in time, with the same being true for solutions to the Euler-Poisson system with positive energy when $\gamma\geq4/3$, see \cite{HaJa1,Sid2}. Note that equilbrium solutions like the ones mentioned in the introduction are not considered in the class of time dependent solutions.

In the vacuum free boundary setting for the Euler system without gravitation, a finite parameter family of expanding global-in-time solutions has been found by Sideris \cite{Sid2,Sid3}, relying on an \emph{affine} ansatz on the Lagrangian flow map, allowing him to reduce the problem to one of solving a system of ODEs. The corresponding solutions are a finite dimensional family of compactly supported, expanding stars. Using an affine ansatz to construct solutions for compressible fluid flow is a technique that goes back to Ovsiannikov \cite{Ov1956} and Dyson \cite{Dyson1968}. Had\v{z}i\'{c} and Jang \cite{HaJa3} showed nonlinear stability of the solutions constructed in \cite{Sid3} for the range $\gamma\in(1,5/3]$, after which Sideris and Shkoller \cite{ShSi2017} established the corresponding result for $\gamma>5/3$. In the nonisentropic case, nonlinear stability of the Sideris solutions was established by Rickard, Had\v{z}i\'{c}, and Jang \cite{RicHaJa}.

In addition, Had\v{z}i\'{c} and Jang \cite{HaJa4} also showed that small perturbations of the Sideris affine solutions give rise to solutions to the vacuum free boundary Euler-Poission system, utilising a scaling structure in the system that allowed them to construct these solutions under the assumption of small densities. In general their methods require $\gamma\in(1,5/3)$. In this range, the gravitational potential terms are subcritical with respect to the pressure term, meaning that the pressure term will dominate the dynamics of the star. This is part of what allowed them to peturb solutions of the Euler system to find their solutions of the Euler-Poisson system. More specifically, their range of $\gamma$ was restricted to $\{1+\frac{1}{n}|n\in\mathbb{Z}_{\geq2}\}\cup(1,14/13)$. Note that these restrcitions, apart from being a subset of $(1,5/3)$, are largely technical and related to estimating the gravitational potential terms.

The author, along with Had\v{z}i\'{c} and Jang \cite{PaHaJa}, then showed that one can construct a family of global-in-time expanding solutions to the vacuum free boundary Euler and Euler-Poisson systems without appealing to an affine type ODE reduction. The main tool was utilising a scaling structure in the nonlinearity of the momentum equation that produced a stabilising effect. This, as well as the scaling exhibited in \cite{HaJa4}, was used to construct solutions with small densities. Analogously to \cite{HaJa4}, the range of $\gamma$ for the Euler-Poisson system is $\{1+\frac{1}{n}|n\in\mathbb{Z}_{\geq2}\}\cup(1,14/13)$. Despite the smallness condition, the initial densities in \cite{PaHaJa} have a wide class of possible profiles, whereas those associated with the solutions exhibited in \cite{Sid3} have a very specific form.

In the absence of a free boundary, global-in-time expanding solutions have been found by Serre \cite{Serre1997}, Grassin \cite{Grassin98}, and Rozanova \cite{Roz}. Serre found solutions by perturbing around linear velocity profiles, an idea which Grassin generalised. Rozanova showed the existence of related solutions without the small density assumption made by Grassin.

With the polytropic equations of state in $\eqref{equation-of-state-kappa}$, our system becomes, for $\kappa=1,2,\dots,N$,
\begin{align}
\dell_{t}\rho_{\kappa}+\nabla\cdot \left(\rho_{\kappa}u_{\kappa}\right)=0\ \ \ \ &\text{in}\ \Omega_{\kappa}(t),\label{euler-1}\\
\rho_{\kappa}\left(\dell_{t}+u_{\kappa}\cdot\nabla\right)u_{\kappa}+\nabla\left(\rho_{\kappa}^{\gamma}\right)+\rho_{\kappa}\nabla\phi=0\ \ \ \ &\text{in}\ \Omega_{\kappa}(t),\label{euler-2}\\
\rho_{\kappa}=0\ \ \ \ &\text{on}\ \dell\Omega_{\kappa}(t),\label{euler-3}\\
\Delta\phi=4\pi\rho,\ \ \lim\limits_{|x|\rightarrow\infty}\phi(t,x)=0\ \ \ \ &\text{on}\ \mathbb{R}^{3},\label{euler-4}\\
\mathcal{V}\left(\dell\Omega_{\kappa}(t)\right)=u_{\kappa}\cdot n_{\kappa}\ \ \ \ &\text{on}\ \dell\Omega_{\kappa}(t),\label{euler-5}\\
\left(\rho_{\kappa},u_{\kappa}\right)=\left(\mathring{\rho}_{\kappa},\mathring{u}_{\kappa}\right)\ \ \ \ &\text{in}\ \Omega_{\kappa}.\label{euler-6}
\end{align}
We shall refer to the system $\eqref{euler-1}$--$\eqref{euler-6}$ with the physical vacuum condition $\eqref{vacuum-boundary-condition}$ as $\textbf{EP}(N,\gamma)$. If we set
\begin{align}
w_{\kappa}(x): =\mathring{\rho}_{\kappa}^{\gamma-1}(x).\label{w-definition}
\end{align}
then
\begin{align}
\frac1C \text{dist}(x,\partial\Omega_{\kappa}) \le w_{\kappa}(x) \le C  \text{dist}(x,\partial\Omega_{\kappa})
\end{align}
in the vicinity of the initial vacuum boundary $\partial\Omega_{\kappa}$. The quantity $\sum_{\kappa=1}^{N}w_{\kappa}$ is proportional to the enthalpy of the system and the $\{w_{\kappa}|\kappa=1,2,\dots,N\}$ will
play a very important role in our analysis. For future simplicity, define
\begin{align}
\alpha=\frac{1}{\gamma-1},\label{alpha-definition}
\end{align}
so that $\mathring{\rho}_{\kappa}= w_{\kappa}^\alpha$.

By contrast to the classical $N$ Body problem, the stars described by our system are subjected to tidal forces which deform the geometry of their supports, and hence they can not be idealised as point particles. The classical problem consists of looking for solutions to the system of $N$ point particles interacting with each other via Newtonian gravity:

\begin{align}
\frac{d^{2}x_{i}}{dt^{2}}=-\sum_{\substack{1\leq i,j\leq N\\i\neq j}}m_{i}\frac{x_{i}-x_{j}}{\left|x_{i}-x_{j}\right|^{3}}\ \ \ x_{i}(0)=\bar{x}_{i}\ \ \ \dot{x}_{i}(0)=v_{i},\label{gravitational-n-body-problem}
\end{align}
where the $x_{i}(t)$ are the particle positions at time $t$, the $m_{i}$ are the particle masses, $\bar{x}_{i}$ are their initial positions, and $v_{i}$ are their initial velocities. Much like the subject of stellar dynamics, the $N$ body problem has a long and rich history; for more on the history and context, see \cite{SieMos,Heggie} and references within.

In our context, as mentioned previously, McCann \cite{McCann} exhibited solutions to the stationary compressible Euler-Poisson system that corresponded to a rotating binary star sysem. Miao and Shahshahani \cite{MiaoShah} showed that in the case of the $2$ body problem for the incompressible free boundary Euler-Poisson system, one can start with initial configurations that would correspond to hyperbolic orbit in the point particle case, and obtain bounded orbits. The reason is that the fluid bodies will naturally deform due to their gravitational effects on each other, and this results in a loss of energy.

In the classical $N$ Body problem $\eqref{gravitational-n-body-problem}$, one can ask what happens to the dynamics of each particle if we start them far away from each other (in some suitable sense), and also point their initial velocities away from each other. We might expect that if the initial distances are large enough, with initial velocities pointing in suitable directions, their dynamics will essentially decouple, and each will travel on a path of constant velocity up to small error. We will show in our work that this type of initial configuration can lead to global-in-time solutions even in the case of the $N$ Body of Euler-Poisson system.

Explicitly, we exhibit an open set (in a suitable topology) of initial positions and velocities that lead to global-in-time solutions of $\textbf{EP}(N,\gamma)$, with each star asymptotically behaving like an expanding star moving with constant velocity. This is further discussed in Sections $\ref{flow-map-section}$, $\ref{main-result-section}$. Here we state a rough version of our theorem.
\begin{theorem}[Main result: informal statement]\label{main-result-rough}
	Let $\gamma=1+\frac{1}{n}$ for $n\in\mathbb{Z}_{\geq2}$ or $\gamma\in(1,14/13)$. Then there exist open sets of initial positions and velocities for each star in the free boundary $N$-body compressible Euler-Poisson system $\textbf{EP}(N,\gamma)$ which lead to global-in-time solutions.
\end{theorem}

\begin{remark}\label{gamma-restrictions}The restriction of possible values of $\gamma$ in our work is for the same reasons as in $\cite{HaJa4,PaHaJa}$, discussed previously in this introduction.
\end{remark}
The plan of the paper is as follows. In Section $\ref{lagrangian-formulation-section}$ we formulate our problem in Lagrangian variables, a technique commonly used to study the physical vacuum free boundary compressible Euler(-Poisson) system. In Section $\ref{rescaling-section}$, we discuss the properties of scaling in the Euler-Poisson system, and how we use these properties to find global-in-time solutions. In Section $\ref{notation-section}$, we fix general notation, and in Section $\ref{energy-function-section}$ we define our energy spaces and higher order energy  and curl functions. In Section $\ref{main-result-a-priori-assumptions-section}$ we state our main result precisely, and list our a priori assumptions. Sections $\ref{gravitational-potential-estimates-section}$, $\ref{curl-estimates-section}$, and $\ref{energy-estimates-section}$ are devoted to estimates of the graviational potential, curl estimates, and energy estimates respectively. Finally, in Section $\ref{proof-of-main-result-section}$, we prove our main result, Theorem $\ref{main-theorem}$.

\section{Lagrangian Formulation}\label{lagrangian-formulation-section}

For simplicity, we will assume that the initial domains $\Omega_{\kappa}$ are closed balls of radius $1$ with centres $\bar{x}_{\kappa}$. We introduce a reference domain:
\begin{align}
\Omega\coloneqq B_{1},\label{material-manifold}
\end{align}
where $B_{1}$ is the closed unit ball centred at $0$ in $\mathbb{R}^{3}$. For $\kappa=1,2,\dots,N$, it is clear that $\Omega_{\kappa}=\Omega+\bar{x}_{\kappa}$.\\

\begin{remark}\label{more-general-initial-domains}The techniques we use to construct global-in-time solutions in this paper work equally well when each $\Omega_{\kappa}$ is a diffeomorphism of a translate of the closed unit ball, close enough to the identity in some norm, with each translation such that the distance bwteen every pair of domains is sufficiently large.
\end{remark}
Following \cite{CoSh2012,JaMa2015}, we utilise Lagrangian coordinates to transform the problem of studying to the free boundary problem $\textbf{EP}(N,\gamma)$ in to one of studying a fixed boundary problem. Define, for each $\kappa=1,2,\dots,N$, the Lagrangian flow map $\eta_{\kappa}$ by
\begin{align}
\eta_{\kappa}:I\times\Omega&\rightarrow\Omega_{\kappa}(t)\nonumber\\
(t,x)&\mapsto\eta_{\kappa}(t,x),\label{flow-maps-domain-range}
\end{align}
which solves the ODE
\begin{align}
\dell_{t}\eta_{\kappa}(t,x)&=u_{\kappa}(t,\eta_{\kappa}(t,x)),\ \ \ t\in I,\nonumber\\
\mathring{\eta}_{\kappa}(x)&=\eta_{\kappa}(0,x)=x+\bar{x}_{k}.\label{flow-maps-ode}
\end{align}
Here $I$ is some time interval. Let
\begin{align}
v_{\kappa}(t,x)=u_{\kappa}(t,\eta_{\kappa}(t,x)),\label{lagrangian-velocity}\\
f_{\kappa}(t,x)=\rho_{\kappa}(t,\eta_{\kappa}(t,x)),\label{lagrangina-density}\\
\cov[\kappa]=\nabla\eta_{\kappa},\label{change-of-variables-matrix}\\
\inv[\kappa]=\left(\cov[\kappa]\right)^{-1},\label{change-of-variables-inverse}\\
\jac_{\kappa}=\det{\cov[\kappa]},\label{change-of-variables-jacobian}\\
a[\kappa]=\jac_{\kappa}\inv[\kappa],\label{change-of-variables-cofactor}\\
\psi_{\kappa}(t,x)=\phi(t,\eta_{\kappa}(x,t)).\label{lagrangian-potential}
\end{align}
We have the following differentiation formulae for $\inv$ and $\jac$:
\begin{align}
\dell\inv[\kappa]^{k}_{i}=-\inv[\kappa]^{k}_{j}\dell\dell_{s}\eta_{\kappa}^{j}\inv[\kappa]^{s}_{i},\ \ \
\dell\jac_{\kappa}=\jac_{\kappa}\inv[\kappa]^{s}_{j}\dell\dell_{s}\eta_{\kappa}^{j},\ \ \ \dell\in\{\dell_{t},\dell_{1},\dell_{2},\dell_{3}\} \label{inverse-jacobian-differentiation-formula}
\end{align}
These identities imply the Piola identity
\begin{align}
	\dell_{j}a[\kappa]^{j}_{i}=0.\label{piola-identity}
\end{align}
For each $\kappa=1,2,\dots,N$, pulling the continuity and momentum equations $\eqref{euler-1}$ and $\eqref{euler-2}$ back via the flow map $\eta_{\kappa}$ gives us, for $i=1,2,3$,
\begin{align}
\dell_{t}f_{\kappa}+f_{\kappa}\inv[\kappa]^{j}_{i}\dell_{j}v_{\kappa}^{i}=0\ \ \ \ \ &\text{in}\ I\times\Omega,\label{pre-lagrangian-euler-1}\\
\dell_{t}v_{\kappa}^{i}+\inv[\kappa]^{j}_{i}\dell_{j}\left(f_{\kappa}^{\gamma}\right)+f_{\kappa}\inv[\kappa]^{j}_{i}\dell_{j}\psi_{\kappa}=0\ \ \ \ \ &\text{in}\ I\times\Omega.\label{pre-lagrangian-euler-2}
\end{align}
Then, applying the formula for $\jac$ in $\eqref{inverse-jacobian-differentiation-formula}$, from $\eqref{pre-lagrangian-euler-1}$ we obtain
\begin{align}
f_{\kappa}(t,x)=\rho_{\kappa}(t,\eta_{\kappa}(t,x))=\mathring{\rho_{\kappa}}(\mathring{\eta_{\kappa}}(x))\jac_{\kappa}^{-1}.\label{lagrangian-density-relation}
\end{align}
Finally, applying $\eqref{piola-identity}$, $\eqref{lagrangian-density-relation}$ to $\eqref{pre-lagrangian-euler-2}$, and recalling from $\eqref{flow-maps-ode}$ that $\dell_{t}\eta_{\kappa}=v_{\kappa}$ our system becomes, for $\kappa=1,2,\dots,N$, and $i=1,2,3$:
\begin{align}
\tilde{w}_{\kappa}^{\alpha}\dell_{tt}\eta_{\kappa}^{i}+\dell_{k}\left(\tilde{w}_{\kappa}^{1+\alpha}\inv[\kappa]^{k}_{i}\jac_{\kappa}^{-1/\alpha}\right)+\tilde{w}_{\kappa}^{\alpha}\inv[\kappa]^{k}_{i}\dell_{k}\psi_{\kappa}=0\ \ \ \ \ \ \ \ \ \ \ &\text{in}\ I\times\Omega,\label{new-lagrangian-euler-1}\\
(v_{\kappa}(t,x),\eta_{\kappa}(t,x))=(\mathring{u}_{\kappa}(\mathring{\eta}_{\kappa}(x)),\mathring{\eta}_{\kappa})\ \ \ \ &\text{in}\ \{t=0\}\times\Omega,\label{new-lagrangian-euler-2}\\
\tilde{w}_{\kappa}=0\ \ \ \ \ \ \ \ \ \ \ &\text{on}\ \dell\Omega,\label{new-lagrangian-euler-3}
\end{align}
where, recalling $w_{\kappa}$ defined in $(\ref{w-definition})$, we define
\begin{align}
\tilde{w}_{\kappa}(x)=\mathring{\rho_{\kappa}}(\mathring{\eta_{\kappa}}(x))^{\gamma-1}=w_{\kappa}(x+\bar{x}_{\kappa}).\label{definition-of-translated-w}
\end{align}
For a more detailed derivation of the system $\eqref{new-lagrangian-euler-1}$--$\eqref{new-lagrangian-euler-3}$, see, for example \cite{JaMa2015}.

\section{Rescaling}\label{rescaling-section}

\subsection{Flow Map}\label{flow-map-section}
In this section we introduce our ansatz for $\eta_{\kappa}$. However, we first give a motivation for the form of our ansatz. In \cite{PaHaJa} the author alongside Had\v{z}i\'{c} and Jang found that in the case of $N=1$, solutions of the form
\begin{align}
\eta(t,x)=(t+1)(x+\theta(\log{(1+t)},x))\label{theta-one-body-definition}
\end{align}
for the Euler and Euler-Poisson systems were global-in-time solutions for small enough $\theta$ (with size measured in a suitable function space), as well as small enough initial density. The suggested form of $\eta$ in $\eqref{theta-one-body-definition}$ is exactly what allows us to take advantage of the nonlinear scaling structure in the Euler, and Euler-Poisson systems. However, we can also view it as encoding the condition that our solution must expand; $\eqref{theta-one-body-definition}$ implies that initially, $u(t,\eta(t,x))=x+\dell_{t}\theta(\log{(1+t)},x)$, and thus if we have sufficient smallness on $\theta$ and its derivatives, the velocity of each particle stays close to the radial direction, which forces expansion.

As discussed in the introduction, in the case of $N\geq2$, if the stars initially are far away from each other, and are then pushed further away, the gravitational interaction between two stars should be small. Thus the motion of each star should be close to that of constant velocity, as in \cite{PaHaJa}, with the velocity essentially decomposing in to two main parts; one driving repulsion, and one driving expansion. For each $\kappa=1,2,\dots,N$, we define the \emph{error} $\theta_{\kappa}$ and \emph{repulsive velocity} $\mu_{\kappa}$ by the following relation
\begin{align}
\eta_{\kappa}(t,x)=(x+\bar{x}_{\kappa})+t(x+\mu_{\kappa}(x))+(t+1)\theta_{\kappa}(\log{(1+t)},x))\label{eta-pre-ansatz}
\end{align}
This form of $\eta_{\kappa}$ encodes the fact that the particles should have a radial velocity to force expansion, and a repulsive velocity, to force the stars away from each other.

As in \cite{HaJa3,PaHaJa}, we define a new logarithmic timescale
\begin{align}
\tau=\log{(1+t)},\label{tau-definition}
\end{align}
and under this transformation, we can write $\eta_{\kappa}$ as
\begin{align}
	\eta_{\kappa}(t,x)=e^{\tau}\left(x+e^{-\tau}\bar{x}_{\kappa}+(1-e^{-\tau})\mu_{\kappa}(x)+\theta_{\kappa}(\tau,x)\right)=e^{\tau}\zeta_{\kappa}(\tau,x),\label{eta-ansatz}
\end{align}
where $\zeta_{\kappa}$ is defined for notational convenience. Under this ansatz, we will prove that solutions to $\eqref{new-lagrangian-euler-1}$--$\eqref{new-lagrangian-euler-3}$ with small enough $\theta_{\kappa}$, as well as some more technical assumptions on $\mu_{\kappa}$, are global-in-time.

\begin{remark}\label{tau-justification}
	Note that the logarithmic timescale in $\eqref{tau-definition}$ is used to study certain types of self similar solutions to evolution equations (see $\cite{EgFo}$), but in our case is for convenience.
\end{remark}

We have some definitions to record. Throughout the paper, the identity matrix $\mathbb{R}^{3}\rightarrow\mathbb{R}^{3}$ will be denoted by $\mathbb{I}$. Let
\begin{align}
\ncov[\kappa]&=\nabla\zeta_{\kappa},\label{rescaled-change-of-variables-matrix}\\
\ninv[\kappa]&=\left(\ncov[\kappa]\right)^{-1},\label{rescaled-change-of-variables-inverse}\\
\njac_{\kappa}&=\det{\ncov[\kappa]},\label{rescaled-change-of-variables-jacobian}\\
\frac{3}{\alpha}&=\beta.\label{beta-definition}
\end{align}
Similarly to $(\ref{inverse-jacobian-differentiation-formula})$, we have for $\ninv[\kappa]$ and $\njac_{\kappa}$:
\begin{align}
\dell\ninv[\kappa]^{k}_{i}=-\ninv[\kappa]^{k}_{j}\dell\dell_{s}\zeta_{\kappa}^{j}\ninv[\kappa]^{s}_{i},\label{rescaled-inverse-differentiation-formula}\\
\dell\njac_{\kappa}=\njac_{\kappa}\ninv[\kappa]^{s}_{j}\dell\dell_{s}\zeta_{\kappa}^{j},\label{rescaled-jacobian-differentiation-formula}
\end{align}
for $\dell\in\{\dell_{\tau},\dell_{1},\dell_{2},\dell_{3}\}$.\\

\noindent Using $\eqref{tau-definition}$--$\eqref{eta-ansatz}$ to rewrite $\eqref{new-lagrangian-euler-1}$, we obtain
\begin{align}
e^{-\tau}\tilde{w}_{\kappa}^{\alpha}\left(\dell_{\tau\tau}\theta_{\kappa}^{i}+\dell_{\tau}\theta_{\kappa}^{i}\right)+e^{-\left(1+\frac{3}{\alpha}\right)\tau}\dell_{k}\left(\tilde{w}_{\kappa}^{1+\alpha}\ninv[\kappa]^{k}_{i}\njac_{\kappa}^{-1/\alpha}\right)+e^{-\tau}\tilde{w}_{\kappa}^{\alpha}\ninv[\kappa]^{k}_{i}\dell_{k}\psi_{\kappa}=0,\label{euler-nonlinear-scaling}
\end{align}
and upon multiplying everything on the left hand side by $e^{\left(1+\frac{3}{\alpha}\right)\tau}$, the system $\eqref{new-lagrangian-euler-1}$--$\eqref{new-lagrangian-euler-3}$ becomes:
\begin{align}
e^{\beta\tau}\tilde{w}_{\kappa}^{\alpha}\left(\dell_{\tau\tau}\theta_{\kappa}^{i}+\dell_{\tau}\theta_{\kappa}^{i}\right)+\dell_{k}\left(\tilde{w}_{\kappa}^{1+\alpha}\ninv[\kappa]^{k}_{i}\njac_{\kappa}^{-1/\alpha}\right)+e^{\beta\tau}\tilde{w}_{\kappa}^{\alpha}\ninv[\kappa]^{k}_{i}\dell_{k}\psi_{\kappa}=0\ \ \ \ \ \ \ \ \ \ \ &\text{in}\  I\times\Omega,\label{rescaled-lagrangian-euler-1}\\
(\dell_{\tau}\theta_{\kappa}(\tau,x),\theta_{\kappa}(\tau,x))=(\mathring{u}_{\kappa}(\mathring{\eta}_{\kappa}(x))-(\mathring{\eta}_{\kappa}(x)+\mu_{\kappa}(x)),0)\ \ \ \ &\text{in}\ \{\tau=0\}\times\Omega,\label{rescaled-lagrangian-euler-2}\\
\tilde{w}_{\kappa}=0\ \ \ \ \ \ \ \ \ \ \ &\text{on}\ \dell\Omega,\label{rescaled-lagrangian-euler-3}
\end{align}

\begin{remark}\label{scaling-leads-to-stabilising-mechanism} The imbalance in exponential powers between the velocity term and the pressure term in $\eqref{euler-nonlinear-scaling}$ leads to a positive exponential in front of the velocity term in $\eqref{rescaled-lagrangian-euler-1}$. This is exactly the nonlinear scaling mechanism, exhibited in $\cite{PaHaJa}$, that will stabilise our solution and help us to prove that it is global-in-time. 
\end{remark}
For the gravitation potential term, note that in Eulerian coordinates, we have the Poisson equation $\eqref{euler-4}$ given by $\Delta\phi=4\pi\rho$, where $\rho$ is the cumulative density $\sum_{\kappa}\rho_{\kappa}$. Due to the compact support of each $\rho_{\kappa}$, we can use the convolution formula for Poisson's equation and write $\phi$ explicitly as
\begin{align}
\phi(t,x)=-\sum_{\kappa=1}^{N}\int_{\Omega(t)}\frac{\rho_{\kappa}(t,z)}{\left|x-z\right|} dx.\label{phi-eulerian-soluton}
\end{align} 
If we then apply the flow map $\eta_{\kappa}$, for a fixed $\kappa$, to both sides of $\eqref{phi-eulerian-soluton}$, we have, for $i=1,2,3$,
\begin{align}
\ninv[\kappa]^{k}_{i}\dell_{k}\psi_{\kappa}&=-\ninv[\kappa]^{k}_{i}\dell_{k}\left(\int_{\Omega}\frac{\tilde{w}_{\kappa}^{\alpha}(z)}{\left|\eta_{\kappa}(t,x)-\eta_{\kappa}(t,z)\right|}dz\right)-\sum_{\kappa'\neq\kappa}\ninv[\kappa]^{k}_{i}\dell_{k}\left(\int_{\Omega}\frac{\tilde{w}^{\alpha}_{\kappa'}(z)}{\left|\eta_{\kappa}(t,x)-\eta_{\kappa'}(t,z)\right|}dz\right)\nonumber\\
&=-e^{-\tau}\ninv[\kappa]^{k}_{i}\dell_{k}\left(\int_{\Omega}\frac{\tilde{w}_{\kappa}^{\alpha}(z)}{\left|\zeta_{\kappa}(\tau,x)-\zeta_{\kappa}(\tau,z)\right|}dz\right)-e^{-\tau}\sum_{\kappa'\neq\kappa}\ninv[\kappa]^{k}_{i}\dell_{k}\left(\int_{\Omega}\frac{\tilde{w}^{\alpha}_{\kappa'}(z)}{\left|\zeta_{\kappa}(\tau,x)-\zeta_{\kappa'}(\tau,z)\right|}dz\right)\nonumber\\
&=-\mathscr{G}_{\kappa}^{i}-\sum_{\kappa'\neq\kappa}\mathscr{I}_{[\kappa,\kappa']}^{i}.\label{lagrangian-gradient-of-lagrangian-potential}
\end{align}
\noindent We can further rewrite $\mathscr{G}_{\kappa}$ and $\mathscr{I}_{[\kappa,\kappa']}$. For $\mathscr{G}_{\kappa}$, note that
\begin{align}
\mathscr{G}_{\kappa}^{i}&=e^{-\tau}\ninv[\kappa](\tau,x)^{k}_{i}\dell_{x_{k}}\left(\int_{\Omega}\frac{\tilde{w}_{\kappa}^{\alpha}(z)}{\left|\zeta_{\kappa}(\tau,x)-\zeta_{\kappa}(\tau,z)\right|}dz\right)=-e^{-\tau}\int_{\Omega}\ninv[\kappa](\tau,z)^{k}_{i}\dell_{z_{k}}\left(\frac{1}{\left|\zeta_{\kappa}(\tau,x)-\zeta_{\kappa}(\tau,z)\right|}\right)\tilde{w}_{\kappa}^{\alpha}(z)\ dz\nonumber\\
&=e^{-\tau}\int_{\Omega}\frac{\dell_{k}\left(\ninv^{k}_{i}[\kappa]\tilde{w}_{\kappa}^{\alpha}\right)}{\left|\zeta_{\kappa}(\tau,x)-\zeta_{\kappa}(\tau,z)\right|}dz,\label{pre-self-interaction-rescaled-form}
\end{align}
where the last step uses integration by parts, noting that $\tilde{w}_{\kappa}$ is $0$ on $\dell\Omega$. For $\mathscr{I}_{[\kappa,\kappa']}$, we have
\begin{align}
\mathscr{I}_{[\kappa,\kappa']}&=e^{-\tau}\ninv[\kappa]^{k}_{i}(\tau,x)\dell_{x_{k}}\left(\int_{\Omega}\frac{\tilde{w}^{\alpha}_{\kappa'}(z)}{\left|\zeta_{\kappa}(\tau,x)-\zeta_{\kappa'}(\tau,z)\right|}dz\right)\nonumber\\
&=-e^{-\tau}(\tau,x)\int_{\Omega}\frac{\ninv[\kappa]^{k}_{i}\dell_{k}\zeta_{\kappa}^{j}(\tau,x)\left(\zeta_{\kappa}^{j}(\tau,x)-\zeta_{\kappa'}^{j}(\tau,z)\right)}{\left|\zeta_{\kappa}(\tau,x)-\zeta_{\kappa'}(\tau,z)\right|^{3}}\tilde{w}_{\kappa}(z)\ dz\nonumber\\
&=e^{-\tau}\int_{\Omega}\frac{\left(\zeta_{\kappa}^{i}(\tau,x)-\zeta_{\kappa'}^{i}(\tau,z)\right)\tilde{w}_{\kappa'}^{\alpha}}{\left|\zeta_{\kappa}(\tau,x)-\zeta_{\kappa'}(\tau,z)\right|^{3}}dz,\label{pre-tidal-terms-rescaled-form}
\end{align}
where the last line is because $\ninv[\kappa]=\left[\nabla\zeta_{\kappa}\right]^{-1}$.

The $\mathscr{G}_{\kappa}$ are the \emph{self-interaction terms} and represent the part of the potential that encodes how the star is affected by its own gravity. The $\mathscr{I}_{[\kappa,\kappa']}$ are the \emph{tidal terms} which encode how different stars affect each other via their gravitational interaction. Note that in the case of $N=1$ only the self-interaction terms are present, and these terms have already been studied in our context in \cite{HaJa4}. The tidal terms are unique to the $N\geq2$ case, and understanding how to control these terms indirectly gives us information on how to configure our system initially, see Remark $\ref{strong-separation-condition-remark}$.

\subsection{Initial Density Profiles}

In this section we will fix our initial density profiles for each $\kappa$. First we define a collection of admissible profiles.
\begin{definition}\label{definition-of-W}
	For each $\kappa\in\{1,2,\dots,N\}$, let $\mathcal{W}_{\kappa}$ be the set of functions $F:\Omega_{\kappa}\rightarrow\mathbb{R}$ with the following properties:
	\begin{itemize}
		\item Letting $\emph{int}(\Omega_{\kappa})\coloneqq\{x\in\Omega_{\kappa}|\left|x-\bar{x}_{\kappa}\right|<1\}$, we have
		\begin{align}
		F\rvert_{\emph{int}(\Omega_{\kappa})}>0,\ \ \ F\rvert_{\dell\Omega_{\kappa}}\equiv0.\label{W-0-on-boundary}
		\end{align} 
		\item
		There exists a positive constant $C>0$ such that  for any $x\in\Omega_{\kappa}$
		\begin{align}
		\frac1 Cd(x,\dell\Omega_{\kappa})\leq F(x)\leq C d(x,\dell\Omega_{\kappa}).\label{W-and-d-equivalence}
		\end{align}
		where $x\mapsto d(x,\partial\Omega_{\kappa})$ is the distance function to $\partial\Omega_{\kappa}$.
		\item The function given by
		\begin{align}
		x\mapsto\frac{F(x)}{d(x,\dell\Omega_{\kappa})}\label{ratio-of-W-d-smooth}
		\end{align}
		is smooth on $\Omega_{\kappa}$.
	\end{itemize}
\end{definition}
Now we fix an initial density profile for each $\kappa=1,2,\dots,N$ by choosing a function from each $\mathcal{W}_{\kappa}$, which we will call $W_{\kappa}$, and a constant $\delta>0$ such that
\begin{align}
\delta W_{\kappa}(x)=w_{\kappa}(x)=\mathring{\rho}_{\kappa}^{\gamma-1}(x).\label{definition-of-delta}
\end{align}
Note that $\delta$ is significant as we can adjust the size of our initial densities by adjusting the value of $\delta$. Throughout, $\delta\ll1$, and we will specify explicit bounds on $\delta$ wherever necessary. This ansatz has already been used in \cite{HaJa4,PaHaJa} to find global-in-time solutions with small density. Under this definition of $\mathring{\rho}_{\kappa}$, the speed of sound given by $\mathring{c}_{\kappa}^{2}=\gamma\delta W_{\kappa}$ will satisfy the physical vacuum condition $\eqref{vacuum-boundary-condition}$.

Finally, since we work from a reference domain $\Omega$, we define for each $\kappa$:
\begin{align}
\tilde{W}_{\kappa}(x)=W_{\kappa}(x+\bar{x}_{\kappa}),\ \ \ x\in\Omega.\label{definition-of-translated-W}
\end{align}
Due to $\eqref{W-and-d-equivalence}$, we have the relation
\begin{align}
\frac1 Cd(x,\dell\Omega)\leq \tilde{W}_{\kappa}(x)\leq C d(x,\dell\Omega),\label{translated-W-and-d-equivalence}
\end{align}
as $d(x+\bar{x}_{\kappa},\dell\Omega_{\kappa})=d(x,\dell\Omega)$ for all $x\in\Omega$.

Clearly $\delta\tilde{W}_{\kappa}=\tilde{w}_{\kappa}$, where $\tilde{w}_{\kappa}$ is defined in $(\ref{definition-of-translated-w})$. We subsitute for $\tilde{W}_{\kappa}$, and our system becomes, for $\kappa=1,2,\dots,N$, and $i=1,2,3$:
\begin{align}
\frac{1}{\delta}e^{\beta\tau}\tilde{W}_{\kappa}^{\alpha}\left(\dell_{\tau\tau}\theta_{\kappa}^{i}+\dell_{\tau}\theta_{\kappa}^{i}\right)+\dell_{k}\left(\tilde{W}_{\kappa}^{1+\alpha}\ninv[\kappa]^{k}_{i}\njac_{\kappa}^{-1/\alpha}\right)+\frac{1}{\delta}e^{\beta\tau}\tilde{W}_{\kappa}^{\alpha}\ninv[\kappa]^{k}_{i}\dell_{k}\psi_{\kappa}=0\ \ \ \ \ \ \ \ \ \ \ &\text{in}\  I\times\Omega,\label{delta-rescaled-lagrangian-euler-1}\\
(\dell_{\tau}\theta_{\kappa}(\tau,x),\theta_{\kappa}(\tau,x))=(\mathring{u}_{\kappa}(\mathring{\eta}_{\kappa}(x))-(\mathring{\eta}_{\kappa}(x)+\mu_{\kappa}(x)),0)\ \ \ \ &\text{in}\ \{\tau=0\}\times\Omega,\label{delta-rescaled-lagrangian-euler-2}\\
\tilde{W}_{\kappa}=0\ \ \ \ \ \ \ \ \ \ \ &\text{on}\ \dell\Omega.\label{delta-rescaled-lagrangian-euler-3}
\end{align}
Using $\eqref{definition-of-delta}$ to replace $\tilde{w}_{\kappa}$ with $\tilde{W}_{\kappa}$ in $\eqref{pre-self-interaction-rescaled-form}$ and $\eqref{pre-tidal-terms-rescaled-form}$, we obtain
\begin{align}
&\mathscr{G}_{\kappa}^{i}=\delta^{\alpha}e^{-\tau}\int_{\Omega}\frac{\dell_{k}\left(\ninv^{k}_{i}[\kappa]\tilde{W}^{\alpha}\right)}{\left|\zeta_{\kappa}(\tau,x)-\zeta_{\kappa}(\tau,z)\right|}dz\label{self-interaction-rescaled-form}\\
&\mathscr{I}_{[\kappa,\kappa']}^{i}=\delta^{\alpha}e^{-\tau}\int_{\Omega}\frac{\left(\zeta_{\kappa}^{i}(\tau,x)-\zeta_{\kappa'}^{i}(\tau,z)\right)\tilde{W}_{\kappa'}^{\alpha}}{\left|\zeta_{\kappa}(\tau,x)-\zeta_{\kappa'}(\tau,z)\right|^{3}}dz.\label{tidal-terms-rescaled-form}
\end{align}

\section{Notation}\label{notation-section}

\subsection{General Notation}
For a function $F:\mathcal O\rightarrow\mathbb{R}$, some domain $\mathcal O$, the support of $F$ is denoted $\supp{F}$. For a real number $\lambda$, the ceiling function, denoted $\ceil*{\lambda}$, is the smallest integer $M$ such that $\lambda\leq M$. For two real numbers $A$ and $B$, we say $A\lesssim B$ if there exists a positive constant $C$ such that
\begin{align}
A\leq CB,\label{lesssim-def}
\end{align}
and for two real valued functions $f$ and $g$, we say $f\lesssim g$ if $f(x)\lesssim g(x)$ holds pointwise. For two real-valued non-negative functions $f,g: \mathcal O\rightarrow \mathbb{R}_{\geq0}$, some domain $\mathcal O$, we say $f\sim g$ if there exist positive constants $c_{1}$ and $c_{2}$ such that
\begin{align}
c_{1}g(x)\leq f(x)\leq c_{1}g(x),\label{two-equivalent-functions-def}
\end{align}
for all $x\in \mathcal O$.


We also record the definition of the radial function on $\Omega=B_{1}$ the unit ball:
\begin{align}
r:B_{1}&\rightarrow\mathbb{R}_{\ge0}\nonumber\\
x&\mapsto r(x)\coloneqq |x|.\label{radial-function}
\end{align}
It is convenient to define shorthand for the distance function on $\Omega$. Define
\begin{align}
d_{\Omega}(x)=d(x,\dell\Omega).\label{distance-function-shorthand}
\end{align}

\subsection{Derivatives}\label{derivatives}

As we have seen above, rectangular derivatives will be denoted as $\dell_{i}$, for $i$ in $1,2,3$. In addition, we define various rectangular and $\zeta_{\kappa}$ Lie derivatives that will be used throughout. The gradient, divergence, and curl on vector fields are given by
\begin{align}
&[\nabla F]^{i}_{j}=\dell_{j}F^{i},\label{gradient}\\
&\dive{F}=\dell_{i}F^{i},\label{divergence}\\
&[\curl{F}]^{i}=\varepsilon_{ijk}\dell_{j}F^{k},\label{curl}
\end{align}
for $i,j=1,2,3.$

The $\zeta_{\kappa}$ versions are given by
\begin{align}
&[\ngrad{F}]^{i}_{j}=\ninv[\kappa]^{k}_{j}\dell_{k}F^{i},\label{zeta-gradient}\\
&\ndiv{F}=\ninv[\kappa]^{k}_{i}\dell_{k}F^{i},\label{zeta-divergence}\\
&[\ncurl{F}]^{i}_{j}=\varepsilon_{ijk}\ninv[\kappa]^{s}_{j}\dell_{s}F^{k}.\label{zeta-curl}
\end{align}
In addition, we also need the matrix $\zeta_{\kappa}$ curl, given by
\begin{align}
[\nCurl{F}]^{i}_{j}=\ninv[\kappa]^{s}_{j}\dell_{s}F^{i}-\ninv[\kappa]^{s}_{i}\dell_{s}F^{j}.\label{zeta-Curl}
\end{align}
Recall $(\ref{material-manifold})$; our reference domain is the closed unit ball in $\mathbb{R}^{3}$, $\Omega=B_{1}$. There exists a natural choice of spherical coordinates $(r,\omega,\phi)$. An advantage of this choice of domain is that we can privilege the outward normal derivative, the direction in which the degeneracy of the problem occurs, due to the vacuum boundary condition $\eqref{vacuum-boundary-condition}$.

Accordingly we will, in essence, use $\dell_{r}$ as the normal derivative, and $\dell_{\omega},\dell_{\phi}$ as the tangential derivatives. However we modify these derivatives by using linear combinations. These modifications allow for better commutation relations with the rectangular derivatives. Let the angular derivatives $\angdell_{ij}$ and radial derivative $\rdell$ be given by
\begin{align}
\angdell_{ij} &:=x_{i}\dell_{j}-x_{j}\dell_{i},\label{angular-derivative-def}\\
\rdell &: = x_i\partial_i =r\dell_{r},\label{radial-derivative-def}
\end{align}
where the $x_{i}$ and $\dell_{j}$ are rectangular, and $i,j$ run through $1,2,3$. On regions separated from the origin, we will use the following decomposition frequently:
\begin{align}
\dell_{i}=\frac{x_{j}}{r^{2}}\angdell_{ji}+\frac{x_{i}}{r^{2}}\rdell.\label{rectangular-as-ang-rad}
\end{align}
\begin{remark}\label{degeneracy-modified-spherical-derivatives}
	The coefficients of the derivatives we have defined go to 0 at the origin, which means we can only use them for estimates on a region separated from the origin. This can be dealt with using a partition of unity argument. Near the boundary we use these modified spherical derivatives, and on the interior, we are free to use rectangular derivatives as the degeneracy at the vacuum boundary is not an issue in this case.
\end{remark}
Now, for $m\in\mathbb{Z}_{\geq0}$, and $\underline{n}=(n_{1},n_{2},n_{3})\in\mathbb{Z}^{3}_{\geq0}$, we define
\begin{align}
\Ndell{m}{n}\coloneqq \rdell^{m}\angdell_{12}^{n_{1}}\angdell_{13}^{n_{2}}\angdell_{23}^{n_{3}}.\label{Ndell-def}
\end{align}
Although there are six non-zero $\angdell$ derivatives to consider, $\angdell_{ij}=-\angdell_{ji}$, so $(\ref{Ndell-def})$ covers all cases. For such an $\underline{n}\in\mathbb{Z}^{3}_{\geq0}$, $|\underline{n}|=n_{1}+n_{2}+n_{3}$. Similarly for rectangular derivatives we define, for $\underline{k}=(k_{1},k_{2},k_{3})\in\mathbb{Z}_{\geq0}^{3}$,
\begin{align}
\nabla^{\underline{k}}=\dell_{1}^{k_{1}}\dell_{2}^{k_{2}}\dell_{3}^{k_{3}}.\label{Dk-def}
\end{align}
We have the commutation relations between the modified spherical  and rectangular derivatives, for $i,j,k,m\in\{1,2,3\}$, given by
\begin{align}
\comm{\angdell_{ij}}{\rdell}&=0,\label{commutator-ang-rad}\\
\comm{\angdell_{ij}}{\angdell_{jk}}&=\angdell_{ik},\label{commutator-ang-ang}\\
\comm{\dell_{m}}{\rdell}&=\dell_{m},\label{commutator-rectangular-rad}\\
\comm{\dell_{m}}{\angdell_{ji}}&=\delta_{mj}\dell_{i}-\delta_{mi}\dell_{j}.\label{commutator-rectangular-ang}
\end{align}
We also define commutators between the higher order differential operator defined in $(\ref{Ndell-def})$, and $\grad$:
\begin{align}
\left(\comm{\grad}{\Ndell{m}{n}}F\right)^{i}_{j}&=\dell_{j}\left(\Ndell{m}{n}F^{i}\right)-\Ndell{m}{n}\left(\dell_{j}F^{i}\right).\label{commutator-grad-Ndell}
\end{align}
We can do the same thing with $\ngrad$:
\begin{align}
\left(\comm{\ngrad}{\Ndell{m}{n}}F\right)^{i}_{j}&=\ninv[\kappa]^{k}_{j}\dell_{k}\left(\Ndell{m}{n}F^{i}\right)-\Ndell{m}{n}\left(\ninv[\kappa]^{k}_{j}\dell_{k}F^{i}\right),\label{commutator-ngrad-Ndell}\\
\left(\comm{\ngrad}{\nabla^{\underline{k}}}F\right)^{i}_{j}&=\ninv[\kappa]^{k}_{j}\dell_{k}\left(\nabla^{\underline{k}}F^{i}\right)-\nabla^{\underline{k}}\left(\ninv[\kappa]^{k}_{j}\dell_{k}F^{i}\right).\label{commutator-ngrad-Dk}
\end{align}
There is no corresponding definition to $(\ref{commutator-ngrad-Dk})$ for $\grad$, as $\grad$ and $\nabla^{\underline{k}}$ commute for all $\underline{k}\in\mathbb{Z}_{\geq0}^{3}$. Note that $(\ref{commutator-ngrad-Ndell})$ and $(\ref{commutator-ngrad-Dk})$ also define analogous objects for $\nCurl$ and $\ndiv$ as the former is $\ngrad-\ngrad^{\intercal}$, and the latter is $\tr{\ngrad}$.

\section{Energy Function}\label{energy-function-section}
Following the strategy set out in Remark $\ref{degeneracy-modified-spherical-derivatives}$, define a smooth radial cutoff function $\chi$ on the closed unit ball such that
\begin{align}
\chi=\twopartdef{1}{r\in[3/4,1]}{0}{r\in[0,1/4]}.\label{cutoff-function}
\end{align}
In addition define $\bar{\chi}$ by
\begin{align}
\bar{\chi}=1-\chi.\label{cutoff-function-conjugate}
\end{align}
Recalling the definition of $\alpha$ given in $(\ref{alpha-definition})$, we now define our energy spaces.
\begin{definition}[Energy Spaces]{\label{energy-function-spaces-definition}}
	Let $b\in\mathbb{Z}_{\geq0}$ and define the space $\X{b}_{\kappa}$, $\kappa=1,2,\dots,N$, by
	\begin{align*}
	\X{b}_{\kappa}=\left\{\tilde{W}_{\kappa}^{\frac{\alpha}{2}}F\in L^{2}(\Omega): \sum_{m+|\underline{n}|,|\underline{k}|=0}^{b}\int_{\Omega} \tilde{W}_{\kappa}^{\alpha+m}\left(\chi\left|\Ndell{m}{n}F\right|^{2}+\bar{\chi}\left|\nabla^{\underline{k}}F\right|^{2}\right)dx<\infty.\right\}
	\end{align*}
	The norm of $\X{b}_{\kappa}$ is given by
	\begin{align}
	\left\|F\right\|^{2}_{\X{b}_{\kappa}}=\sum_{m+|\underline{n}|=0}^{b}\int_{\Omega}\chi \tilde{W}_{\kappa}^{\alpha+m}\left|\Ndell{m}{n}F\right|^{2}dx+\sum_{|\underline{k}|=0}^{b}\int_{\Omega}\bar{\chi} \tilde{W}_{\kappa}^{\alpha}\left|\nabla^{\underline{k}}F\right|^{2}dx.\label{energy-space-norm}
	\end{align}
	For $\mathscr{D}\in\{\grad,\ngrad,\dive,\ndiv,\Curl,\nCurl\}$ define the space $\Y{b}{\kappa}{\mathscr{D}}$ by
	\begin{align*}
	\Y{b}{\kappa}{\mathscr{D}}=\left\{\tilde{W}_{\kappa}^{\frac{1+\alpha}{2}}\mathscr{D}{F}\in L^{2}(\Omega): \sum_{m+|\underline{n}|,|\underline{k}|=0}^{b}\int_{\Omega} \tilde{W}_{\kappa}^{1+\alpha+m}\njac_{\kappa}^{-\frac{1}{\alpha}}\left(\chi\left|\mathscr{D}{\Ndell{m}{n}F}\right|^{2}+\bar{\chi}\left|\mathscr{D}{\nabla^{\underline{k}}F}\right|^{2}\right)dx<\infty.\right\},
	\end{align*}
	with semi norm given by
	\begin{align}
	\left\|F\right\|^{2}_{\Y{b}{\kappa}{\mathscr{D}}}=\sum_{m+|\underline{n}|=0}^{b}\int_{\Omega}\chi \tilde{W}_{\kappa}^{1+\alpha+m}\njac_{\kappa}^{-\frac{1}{\alpha}}\left|\mathscr{D}{\Ndell{m}{n}F}\right|^{2}dx+\sum_{|\underline{k}|=0}^{b}\int_{\Omega}\bar{\chi} \tilde{W}_{\kappa}^{1+\alpha}\njac_{\kappa}^{-\frac{1}{\alpha}}\left|\mathscr{D}{\nabla^{\underline{k}}F}\right|^{2}dx.\label{energy-space-norm-2}
	\end{align}
\end{definition}
\begin{remark}\label{power-of-W-on-interior-norm}The powers of $\tilde{W}_{\kappa}$ in the integrals involving $\bar{\chi}$ in $(\ref{energy-space-norm})-(\ref{energy-space-norm-2})$ are not consistent with what we see in the definition of the energy spaces. However, on $\supp{\bar{\chi}}$, $\tilde{W}_{\kappa}\sim1$ so this discrepancy does not create issues.
\end{remark}
We now define our higher order energy function and $\nCurl$ energy function. 
\begin{definition}[Higher Order Energy and Curl Functions]{\label{energy-function}}Let $b\in\mathbb{Z}_{\geq0}$. Then the individual energy functions of order $b$ for $(\dell_{\tau}\theta_{\kappa},\theta_{\kappa})$, $\kappa=1,2,\dots,N$, is given by
	\begin{align}
	S_{\kappa}^{b}(\tau)=S^{b}(\dell_{\tau}\theta_{\kappa},\theta_{\kappa},\tau)\coloneqq\sup_{0\leq \tau'\leq \tau}\left(\frac{1}{\delta}e^{\beta\tau'}\left\|\dell_{\tau}\theta_{\kappa}(\tau')\right\|_{\X{b}_{\kappa}}^{2}+\left\|\theta_{\kappa}(\tau')\right\|_{\X{b}_{\kappa}}^{2}+\left\|\theta_{\kappa}(\tau')\right\|_{\Y{b}{\kappa}{\ngrad}}^{2}+\frac{1}{\alpha}\left\|\theta_{\kappa}(\tau')\right\|_{\Y{b}{\kappa}{\ndiv}}^{2}\right).\label{individual-energy-function}
	\end{align}
	The $\nCurl$ energy function of order $b$ for $(\dell_{\tau}\theta_{\kappa},\theta_{\kappa})$ is given by
	\begin{align}
	C^{b}_{\kappa}(t)=C^{b}(\dell_{\tau}\theta_{\kappa},\theta_{\kappa},\tau)\coloneqq\sup_{0\leq \tau'\leq \tau}\left(\left\|\dell_{\tau}\theta_{\kappa}(\tau')\right\|_{\Y{b}{\kappa}{\nCurl}}^{2}+\left\|\theta_{\kappa}(\tau')\right\|_{\Y{b}{\kappa}{\nCurl}}^{2}\right).\label{individual-curl-energy-function}
	\end{align}
	Then the cumulative energy function and $\nCurl$ energy function of order $b$ are
	\begin{align}
	S_{b}(\tau)&=\sum_{\kappa=1}^{N}S^{b}_{\kappa}(\tau),\label{cumulative-energy-function}\\
	C_{b}(\tau)&=\sum_{\kappa=1}^{N}C^{b}_{\kappa}(\tau).\label{cumulative-curl-energy-function}
	\end{align}
\end{definition}

\section{Main Result and A Priori Assumptions}\label{main-result-a-priori-assumptions-section}

\subsection{Main Result}\label{main-result-section}

In this section we will state the central theorem of this paper, global-in-time existence of the system $\eqref{delta-rescaled-lagrangian-euler-1}$--$\eqref{delta-rescaled-lagrangian-euler-3}$. First we state the local-in-time theory for this system. Details on its construction are given in Appendix \ref{local-well-posedness}.

Recall the definition of $\delta$, the smallness parameter for the initial density profiles defined in $\eqref{definition-of-delta}$, and the repulsive velocities of each star $\mu_{\kappa}$, defined in $\eqref{eta-pre-ansatz}$.
\begin{theorem}[Local Well-Posedness of the Free Boundary $N$-Body Euler-Poisson System]\label{local-well-posedness-theorem}Assume $\gamma=1+\frac{1}{n}$ for some $n\in\mathbb{Z}_{\geq2}$, or $\gamma\in(1,14/13)$. Let $M$ be an integer such that $M\geq2\ceil*{\alpha}+12$. Additionally, assume that $\nabla\tilde{W}_{\kappa},\nabla\mu_{\kappa}\in\X{M}_{\kappa}$ for $\kappa=1,2,\dots,N$. Let $\mathring{u}_{\kappa}:\Omega_{\kappa}\rightarrow\mathbb{R}^{3}$ and $(\mathring{\eta}_{\kappa}(x)=x+\bar{x}_{\kappa}):\Omega\rightarrow\Omega_{\kappa}$ be such that the \emph{separation condition}
	\begin{align}
	&\min_{\substack{\kappa,\kappa'=1,2,\dots,N \\ \kappa\neq\kappa'}}\dist(\Omega_{\kappa},\Omega_{\kappa'})=\dist(\Omega+\bar{x}_{\kappa},\Omega+\bar{x}_{\kappa'})>0\label{separation-condition-local-theory}
	\end{align}
	holds, and the bound
	\begin{align}
	S_{M}(0)+C_{M}(0)=\sum_{\kappa=1}^{N}\left[S_{\kappa}^{M}(\mathring{u}_{\kappa}(\mathring{\eta}_{\kappa}(x))-(\mathring{\eta}_{\kappa}(x)+\mu_{\kappa}(x)),0,0)+C_{\kappa}^{M}(\mathring{u}_{\kappa}(\mathring{\eta}_{\kappa}(x))-(\mathring{\eta}_{\kappa}(x)+\mu_{\kappa}(x)),0,0)\right]<\infty\label{initial-data-inequality-local-theorem}
	\end{align}
	holds. Then there exists a $T>0$ such that we can find a unique solution $\{(\dell_{\tau}\theta_{\kappa},\theta_{\kappa})|\kappa=1,2,\dots,N\}$ to $\eqref{delta-rescaled-lagrangian-euler-1}$--$\eqref{delta-rescaled-lagrangian-euler-3}$ on the interval $[0,T]$ with
	\begin{align}
	&S_{M}(\tau) + C_{M}(\tau)
	\leq2(S_{M}(0)+C_{M}(0)+\sqrt{\delta}),\ \ \ \forall\tau\in[0,T],\label{local-in-time-energy-bound}\\
	&(\dell_{\tau}\theta_{\kappa}(0),\theta_{\kappa}(0))=(\mathring{u}_{\kappa}(\mathring{\eta}_{\kappa}(x))-(\mathring{\eta}_{\kappa}(x)+\mu_{\kappa}(x)),0)\ \ \ \kappa=1,2,\dots,N.\label{local-in-time-initial-data}
	\end{align}
	Moreover, the function $\tau\mapsto S_{M}(\tau)$ is continuous.
\end{theorem}
We introduce the \emph{Strong Separation Condition} (SSC), that will be key in formulating our main result.
\begin{definition}[Strongly Separated Initial Configurations of the Free Boundary $N$-Body Euler-Poisson System]\label{strong-separation-condition-definition} We say that an initial configuration for the stars in the Free Boundary $N$-Body Euler-Poisson System satisfies the Strong Separation Condition for $L>0$ if
	\begin{align} 
	&\min_{\substack{\kappa,\kappa'=1,2,\dots,N \\ \kappa\neq\kappa'}}\inf_{\lambda\in[0,1]}\inf_{x,z\in\Omega}\left|(1-\lambda)(\mu_{\kappa}(x)-\mu_{\kappa'}(z))+\lambda(\bar{x}_{\kappa}-\bar{x}_{\kappa'})\right|\geq L\label{a-priori-lower-bound-convex-combination-mu-centres-main-theorem}
	\end{align}
holds.
\end{definition}
\begin{remark}\label{covex-combination} We see that the Strong Separation Condition is a lower bound on any convex combination of the relative positions and velocities.
\end{remark}
\begin{remark}\label{strong-separation-condition-remark}Note that the SSC implies the separation condition $\eqref{separation-condition-local-theory}$ automatically for large enough $L$. Moreover, we remark that $\eqref{a-priori-lower-bound-convex-combination-mu-centres-main-theorem}$ arises naturally from finding sufficient bounds on the tidal terms $\mathscr{I}_{[\kappa,\kappa']}$, defined in $\eqref{lagrangian-gradient-of-lagrangian-potential}$, to prove global-in-time existence of our solutions. Hence studying the tidal terms gives us crucial information on the initial geometry of our star configurations.
\end{remark}
\begin{theorem}[Global Existence for the Free Boundary $N$-Body Euler-Poisson System]\label{main-theorem}
	Assume $\gamma=1+\frac{1}{n}$ for some $n\in\mathbb{Z}_{\geq2}$, or $\gamma\in(1,14/13)$. Let $\{(\dell_{\tau}\theta_{\kappa},\theta_{\kappa})|\kappa=1,2,\dots,N\}$ be the solution to $(\ref{delta-rescaled-lagrangian-euler-1})-(\ref{delta-rescaled-lagrangian-euler-3})$ on $[0,T]$, in the sense of Theorem \ref{local-well-posedness-theorem}, for some $M\geq2\ceil*{\alpha}+12$. Additionally, assume that $\nabla\tilde{W}_{\kappa}\in\X{M}_{\kappa}$ for $\kappa=1,2,\dots,N$. Then there exists $L>0$ sufficiently large, and $\delta,\varepsilon_{0}>0$ sufficiently small such that if the Strong Separation Condition from Definition $\ref{strong-separation-condition-definition}$ holds for $L$, then for all $0\leq\varepsilon\leq\varepsilon_{0}$ with
	\begin{align}
	S_{M}(0)+C_{M}(0)+\left\|\nabla\mu_{\kappa}\right\|_{\X{M}_{\kappa}}^{2}\leq\varepsilon,\label{global-in-time-initial-data-smallness}
	\end{align}
	there exists a global-in-time solution to $\eqref{delta-rescaled-lagrangian-euler-1}$--$\eqref{delta-rescaled-lagrangian-euler-3}$, with the initial conditions as in $\eqref{delta-rescaled-lagrangian-euler-2}$, and there exists a constant $C$ such that
	\begin{align}
	S_{M}(\tau)\leq C\left(\varepsilon+\sqrt{\delta}\right),\ \ \ \forall\tau\in[0,\infty).\label{global-energy-estimate}
	\end{align}
	Additionally, there are $\tau$\bcr-\ec independent functions $\theta_{\kappa}^{(\infty)}:\Omega\rightarrow\mathbb{R}^{3}$ such that
	\begin{align}
	\left\|\theta_{\kappa}(\tau)-\theta_{\kappa}^{(\infty)}\right\|_{\X{M}_{\kappa}}\rightarrow0\ \ \ \tau\rightarrow\infty.\label{theta-infinity-statement}
	\end{align}
\end{theorem}

\subsection{A Priori Assumptions}\label{assumptions-section}

In this section we list all the assumptions we will make to prove Theorem $\ref{main-theorem}$. First we make explicit the assumption, implicit in $\eqref{global-in-time-initial-data-smallness}$, that $\nabla\mu_{\kappa}$ must be small in the $\X{M}_{\kappa}$ norm for $\kappa=1,2,\dots,N$. Let $0<\varepsilon_{1}\ll 1$. We assume
\begin{align}
\left\|\nabla\mu_{\kappa}\right\|_{\X{M}_{\kappa}}\leq\sqrt{\varepsilon_{1}},\ \ \kappa=1,2,\dots,N.\label{a-priori-smallness-mu-norm}
\end{align}
Let $0<\varepsilon_{2}\ll1$. We specify the $L$ for which the Strong Separation Condition $\eqref{a-priori-lower-bound-convex-combination-mu-centres-main-theorem}$ holds. Assume
\begin{align}
\min_{\substack{\kappa,\kappa'=1,2,\dots,N \\ \kappa\neq\kappa'}}\inf_{\lambda\in[0,1]}\inf_{x,z\in\Omega}\left|(1-\lambda)(\mu_{\kappa}(x)-\mu_{\kappa'}(z))+\lambda(\bar{x}_{\kappa}-\bar{x}_{\kappa'})\right|\geq 3+\varepsilon_{2}.\label{a-priori-lower-bound-convex-combination-mu-centres}
\end{align} 
Now we state our a priori assumptions on the solution to $\eqref{delta-rescaled-lagrangian-euler-1}$--$\eqref{delta-rescaled-lagrangian-euler-3}$. There exists a $T>0$, depending on $\varepsilon_{2}$, such that on the time interval $[0,T]$, the solution to $\eqref{delta-rescaled-lagrangian-euler-1}$--$\eqref{delta-rescaled-lagrangian-euler-3}$ satisfies
\begin{align}
&S_{M}(\tau)\leq \varepsilon_{2}, \ \ \forall\tau\in[0,T],\label{a-priori-bound-cumulative-energy}\\
&\left\|\ninv[\kappa]-\mathbb{I}\right\|_{L^{\infty}(\Omega)}\leq\varepsilon_{2},\ \ \kappa=1,2,\dots,N,\ \ \forall\tau\in[0,T],\label{a-priori-bound-grad-zeta}\\
&\left\|\njac-1\right\|_{L^{\infty}(\Omega)}\leq\varepsilon_{2},\ \ \kappa=1,2\dots,N\ \ \forall\tau\in[0,T],\label{a-priori-bound-jacobian}\\
&\sum_{\kappa=1}^{N}\left\|\theta_{\kappa}\right\|_{L^{\infty}(\Omega)}\leq\varepsilon_{2}, \ \ \forall\tau\in[0,T].\label{a-priori-sum-of-theta-L-infinity-norms}
\end{align}
As a justification for why we can make our a priori assumptions, note that such a time interval $[0,T]$ where $\eqref{a-priori-lower-bound-convex-combination-mu-centres}$--$\eqref{a-priori-sum-of-theta-L-infinity-norms}$ hold must exist for small enough initial data by the local well-posedness theory set out in Theorem $\ref{local-well-posedness-theorem}$. Once we have used these assumptions to prove Theorem $\ref{main-theorem}$, we shall use the global-in-time energy bound $\eqref{global-energy-estimate}$ to improve our a priori assumptions, thereby justifying them via a continuity argument.

Just like $\delta$ defined in $\eqref{definition-of-delta}$, both $\varepsilon_{1}$ and $\varepsilon_{2}$ will be need to be small for our proofs. Throughout they will be $\ll1$, and where needed we will state any explicit bounds for them.
\begin{remark}[Asymptotic velocities]\label{asymptotic-velocities-and-possible-initial-configurations} Note that assumption $\eqref{global-in-time-initial-data-smallness}$ means that each repuslive velocity $\mu_{\kappa}$ is very close to a constant vector $\bar{\mu}_{\kappa}$. Thus we can write our initial velocities as
	\begin{align}
	\mathring{u}_{\kappa}(x)=x-\bar{x}_{\kappa}+\bar{\mu}_{\kappa}+o\left(1\right),\ \text{as}\ \varepsilon,\delta\to0, \ x\in\Omega_{\kappa}.\label{initial-velocities-approximate-form}
	\end{align}
As discussed in Section $\ref{flow-map-section}$, $x-\bar{x}_{\kappa}$ corresponds to expansion, and $\bar{\mu}_{\kappa}$ corresponds to repulsion. Moreover, the Lagrangian flow maps, and the Lagrangian velocities take the form
	\begin{align}
	\eta_{\kappa}(t,x)&=x+\bar{x}_{\kappa}+t\left(x+\bar{\mu}_{\kappa}+o\left(1\right)\right),\ \text{as}\ \varepsilon,\delta\to0, \ x\in\Omega,\label{lagrangian-flow-map-approximate-form}\\
	v_{\kappa}(t,x)=u_{\kappa}(t,\eta_{\kappa}(t,x))&=\dell_{t}\eta_{\kappa}(t,x)=x+\bar{\mu}_{\kappa}+o\left(1\right),\ \text{as}\ \varepsilon,\delta\to0, \ x\in\Omega.\label{lagrangian-velocity-approximate-form}
	\end{align}
\end{remark}

\begin{remark}[Examples of constant repulsive velocities]\label{constant-repulsive-velocities} We find a large class of initial configurations that satisfy the Strong Separation Condition $\eqref{a-priori-lower-bound-convex-combination-mu-centres}$ if we set $\mu_{\kappa}(x)\equiv\bar{\mu}_{\kappa}=\bar{x}_{\kappa}$ for each $\kappa=1,2,\dots,N$. Then $\eqref{a-priori-lower-bound-convex-combination-mu-centres}$ simplifies to
	\begin{align}
	\min_{\substack{\kappa,\kappa'=1,2,\dots,N \\ \kappa\neq\kappa'}}\left|\bar{x}_{\kappa}-\bar{x}_{\kappa'}\right|\geq 3+\varepsilon_{2}.\label{simplified-strong-separation-condition}
	\end{align}
Thus, any initial configuration of stars, with repulsive velocity equal to the initial displacement of their centre from the origin, launches a global-in-time solution as long as they are sufficiently separated so that $\eqref{simplified-strong-separation-condition}$ holds. Figure $\ref{fig:initial-mu}$ represents a particularly symmetric example.

\begin{figure}
	\centering
	\begin{tikzpicture}
	\filldraw[black] (5,5) circle (2pt) node[anchor=north west] {Origin};
	\filldraw[color=blue!60, fill=blue!5, very thick](6,3) circle (0.25);
	\draw[color=blue!60, very thick, dashed] (6,3) circle (0.4);
	\filldraw[color=blue!60, fill=blue!5, very thick](7,5) circle (0.25);
	\draw[color=blue!60, very thick, dashed] (7,5) circle (0.4);
	\filldraw[color=blue!60, fill=blue!5, very thick](6,7) circle (0.25);
	\draw[color=blue!60, very thick, dashed] (6,7) circle (0.4);
	\filldraw[color=blue!60, fill=blue!5, very thick](4,7) circle (0.25);
	\draw[color=blue!60, very thick, dashed] (4,7) circle (0.4);
	\filldraw[color=blue!60, fill=blue!5, very thick](3,5) circle (0.25);
	\draw[color=blue!60, very thick, dashed] (3,5) circle (0.4);
	\filldraw[color=blue!60, fill=blue!5, very thick](4,3) circle (0.25);
	\draw[color=blue!60, very thick, dashed] (4,3) circle (0.4);
	\draw[red,thick,->] (6,3)--(7,1);
	\draw[red,thick,->] (7,5)--(9,5);
	\draw[red,thick,->] (6,7)--(7,9);
	\draw[red,thick,->] (4,7)--(3,9);
	\draw[red,thick,->] (3,5)--(1,5);
	\draw[red,thick,->] (4,3)--(3,1);
	\draw[red,thick,dotted] (5,5)--(6,3);
	\draw[red,thick,dotted] (5,5)--(7,5);
	\draw[red,thick,dotted] (5,5)--(6,7);
	\draw[red,thick,dotted] (5,5)--(4,7);
	\draw[red,thick,dotted] (5,5)--(3,5);
	\draw[red,thick,dotted] (5,5)--(4,3);
	\draw[black,thick,->] (5,9.5)--(4.5,6) node[pos=0,black,anchor=south] {Initial displacement of centres from origin, $\bar{x}_{\kappa}$.};
	\draw[black,thick,->] (5,9.5)--(5.5,6);
	\draw[black,thick,->] (8,3)--(6.5,2) node[pos=0,black,anchor=south west] {Repulsive velocities, $\mu_{\kappa}(x)\equiv\bar{x}_{\kappa}$.};
	\draw[black,thick,->]  (8,3)--(8,5);
	\draw[black,thick,->]  (2,7)--(3.75,7) node[pos=0,black,anchor=north east] {Stars at $\tau=0$.};
	\draw[black,thick,->]  (2,7)--(2.888,5.224);
	\draw[black,thick,->] (2,2)--(3.6,3) node[pos=0,black,anchor=north east] {Expansion of stars for $\tau>0$.};
	\end{tikzpicture}
	\caption{Example of an initial configuration with $\mu_{\kappa}(x)\equiv\bar{x}_{\kappa}$.} \label{fig:initial-mu}
\end{figure}
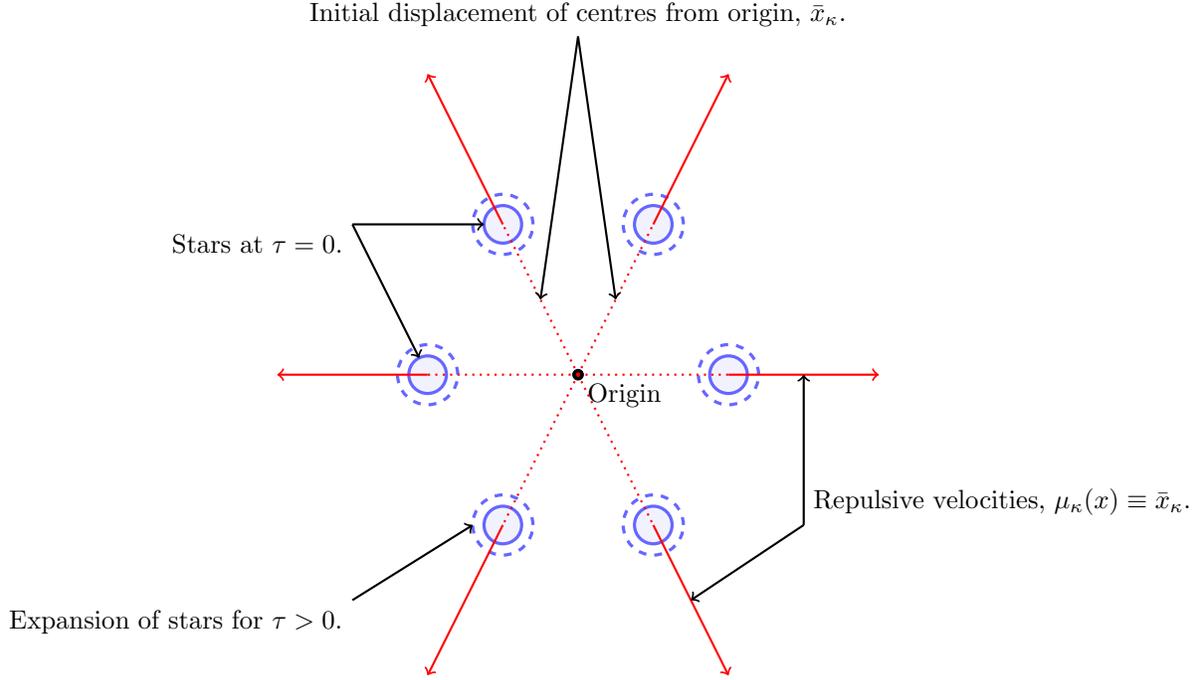
\end{remark}

\section{Estimates for the Gravitational Potential}\label{gravitational-potential-estimates-section}
In this section we obtain the estimates we need for the self interaction and tidal terms defined in $\eqref{lagrangian-gradient-of-lagrangian-potential}$, $\mathscr{G}_{\kappa}$ and $\mathscr{I}_{[\kappa,\kappa']}$, to prove global-in-time existence of our solution. The estimates for $\mathscr{G}_{\kappa}$ will be closely follow the methods used in~\cite{HaJa4}, as they have to deal with the corresponding term in the One Body Euler-Poisson system.

However, the tidal terms $\mathscr{I}_{[\kappa,\kappa']}$ will clearly only appear in the case of 
two or more interacting bodies, and in the estimates for these terms we make crucial use of the Strong Separation Condition $\eqref{a-priori-lower-bound-convex-combination-mu-centres}$.

\subsection{Tidal Term Estimates}

\begin{proposition}\label{tidal-term-estimates-proposition}
	Assume $\gamma=1+\frac{1}{n}$ for some $n\in\mathbb{N}$, or $\gamma\in(1,14/13)$. Let $\{(\dell_{\tau}\theta_{\kappa},\theta_{\kappa})|\kappa=1,2,\dots,N\}$ be the solution to $(\ref{delta-rescaled-lagrangian-euler-1})-(\ref{delta-rescaled-lagrangian-euler-3})$ on $[0,T]$, in the sense of Theorem \ref{local-well-posedness-theorem}, for some $M\geq2\ceil*{\alpha}+12$. Suppose the assumptions $\eqref{a-priori-smallness-mu-norm}$--$\eqref{a-priori-sum-of-theta-L-infinity-norms}$ hold. Assume that $\nabla\tilde{W}_{\kappa}\in\X{M}_{\kappa}$ for $\kappa=1,2,\dots,N$. Then for all $\tau\in[0,T]$, we have:
	\begin{align}
	\sum_{\kappa'\neq\kappa}\left(\sum_{m+|\underline{n}|=0}^{M}\int_{\Omega}\chi\tilde{W}_{\kappa}^{\alpha+m}\left|\Ndell{m}{n}\mathscr{I}_{[\kappa,\kappa']}\right|^{2}dx+\sum_{|\underline{k}|=0}^{M}\int_{\Omega}\bar{\chi}\tilde{W}_{\kappa}^{\alpha}\left|\nabla^{\underline{k}}\mathscr{I}_{[\kappa,\kappa']}\right|^{2}dx\right)\lesssim \delta^{2\alpha}e^{-2\tau}.\label{tidal-term-estimates-statement}
	\end{align}
\end{proposition}
	
	\begin{proof}[\textbf{Proof}]
		We begin with the integrals on the left hand side of $\eqref{tidal-term-estimates-statement}$ that are localised on $\supp{\chi}$. For $i=1,2,3$, we have
		\begin{align}
		&\Ndell{m}{n}\mathscr{I}_{[\kappa,\kappa']}^{i}=\delta^{\alpha}e^{-\tau}\Ndell{m}{n}\left(\int_{\Omega}\frac{\left(\zeta_{\kappa}^{i}(x)-\zeta_{\kappa'}^{i}(z)\right)\tilde{W}_{\kappa'}^{\alpha}}{\left|\zeta_{\kappa}(x)-\zeta_{\kappa'}(z)\right|^{3}}dz\right)\nonumber\\
		&=\underbrace{\delta^{\alpha}e^{-\tau}\Ndell{m}{n}\zeta_{\kappa}^{i}(x)\int_{\Omega}\frac{\tilde{W}_{\kappa'}^{\alpha}dz}{\left|\zeta_{\kappa}(x)-\zeta_{\kappa'}(z)\right|^{3}}}_{\mathcal{I}_{1}}-\underbrace{\delta^{\alpha}e^{-\tau}\int_{\Omega}\Ndell{m}{n}\left(\frac{1}{\left|\zeta_{\kappa}(x)-\zeta_{\kappa'}(z)\right|^{3}}\right)\zeta_{\kappa'}^{i}(z)\tilde{W}_{\kappa'}^{\alpha}dz}_{\mathcal{I}_{2}}\nonumber\\
		&+\underbrace{\sum_{\substack{a+c=m \\ |\underline{b}|+|\underline{d}|=|\underline{n}| \\ a+|\underline{b}|\leq m+|\underline{n}|-1}}\delta^{\alpha}e^{-\tau}\mathscr{L}(a,\underline{b},c,\underline{d})\Ndell{a}{b}\zeta_{\kappa}^{i}(x)\int_{\Omega}\Ndell{c}{d}\left(\frac{1}{\left|\zeta_{\kappa}(x)-\zeta_{\kappa'}(z)\right|^{3}}\right)\zeta_{\kappa'}^{i}(z)\tilde{W}_{\kappa'}^{\alpha}dz}_{\mathcal{I}_{3}},\label{tidal-term-estimates-1}
		\end{align}
		due to applying the Liebniz rule to $\Ndell{m}{n}\mathscr{I}^{i}_{[\kappa,\kappa']}$, and the last term on the right hand side, $\mathcal{I}_{3}$, is every resulting term except for the cases when all the derivatives fall on either $\left(\zeta_{\kappa}^{i}(x)-\zeta_{\kappa'}^{i}(z)\right)$, which is $\mathcal{I}_{1}$, or $\left|\zeta_{\kappa}(x)-\zeta_{\kappa'}(z)\right|^{-3}$, which is $\mathcal{I}_{2}$. The $\mathscr{L}(a,\underline{b},c,\underline{d})$ are constant coefficients resulting from the application of the Leibniz rule.
		
		We will show the estimates for $\mathcal{I}_{1}$ and $\mathcal{I}_{2}$, with $\mathcal{I}_{3}$ following similarly. Since these tidal terms measure the gravitational interaction between the stars, it is natural that the separation of the bodies would influence our estimates, and indeed for both $\mathcal{I}_{1}$ and $\mathcal{I}_{2}$, the Strong Separation Condition $\eqref{a-priori-lower-bound-convex-combination-mu-centres}$ is used. We begin with $\mathcal{I}_{1}$ which is the simplest term, and then move on to $\mathcal{I}_{2}$ which requires more care, especially when applying the Leibniz rule to $\Ndell{m}{n}\left(\left|\zeta_{\kappa}(x)-\zeta_{\kappa'}(z)\right|^{-3}\right)$.\\
		
		\noindent \textbf{Bound for $\mathcal{I}_{1}$.}\\
		
		\noindent We have that $\zeta_{\kappa}(\tau,x)=x+e^{-\tau}\bar{x}_{\kappa}+(1-e^{-\tau})\mu_{\kappa}(x)+\theta_{\kappa}(\tau,x)$, and so
		\begin{align}
		\frac{1}{\left|\zeta_{\kappa}(x)-\zeta_{\kappa'}(z)\right|^{3}}&=\frac{1}{\left|(1-e^{-\tau})(\mu_{\kappa}(x)-\mu_{\kappa'}(z))+e^{-\tau}(\bar{x}_{\kappa}-\bar{x}_{\kappa'})+(x-z)+(\theta_{\kappa}(\tau,x)-\theta_{\kappa'}(\tau,z))\right|^{3}}\nonumber\\
		&\leq\frac{1}{\bigg\vert\left|(1-e^{-\tau})(\mu_{\kappa}(x)-\mu_{\kappa'}(z))+e^{-\tau}(\bar{x}_{\kappa}-\bar{x}_{\kappa'})\right|-\left|(x-z)+(\theta_{\kappa}(\tau,x)-\theta_{\kappa'}(\tau,z))\right|\bigg\vert^{3}}\nonumber\\
		&\leq1,\label{denominator-tidal-term-integral-bound}
		\end{align}
		where the first bound is due to the reverse triangle inequality, and the second is due to the Strong Separation Condition $\eqref{a-priori-lower-bound-convex-combination-mu-centres}$, with $\lambda=e^{-\tau}$, as well as the bound
		\begin{align*}
		\left|(x-z)+(\theta_{\kappa}(\tau,x)-\theta_{\kappa'}(\tau,z))\right|\leq2 +\sup_{0\leq\tau\leq T}\left(\sum_{\kappa=1}^{N}\left\|\theta_{\kappa}\right\|_{L^{\infty}(\Omega)}\right)\leq2+\varepsilon_{2},
		\end{align*}
		due to the a priori assumption $\eqref{a-priori-sum-of-theta-L-infinity-norms}$. Thus
		\begin{align}
		&\int_{\Omega}\chi\tilde{W}_{\kappa}^{\alpha+m}\left|\mathcal{I}_{1}\right|^{2}dx\lesssim\delta^{2\alpha}e^{-2\tau}\int_{\Omega}\chi\tilde{W}_{\kappa}^{\alpha+m}\left|\Ndell{m}{n}\zeta_{\kappa}(x)\int_{\Omega}\frac{\tilde{W}_{\kappa'}^{\alpha}dz}{\left|\zeta_{\kappa}(x)-\zeta_{\kappa'}(z)\right|^{3}}\right|^{2}dx\lesssim\delta^{2\alpha}e^{-2\tau}\int_{\Omega}\chi\tilde{W}_{\kappa}^{\alpha+m}\left|\Ndell{m}{n}\zeta_{\kappa}\right|^{2}dx\nonumber\\
		&\lesssim\delta^{2\alpha}e^{-2\tau}\int_{\Omega}\chi\tilde{W}_{\kappa}^{\alpha+m}\left|\Ndell{m}{n}\theta_{\kappa}\right|^{2}dx+\delta^{2\alpha}e^{-2\tau}\int_{\Omega}\chi\tilde{W}_{\kappa}^{\alpha+m}\left|\Ndell{m}{n}\left(x+e^{-\tau}\bar{x}_{\kappa}+(1-e^{-\tau})\mu_{\kappa}\right)\right|^{2}dx\nonumber\\
		&\lesssim\delta^{2\alpha}e^{-2\tau}\left(1+S_{M}(\tau)\right)+\delta^{2\alpha}e^{-2\tau}(1-e^{-\tau})\int_{\Omega}\chi\tilde{W}_{\kappa}^{\alpha+m}\left|\Ndell{m}{n}\mu_{\kappa}\right|^{2}dx.\label{tidal-term-I-1-estimate-1}
		\end{align}
		For the second bound, we use $(\ref{denominator-tidal-term-integral-bound})$ and the fact that $\tilde{W}_{\kappa}$ is bounded on $\Omega$ for all $\kappa$ to bound the $z$ integral by a constant. Note that $x+e^{-\tau}\bar{x}_{\kappa}$ is smooth so we can bound derivatives of this term in $L^{\infty}(\Omega)$. We  bound the $\theta_{\kappa}$ term by the cumulative energy function $S_{M}$. It remains to bound the $\mu_{\kappa}$ on the right hand side.
		
		If $m+|\underline{n}|>0$, then using $\eqref{angular-derivative-def}$ and $\eqref{radial-derivative-def}$, we have
		\begin{align}
		\delta^{2\alpha}e^{-2\tau}(1-e^{-\tau})\int_{\Omega}\chi\tilde{W}_{\kappa}^{\alpha+m}\left|\Ndell{m}{n}\mu_{\kappa}\right|^{2}dx\lesssim\delta^{2\alpha}e^{-2\tau}\sum_{a+|\underline{b}|=0}^{m+|\underline{n}|}\int_{\Omega}\chi\tilde{W}_{\kappa}^{\alpha+a}\left|\Ndell{a}{b}\nabla\mu_{\kappa}\right|^{2}dx\lesssim\delta^{2\alpha}e^{-2\tau}\label{tidal-term-I-1-estimate-2}
		\end{align}
		as $\nabla\mu_{\kappa}\in\X{M}_{\kappa}$ from $\eqref{a-priori-smallness-mu-norm}$. If $m+|\underline{n}|=0$, it is enough to bound $\mu_{\kappa}$ in $L^{\infty}(\Omega)$ as $\chi$ and $\tilde{W}_{\kappa}^{\alpha}$ are bounded in $L^{\infty}(\Omega)$. We have
		\begin{align}
		\left\|\mu_{\kappa}\right\|_{L^{\infty}(\Omega)}\lesssim1+\left\|\nabla\mu_{\kappa}\right\|_{\X{M}_{\kappa}}\lesssim1.\label{mu-L-infinity-estimate}
		\end{align}
		To obtain this bound, we use the mean value theorem and $\eqref{L-infinity-energy-space-bound-no-weights-statement-1}$ in Lemma $\ref{L-infinity-energy-space-bound}$, as well as $\nabla\mu_{\kappa}\in\X{M}_{\kappa}$ again. Combining bounds $\eqref{tidal-term-I-1-estimate-1}$--$\eqref{mu-L-infinity-estimate}$ and applying a priori assumption $\eqref{a-priori-bound-cumulative-energy}$, we obtain
		\begin{align}
		\int_{\Omega}\chi\tilde{W}_{\kappa}^{\alpha}\left|\mathcal{I}_{1}\right|^{2}dx\lesssim\delta^{2\alpha}e^{-2\tau}(1+S_{M}(\tau))\lesssim\delta^{2\alpha}e^{-2\tau}.\label{tidal-term-I-1-estimate-3}
		\end{align}
		\noindent \textbf{Bound for $\mathcal{I}_{2}$.}\\
		
		\noindent When $m+|\underline{n}|=0$, we have that
		\begin{align}
		\int_{\Omega}\chi\tilde{W}_{\kappa}^{\alpha+m}\left|\mathcal{I}_{2}\right|^{2}dx&\lesssim\delta^{2\alpha}e^{-2\tau}\int_{\Omega}\chi\tilde{W}_{\kappa}^{\alpha+m}\left|\int_{\Omega}\frac{\zeta_{\kappa'}(z)\tilde{W}_{\kappa'}^{\alpha}dz}{\left|\zeta_{\kappa}(x)-\zeta_{\kappa'}(z)\right|^{3}}\right|^{2}dx\lesssim\delta^{2\alpha}e^{-2\tau}\left\|\zeta_{\kappa'}\right\|_{L^{\infty}(\Omega)}^{2}\lesssim\delta^{2\alpha}e^{-2\tau}.\label{tidal-term-I-2-zero-order-estimate}
		\end{align}
		Here we have used $(\ref{denominator-tidal-term-integral-bound})$, and the fact that $\chi$ and $\tilde{W}_{\kappa}$ are in $L^{\infty}(\Omega)$ for all $\kappa\in\{1,2,\dots,N\}$. To bound $\left\|\zeta_{\kappa'}\right\|_{L^{\infty}(\Omega)}$ we have
		\begin{align}
		\left\|\zeta_{\kappa'}\right\|_{L^{\infty}(\Omega)}\lesssim1+\left\|\theta_{\kappa'}\right\|_{L^{\infty}(\Omega)}+\left\|\mu_{\kappa'}\right\|_{L^{\infty}(\Omega)}.\label{tidal-term-I-2-zero-order-estimate-2}
		\end{align}
		Then we use $\eqref{L-infinity-energy-space-bound-no-weights-statement-1}$ in Lemma $\ref{L-infinity-energy-space-bound}$ for the $\theta_{\kappa'}$ term, and $\eqref{mu-L-infinity-estimate}$ for the $\mu_{\kappa'}$ term to obtain
		\begin{align}
		\left\|\zeta_{\kappa'}\right\|_{L^{\infty}(\Omega)}\lesssim1+\left\|\theta_{\kappa'}\right\|_{\X{M}_{\kappa'}}+\left\|\nabla\mu_{\kappa'}\right\|_{\X{M}_{\kappa'}}\lesssim1+\sqrt{S_{M}(\tau)}\label{zeta-L-infinity-bound}
		\end{align}
		for all $\kappa'\in\{1,2,\dots,N\}$. Bounds $\eqref{tidal-term-I-2-zero-order-estimate-2}$ and $\eqref{zeta-L-infinity-bound}$ along with the a priori assumption $\eqref{a-priori-bound-cumulative-energy}$ give the final bound in $\eqref{tidal-term-I-2-zero-order-estimate}$.
		
		Now we concentrate on the case where $m+|\underline{n}|\geq1$. First we write $\Ndell{m}{n}\left(\left|\zeta_{\kappa}(x)-\zeta_{\kappa'}(\tau.z)\right|^{-3}\right)$ as
		\begin{align}
		\sum_{p=1}^{m+|\underline{n}|}\sum_{\substack{a+|\underline{b}|=m+|\underline{n}|-p \\ c+\sum_{lq} d_{lq}=p \\ a+c=m \\ |\underline{b}|+\sum_{lq}d_{lq}=|\underline{n}|}}\frac{\Ndell{a}{b}\left(\left(\rdell\zeta_{\kappa}\cdot\left(\zeta_{\kappa}-\zeta_{\kappa'}\right)\right)^{c}\left(\angdell_{12}\zeta_{\kappa}\cdot\left(\zeta_{\kappa}-\zeta_{\kappa'}\right)\right)^{d_{12}}\left(\angdell_{23}\zeta_{\kappa}\cdot\left(\zeta_{\kappa}-\zeta_{\kappa'}\right)\right)^{d_{23}}\left(\angdell_{13}\zeta_{\kappa}\cdot\left(\zeta_{\kappa}-\zeta_{\kappa'}\right)\right)^{d_{13}}\right)}{\left|\zeta_{\kappa}-\zeta_{\kappa'}\right|^{3+2p}},\label{tidal-term-I-2-denominator-derivative-expansion-1}
		\end{align}
		so to bound $\Ndell{m}{n}(\left|\zeta_{\kappa}(x)-\zeta_{\kappa'}(z)\right|^{-3})$ effectively, it is enough to bound each of the terms in the sum separately, for every valid choice of $p,a,\underline{b},c,d_{12},d_{23}$, and $d_{13}$. Strictly speaking the correct expansion would include varying constants in front of every term in $(\ref{tidal-term-I-2-denominator-derivative-expansion-1})$ depending on each index being summed over. They have all been set to $1$ as they are not important when finding sufficient bounds for each term. We can write every separate term in $(\ref{tidal-term-I-2-denominator-derivative-expansion-1})$ as
		\begin{align}
		\sum_{\sum_{i}(a_{i}+|\underline{b}_{i}|+\sum_{lq}(a^{(lq)}_{i}+|\underline{b}^{(lq)}_{i}|))=a+|\underline{b}|}\frac{1}{\left|\zeta_{\kappa}-\zeta_{\kappa'}\right|^{3+2p}}\prod_{i=1}^{c}\rdell^{a_{i}}\angdell^{\underline{b}_{i}}\left(\rdell\zeta_{\kappa}\cdot\left(\zeta_{\kappa}-\zeta_{\kappa'}\right)\right)\prod_{\substack{lq\in\{12,23,13\}\\1\leq i\leq d_{lq}}}\rdell^{a_{i}^{(lq)}}\angdell^{\underline{b}^{(lq)}_{i}}\left(\angdell_{lq}\zeta_{\kappa}\cdot\left(\zeta_{\kappa}-\zeta_{\kappa'}\right)\right),\label{tidal-term-I-2-denominator-derivative-expansion-2}
		\end{align}
		where, once again, constants in front of each term in the sum have been set to $1$. For any of the $c, d_{12}, d_{23}$ or $d_{13}$ equal to $0$, the product above is an empty product, equal to $1$.
		
		Finally we can expand the derivatives above to see that $(\ref{tidal-term-I-2-denominator-derivative-expansion-2})$ can be written as a linear combination of terms of the form
		\begin{align}
		\left|\zeta_{\kappa}-\zeta_{\kappa'}\right|^{-(3+2p)}\prod_{i=1}^{c}\left(\rdell^{a_{i_{1}}+1}\angdell^{\underline{b}_{i_{1}}}\zeta_{\kappa}\cdot\rdell^{a_{i_{2}}}\angdell^{\underline{b}_{i_{2}}}\left(\zeta_{\kappa}-\zeta_{\kappa'}\right)\right)\prod_{lq\in\{12,23,13\}}\prod_{i=1}^{d_{lq}}\left(\rdell^{a_{i_{1}}^{(lq)}}\angdell^{\underline{b}^{(lq)}_{i_{1}}}\angdell_{lq}\zeta_{\kappa}\cdot\rdell^{a_{i_{2}}^{(lq)}}\angdell^{\underline{b}^{(lq)}_{i_{2}}}\left(\zeta_{\kappa}-\zeta_{\kappa'}\right)\right),\label{tidel-term-I-2-denominator-derivative-expansion-3}
		\end{align}
		subject to the condition that
		\begin{align}
		a_{i_{1}}+|\underline{b}_{i_{1}}|+a_{i_{2}}+|\underline{b}_{i_{2}}|&=a_{i}+|\underline{b}_{i}|,\label{tidal-term-I-2-denominator-derivative-expansion-index-condition-1}\\
		a_{i_{1}}^{(lq)}+|\underline{b}_{i_{1}}^{(lq)}|+a_{i_{2}}^{(lq)}+|\underline{b}_{i_{2}}^{(lq)}|&=a_{i}^{(lq)}+|\underline{b}_{i}^{(lq)}|,\ \ lq\in\{12,23,13\}.\label{tidal-term-I-2-denominator-derivative-expansion-index-condition-2}
		\end{align}
		Thus we see that
		\begin{align}
		\sum_{i=1}^{c}\left(a_{i_{1}}+|\underline{b}_{i_{1}}|+a_{i_{2}}+|\underline{b}_{i_{2}}|\right)+\sum_{lq\in\{12,23,13\}}\sum_{i=1}^{d_{lq}}\left(a_{i_{1}}^{(lq)}+|\underline{b}_{i_{1}}^{(lq)}|+a_{i_{2}}^{(lq)}+|\underline{b}_{i_{2}}^{(lq)}|\right)=a+|\underline{b}|=m+|\underline{n}|-p.\label{tidal-term-I-2-denominator-derivative-expansion-index-condition-3}
		\end{align}
		Then note that we have
		\begin{align}
		&\left|\int_{\Omega}\frac{\zeta_{\kappa'}\tilde{W}_{\kappa'}^{\alpha}}{\left|\zeta_{\kappa}-\zeta_{\kappa'}\right|^{3+2p}}\prod_{i=1}^{c}\left(\rdell^{a_{i_{1}}+1}\angdell^{\underline{b}_{i_{1}}}\zeta_{\kappa}\cdot\rdell^{a_{i_{2}}}\angdell^{\underline{b}_{i_{2}}}\left(\zeta_{\kappa}-\zeta_{\kappa'}\right)\right)\prod_{lq\in\{12,23,13\}}\prod_{i=1}^{d_{lq}}\left(\rdell^{a_{i_{1}}^{(lq)}}\angdell^{\underline{b}^{(lq)}_{i_{1}}}\angdell_{lq}\zeta_{\kappa}\cdot\rdell^{a_{i_{2}}^{(lq)}}\angdell^{\underline{b}^{(lq)}_{i_{2}}}\left(\zeta_{\kappa}-\zeta_{\kappa'}\right)\right)dz\right|^{2}\nonumber\\
		&\lesssim\int_{\Omega}\tilde{W}_{\kappa'}^{2\alpha}\left|\rdell^{a_{1}}\angdell^{\underline{b}_{1}}\zeta_{\kappa}\right|^{2}\dots\left|\rdell^{a_{p_{1}}}\angdell^{\underline{b}_{p_{1}}}\zeta_{\kappa}\right|^{2}\left|\zeta_{\kappa}-\zeta_{\kappa'}\right|^{2(2p-p_{1})}\left|\zeta_{\kappa'}\right|^{2}dz\lesssim\left|\rdell^{a_{1}}\angdell^{\underline{b}_{1}}\zeta_{\kappa}\right|^{2}\dots\left|\rdell^{a_{p_{1}}}\angdell^{\underline{b}_{p_{1}}}\zeta_{\kappa}\right|^{2}.\label{tidal-term-I-2-estimate-1}
		\end{align}
		For the first inequality, we first use $(\ref{denominator-tidal-term-integral-bound})$ to bound $|\zeta_{\kappa}-\zeta_{\kappa'}|^{-(3+2p)}$ in $L^{\infty}(\Omega)$. We then apply $L^{2}(\Omega)$ Cauchy-Schwartz to the $z$ integral over $\Omega$, and apply $\mathbb{R}^{3}$ Cauchy-Schwartz to obtain
		\begin{align}
		\left|\Ndell{e}{f}\zeta_{\kappa}\cdot\Ndell{g}{h}\left(\zeta_{\kappa}-\zeta_{\kappa'}\right)\right|\leq\left|\Ndell{e}{f}\zeta_{\kappa}\right|\left|\Ndell{g}{h}\left(\zeta_{\kappa}-\zeta_{\kappa'}\right)\right|\label{tidal-term-I-2-inner-product-terms}
		\end{align}
		To see that we can write the upper bound as on the second line in $(\ref{tidal-term-I-2-estimate-1})$, notice that if $g+|\underline{h}|\geq1$, then $\Ndell{g}{h}\left(\zeta_{\kappa}-\zeta_{\kappa'}\right)=\Ndell{g}{h}\zeta_{\kappa}$. Thus once we bound all terms $\Ndell{e}{f}\zeta_{\kappa}\cdot\Ndell{g}{h}\left(\zeta_{\kappa}-\zeta_{\kappa'}\right)$ like $(\ref{tidal-term-I-2-inner-product-terms})$, from the resulting upper bound we can collect the $p_{1}$ terms that look like $\left|\Ndell{g}{h}\zeta_{\kappa}\right|$, and the $2p-p_{1}$ terms that look like $|\zeta_{\kappa}-\zeta_{\kappa'}|$.
		
		The second inequality in $(\ref{tidal-term-I-2-estimate-1})$ is because $\left|\rdell^{a_{1}}\angdell^{\underline{b}_{1}}\zeta_{\kappa}\right|^{2}\dots\left|\rdell^{a_{p_{1}}}\angdell^{\underline{b}_{p_{1}}}\zeta_{\kappa}\right|^{2}$ has no $z$ dependence and can be taken outside of the integral. The remaining integrand is $\tilde{W}_{\kappa'}^{2\alpha}\left|\zeta_{\kappa}-\zeta_{\kappa'}\right|^{2(2p-p_{1})}\left|\zeta_{\kappa'}\right|^{2}$ which can be bounded in $L^{\infty}(\Omega)$, using the assumptions on $\tilde{W}_{\kappa'}$ in Definition $\ref{definition-of-W}$, as well as $\eqref{mu-L-infinity-estimate}$, $\eqref{zeta-L-infinity-bound}$, and the a priori assumption $(\ref{a-priori-bound-cumulative-energy})$. Therefore we look to bound
		\begin{align}
		\delta^{2\alpha}e^{-2\tau}\int_{\Omega}\chi\tilde{W}^{\alpha+m}_{\kappa}\left|\rdell^{a_{1}}\angdell^{\underline{b}_{1}}\zeta_{\kappa}\right|^{2}\dots\left|\rdell^{a_{p_{1}}}\angdell^{\underline{b}_{p_{1}}}\zeta_{\kappa}\right|^{2}dx.\label{tidal-term-estimate-I-2-estimate-2}
		\end{align}
		We can assume without loss of generality that $1\leq a_{1}+|\underline{b}_{1}|\leq\dots\leq a_{p_{1}}+|\underline{b}_{p_{1}}|\leq m+|\underline{n}|$. We write
		\begin{align}
		&\delta^{2\alpha}e^{-2\tau}\int_{\Omega}\chi\tilde{W}^{\alpha+m}_{\kappa}\left|\rdell^{a_{1}}\angdell^{\underline{b}_{1}}\zeta_{\kappa}\right|^{2}\dots\left|\rdell^{a_{p_{1}}}\angdell^{\underline{b}_{p_{1}}}\zeta_{\kappa}\right|^{2}dx\nonumber\\
		&=\delta^{2\alpha}e^{-2\tau}\int_{\Omega}\chi\tilde{W}^{m-\sum a_{i}}_{\kappa}\tilde{W}_{\kappa}^{a_{1}}\left|\rdell^{a_{1}}\angdell^{\underline{b}_{1}}\zeta_{\kappa}\right|^{2}\dots\tilde{W}_{\kappa}^{\alpha+a_{p_{1}}}\left|\rdell^{a_{p_{1}}}\angdell^{\underline{b}_{p_{1}}}\zeta_{\kappa}\right|^{2}dx.\label{tidal-term-estimate-I-2-estimate-3}
		\end{align}
		For $i=1,\dots,p_{1}-1$, we have $a_{i}+|\underline{b}_{i}|\leq a_{p_{1}}+|\underline{b}_{p_{1}}|$, and therefore $a_{i}+|\underline{b}_{i}|\leq (m+|\underline{n}|)/2\leq M/2$. Thus, we can apply $\eqref{L-infinity-energy-space-bound-weights-statement-1}$ from Lemma $\ref{L-infinity-energy-space-bound}$ to $\tilde{W}_{\kappa}^{a_{i}}\left|\rdell^{a_{i}}\angdell^{\underline{b}_{i}}\zeta_{\kappa}\right|^{2}$ for $i=1,\dots,p_{1}-1$, and bound these terms in $L^{\infty}(\Omega)$. We obtain
		\begin{align}
		&\delta^{2\alpha}e^{-2\tau}\int_{\Omega}\chi\tilde{W}^{m-\sum a_{i}}_{\kappa}\tilde{W}_{\kappa}^{a_{1}}\left|\rdell^{a_{1}}\angdell^{\underline{b}_{1}}\zeta_{\kappa}\right|^{2}\dots\tilde{W}_{\kappa}^{\alpha+a_{p_{1}}}\left|\rdell^{a_{p_{1}}}\angdell^{\underline{b}_{p_{1}}}\zeta_{\kappa}\right|^{2}dx\nonumber\\
		&\lesssim\delta^{2\alpha}e^{-2\tau}(1+S_{M}(\tau))^{p_{1}-1}\int_{\Omega}\chi\tilde{W}^{m-\sum a_{i}}_{\kappa}\tilde{W}_{\kappa}^{\alpha+a_{p_{1}}}\left|\rdell^{a_{p_{1}}}\angdell^{\underline{b}_{p_{1}}}\theta_{\kappa}\right|^{2}dx\nonumber\\
		&+\delta^{2\alpha}e^{-2\tau}(1+S_{M}(\tau))^{p_{1}-1}\int_{\Omega}\chi\tilde{W}^{m-\sum a_{i}}_{\kappa}\tilde{W}_{\kappa}^{\alpha+a_{p_{1}}}\left|\rdell^{a_{p_{1}}}\angdell^{\underline{b}_{p_{1}}}(x+e^{-\tau}\bar{x}_{\kappa}+(1-e^{-\tau})\mu_{\kappa}(x))\right|^{2}dx.\label{tidal-term-estimate-I-2-estimate-4}
		\end{align}
		The $\theta_{\kappa}$ term on the right hand side can be bounded by the cumulative energy function $S_{M}$. For the remaining term, we use $\eqref{tidal-term-I-1-estimate-2}$ to bound the $\mu_{\kappa}$ term, and note that we can bound $\tilde{W}_{\kappa}^{m-\sum a_{i}}$ in $L^{\infty}(\Omega)$, as $\sum a_{i}\leq m$. Thus we obtain
		\begin{align}
		\delta^{2\alpha}e^{-2\tau}\int_{\Omega}\chi\tilde{W}^{m-\sum a_{i}}_{\kappa}\tilde{W}_{\kappa}^{a_{1}}\left|\rdell^{a_{1}}\angdell^{\underline{b}_{1}}\zeta_{\kappa}\right|^{2}\dots\tilde{W}_{\kappa}^{\alpha+a_{p_{1}}}\left|\rdell^{a_{p_{1}}}\angdell^{\underline{b}_{p_{1}}}\zeta_{\kappa}\right|^{2}dx\lesssim\delta^{2\alpha}e^{-2\tau}\left(1+S_{M}(\tau)\right)^{p_{1}}\lesssim\delta^{2\alpha}e^{-2\tau},\label{tidal-term-estimate-I-2-estimate-5}
		\end{align}	
		where the last bound is due to the a priori assumption $\eqref{a-priori-bound-cumulative-energy}$.
				
		The bounds $(\ref{tidal-term-I-2-zero-order-estimate})$ and $(\ref{tidal-term-estimate-I-2-estimate-5})$ cover all possible cases for $\mathcal{I}_{2}$, thus we have
		\begin{align}
		\int_{\Omega}\chi\tilde{W}_{\kappa}^{\alpha+m}\left|\mathcal{I}_{2}\right|^{2}dx\lesssim\delta^{2\alpha}e^{-2\tau},\label{tidal-term-estimate-I-2-estimate-6}
		\end{align}
		as required.\\
		
		\noindent\textbf{Bound for $\mathcal{I}_{3}$.}\\
		
		\noindent Similarly to $\mathcal{I}_{2}$, to bound any of the terms in the sum that forms $\mathcal{I}_{3}$ it is sufficient to bound all terms of the form
		\begin{align}
		\delta^{2\alpha}e^{-2\tau}\int_{\Omega}\chi\tilde{W}^{\alpha+m}_{\kappa}\left|\rdell^{a_{1}}\angdell^{\underline{b}_{1}}\zeta_{\kappa}\right|^{2}\dots\left|\rdell^{a_{p}}\angdell^{\underline{b}_{p}}\zeta_{\kappa}\right|^{2}\left|\rdell^{a}\angdell^{\underline{b}}\zeta_{\kappa}\right|^{2}dx,\label{tidal-term-estimate-I-3-estimate-1}
		\end{align}
		subject to the condition that $a+\sum a_{i}=m$, and $|\underline{b}|+\sum|\underline{b}_{i}|=|\underline{n}|$. This requires exactly the same strategy as $(\ref{tidal-term-estimate-I-2-estimate-4})$, and so we immediately obtain
		\begin{align}
		\int_{\Omega}\chi\tilde{W}_{\kappa}^{\alpha+m}\left|\mathcal{I}_{3}\right|^{2}dx\lesssim\delta^{2\alpha}e^{-2\tau}.\label{tidal-term-estimate-I-3-estimate-2}
		\end{align}
		The bounds in $(\ref{tidal-term-I-1-estimate-1})$, $(\ref{tidal-term-estimate-I-2-estimate-6})$, and $(\ref{tidal-term-estimate-I-3-estimate-2})$, along with analogous estimates on $\supp{\bar{\chi}}$ give the proposition.
	\end{proof}
		Now we move on to estimates for the gravitational potentials acting on each separate body coming from their own mass.	
	
\subsection{Self Interaction Term Estimates}
In this section we estimate the term coming from the gravitational effect of a body in this system on itself, $\mathscr{G}_{\kappa}$. Let us recall that we can write 
		\begin{align}
		\mathscr{G}_{\kappa}^{i}(\tau,x)=\delta^{\alpha}e^{-\tau}\int_{\Omega}\frac{\dell_{k}\left(\tilde{W}_{\kappa}^{\alpha}\ninv[\kappa]^{k}_{i}\right)}{\left|\zeta_{\kappa}(\tau,x)-\zeta_{\kappa}(\tau,z)\right|}dz.\label{self-interaction-term-greens-function}
		\end{align}
		The key estimate for $\mathscr{G}_{\kappa}$ is laid out in the following proposition.
		
		\begin{proposition}\label{self-interaction-term-estimates-proposition}
			Assume $\gamma=1+\frac{1}{n}$ for some $n\in\mathbb{Z}_{\geq2}$, or $\gamma\in(1,14/13)$. Let $\{(\dell_{\tau}\theta_{\kappa},\theta_{\kappa})|\kappa=1,2,\dots,N\}$ be the solution to $(\ref{delta-rescaled-lagrangian-euler-1})-(\ref{delta-rescaled-lagrangian-euler-3})$ on $[0,T]$, in the sense of Theorem \ref{local-well-posedness-theorem}, for some $M\geq2\ceil*{\alpha}+12$. Suppose the assumptions $\eqref{a-priori-smallness-mu-norm}$--$\eqref{a-priori-sum-of-theta-L-infinity-norms}$ hold. Assume that $\nabla\tilde{W}_{\kappa}\in\X{M}_{\kappa}$ for $\kappa=1,2,\dots,N$. Then for all $\tau\in[0,T]$, we have:
			\begin{align}
			\sum_{\kappa=1}^{N}\left(\sum_{m+|\underline{n}|=0}^{M}\int_{\Omega}\chi\tilde{W}_{\kappa}^{\alpha+m}\left|\Ndell{m}{n}\mathscr{G}_{\kappa}\right|^{2}dx+\sum_{|\underline{k}|=0}^{M}\int_{\Omega}\bar{\chi}\tilde{W}_{\kappa}^{\alpha}\left|\nabla^{\underline{k}}\mathscr{G}_{\kappa}\right|^{2}dx\right)\lesssim \delta^{2\alpha}e^{-2\tau}.\label{self-interaction-term-estimates-statement}
			\end{align}
		\end{proposition}
		We begin with estimates for the integrals near the boundary, on $\supp{\chi}$. The strategy follows the strategy in~\cite{HaJa4}. We first bound tangential derivatives of $\mathscr{G}$, and use these, as well as the Poisson equation  $\Delta\phi=\rho$, to close estimates in the normal direction. The esimtates for tangential derivatives of $\mathscr{G}_{\kappa}$ near the boundary are stated in the following proposition:		
		\begin{proposition}[Tangential self-interaction estimates]\label{self-interaction-term-tangential-estimates-proposition}	Assume $\gamma=1+\frac{1}{n}$ for some $n\in\mathbb{Z}_{\geq2}$, or $\gamma\in(1,14/13)$. Let $\{(\dell_{\tau}\theta_{\kappa},\theta_{\kappa})|\kappa=1,2,\dots,N\}$ be the solution to $(\ref{delta-rescaled-lagrangian-euler-1})-(\ref{delta-rescaled-lagrangian-euler-3})$ on $[0,T]$, in the sense of Theorem \ref{local-well-posedness-theorem}, for some $M\geq2\ceil*{\alpha}+12$. Suppose the assumptions $\eqref{a-priori-smallness-mu-norm}$--$\eqref{a-priori-sum-of-theta-L-infinity-norms}$ hold. Assume that $\nabla\tilde{W}_{\kappa}\in\X{M}_{\kappa}$ for $\kappa=1,2,\dots,N$. Then for all $\tau\in[0,T]$, we have:
			\begin{align}
			\sum_{\kappa=1}^{N}\sum_{|\underline{n}|=0}^{M}\int_{\Omega}\chi\tilde{W}_{\kappa}^{\alpha}\left|\angdell^{\underline{n}}\mathscr{G}_{\kappa}\right|^{2}dx\lesssim \delta^{2\alpha}e^{-2\tau}.\label{self-interaction-term-tangential-estimates-statement}
			\end{align}
		\end{proposition}
	
	\begin{proof}[\textbf{Proof}]
		The proof relies on the following lemma, which was first used by Gu and Lei \cite{GuLei} in the context of the spatial domain being $\mathbb{T}^{2}\times\mathbb{R}$ rather than the closed unit ball. It was then adapted in to our setting by Had\v{z}i\'{c} and Jang \cite{HaJa4}. For a full proof of the proposition, see \cite{HaJa4}, whose methods apply directly to our case.
	\end{proof}
	
	\begin{lemma}\label{repeated-angular-derivatives-integration-by-parts}
		Let $\angdell_{x}$ and $\angdell_{z}$ denote angular derivatives in the $x$ and $z$ variables respetively. Let $h_{1}:\mathbb{R}^{3}\rightarrow\mathbb{R}$, and $h_{2}:\mathbb{R}^{3}\times\mathbb{R}^{3}\rightarrow\mathbb{R}$. Let $\underline{\nu}\in\mathbb{Z}_{\geq0}^{3}$. Then
		\begin{align}
		\angdell_{x}^{\underline{\nu}}\int_{\Omega}h_{1}(z)h_{2}(x,z)dz=\sum_{|\underline{\nu}'|\leq|\underline{\nu}|}c_{\underline{\nu}'}\int_{\Omega}\angdell_{z}^{\underline{\nu}-\underline{\nu}'}h_{1}\left(\angdell_{x}+\angdell_{z}\right)^{\underline{\nu}'}h_{2}dz,
		\end{align}
		for some constants $c_{\underline{\nu}'}$.
	\end{lemma}
	
	\begin{proof}[\textbf{Proof}] The proof is a combination of the identity $\angdell_{x}=\left(\angdell_{x}+\angdell_{z}\right)-\angdell_{z}$, as well as integration by parts in the $z$ variable.
	\end{proof}
	We now use Proposition $\ref{self-interaction-term-tangential-estimates-proposition}$ to prove Proposition $\ref{self-interaction-term-estimates-proposition}$. Once again, the proof strategy very closely follows that of the one in \cite{HaJa4}.
	
	\begin{proof}[\textbf{Proof of Proposition $\ref{self-interaction-term-estimates-proposition}$}]
		Recall that we are looking to prove
		\begin{align}
		\sum_{\kappa=1}^{N}\left(\sum_{m+|\underline{n}|=0}^{M}\int_{\Omega}\chi\tilde{W}_{\kappa}^{\alpha+m}\left|\Ndell{m}{n}\mathscr{G}_{\kappa}\right|^{2}dx+\sum_{|\underline{k}|=0}^{M}\int_{\Omega}\bar{\chi}\tilde{W}_{\kappa}^{\alpha}\left|\nabla^{\underline{k}}\mathscr{G}_{\kappa}\right|^{2}dx\right)\lesssim \delta^{2\alpha}e^{-2\tau}.\label{self-interaction-term-estimates-1}
		\end{align}
		We first concentrate on the interior bounds, localised on $\supp{\bar{\chi}}$. Recall the definition of $\chi$ and $\bar{\chi}$ as smooth radial cutoff functions in $\eqref{cutoff-function}$ and $\eqref{cutoff-function-conjugate}$. We know that $\supp{\bar{\chi}}=B_{\frac{3}{4}}$, the ball of radius $3/4$ around $0$. Define another smooth radial function $\bar{\chi}_{\nu}$ such that
		\begin{align}
		\supp{\bar{\chi}_{\nu}}=B_{\frac{3}{4}+\nu}\label{cutoff-function-conjugate-epsilon}
		\end{align}
		such that $\frac{3}{4}+\nu\in(3/4,1)$. Correspondingly we define
		\begin{align}
		\chi_{\nu}=1-\bar{\chi}_{\nu}.\label{cutoff-function-epsilon}
		\end{align}
		To bound the integrals on $\supp{\bar{\chi}}$ we have
		\begin{align}
		\sum_{\kappa=1}^{N}\sum_{|\underline{k}|=0}^{M}\int_{\Omega}\bar{\chi}\tilde{W}_{\kappa}^{\alpha}\left|\nabla^{\underline{k}}\mathscr{G}_{\kappa}\right|^{2}dx&\lesssim\delta^{2\alpha}e^{-2\tau}\sum_{\kappa=1}^{N}\sum_{|\underline{k}|=0}^{M}\int_{\Omega}\bar{\chi}\tilde{W}_{\kappa}^{\alpha}\left|\nabla^{\underline{k}}\int_{\Omega}\frac{\chi_{\nu}\dell_{k}\left(\ninv[\kappa]^{k}_{*}\tilde{W}_{\kappa}^{\alpha}\right)}{\left|\zeta_{\kappa}(\tau,x)-\zeta_{\kappa}(\tau,z)\right|}dz\right|^{2}dx\nonumber\\
		&+\delta^{2\alpha}e^{-2\tau}\sum_{\kappa=1}^{N}\sum_{|\underline{k}|=0}^{M}\int_{\Omega}\bar{\chi}\tilde{W}_{\kappa}^{\alpha}\left|\nabla^{\underline{k}}\int_{\Omega}\frac{\bar{\chi}_{\nu}\dell_{k}\left(\ninv[\kappa]^{k}_{*}\tilde{W}_{\kappa}^{\alpha}\right)}{\left|\zeta_{\kappa}(\tau,x)-\zeta_{\kappa}(\tau,z)\right|}dz\right|^{2}dx.\label{self-interaction-term-interior-estimates-1}
		\end{align}
		For the first term on the right hand side above, we can use the methods we used in Proposition $\ref{tidal-term-estimates-proposition}$ as
		\begin{align*}
		0<\nu\leq|x-z|\lesssim	\left|\zeta_{\kappa}(\tau,x)-\zeta_{\kappa}(\tau,z)\right|
		\end{align*} 
		for $x\in\supp{\bar{\chi}}$, $z\in\supp{\chi_{\nu}}$, using the mean value theorem and $\eqref{a-priori-bound-grad-zeta}$. For the second term we can use the same methods as Proposition $\ref{self-interaction-term-tangential-estimates-proposition}$, noting that $\bar{\chi}_{\nu}$ and its derivatives are $0$ on $\dell\Omega$, so there are no boundary terms when we integrate by parts to get a obtain identity to Lemma $\ref{repeated-angular-derivatives-integration-by-parts}$. Therefore, we have
		\begin{align}
		\sum_{\kappa=1}^{N}\sum_{|\underline{k}|=0}^{M}\int_{\Omega}\bar{\chi}\tilde{W}_{\kappa}^{\alpha}\left|\nabla^{\underline{k}}\mathscr{G}_{\kappa}\right|^{2}dx\lesssim\delta^{2\alpha}e^{-2\tau}.\label{self-interaction-term-interior-estimates-2}
		\end{align}
		Now we concentrate on the $\supp{\chi}$ bounds. As in~\cite{HaJa4}, we use induction on the number of radial derivatives. For now we fix a $\kappa$. Our induction hypothesis is
		\begin{align}
		\sum_{0\leq |\underline{n}|\leq M-a}\int_{\Omega}\chi\tilde{W}_{\kappa}^{\alpha+a}\left|\Ndell{m}{n}\mathscr{G}_{\kappa}\right|^{2}dx\leq C(a)\left(\delta^{2\alpha}e^{-2\tau}+(\varepsilon_{1}+\varepsilon_{2})\left\|\mathscr{G}_{\kappa}\right\|_{\X{M}_{\kappa}}^{2}\right),\label{self-interaction-term-estimates-2}
		\end{align}
		where $C(a)$ is a constant depending only on the number of radial derivatives and our a priori assumptions $\eqref{a-priori-smallness-mu-norm}$--$\eqref{a-priori-bound-jacobian}$, and $\varepsilon_{1},\varepsilon_{2}$ are the constants mentioned in the a priori assumptions, Section $\ref{assumptions-section}$. The case where $m=0$ has been dealt with in Proposition $\ref{self-interaction-term-tangential-estimates-proposition}$. Now assume it holds for some $0<a<M$. Following \cite{HaJa4}, we use the Poisson equation $\eqref{euler-4}$ to write $\rdell\mathscr{G}_{\kappa}$ in a form amenable to the necessary estimates.
		
		Recall that $\eqref{euler-4}$ gives us the relation $\Delta\phi=4\pi\rho$, and we know from $\eqref{definition-cumulative-density}$ that $\rho$ is the cumulative density defined by $\rho=\sum_{\kappa}\rho_{\kappa}$, so, where $\rho,\{\rho_{\kappa}\}$ are thought of as functions on $\mathbb{R}^{3}$ with compact support, we can define $\phi_{\kappa}$ to be the solution to
		\begin{align}
		\Delta\phi_{\kappa}=4\pi\rho_{\kappa}\label{kappa-separated-elliptic-equation}
		\end{align}
		on $\mathbb{R}^{3}$, with $\phi_{\kappa}\rightarrow0$ as $|x|\rightarrow\infty$ as the boundary condition. We know that $\phi_{\kappa}=\Phi*\rho_{\kappa}$, where $\Phi$ is the fundamental solution, so in Lagrangian coordinates we have
		\begin{align}
		\ninv[\kappa]^{k}_{i}\dell_{k}\phi_{\kappa}(t,\eta_{\kappa}(t,x))=-\mathscr{G}_{\kappa}^{i},
		\end{align} 
		where we recall the definition of $\mathscr{G}_{\kappa}$ from $\eqref{lagrangian-gradient-of-lagrangian-potential}$. We take the Lagrangian divergence of this relationship, and comparing to the pullback via $\eta_{\kappa}$ of $\eqref{kappa-separated-elliptic-equation}$, obtain
		\begin{align}
		\ndiv\mathscr{G}_{\kappa}=\ninv[\kappa]^{j}_{i}\dell_{j}\mathscr{G}_{\kappa}^{i}=4\pi\rho_{\kappa}(t,\eta_{\kappa}(t,x)).\label{self-interaction-term-estimates-3}
		\end{align}
		Then, using $\eqref{lagrangian-density-relation}, \eqref{definition-of-translated-w}, \eqref{eta-ansatz}$, and $\eqref{rescaled-change-of-variables-jacobian}$, we have
		\begin{align}
		\ndiv\mathscr{G}_{\kappa}=4\pi\delta^{\alpha}e^{-3\tau}\tilde{W}_{\kappa}^{\alpha}\njac_{\kappa}^{-1}.\label{self-interaction-term-estimates-4}
		\end{align}
		Moreover, since $\mathscr{G}_{\kappa}$ can be written as the Lagrangian gradient of some potential function, we immediately have
		\begin{align}
		\nCurl\mathscr{G}_{\kappa}=0.\label{self-interaction-term-estimates-5}
		\end{align}
		Now we define $\mathscr{R}_{\kappa}$:
		\begin{align}
		\mathscr{R}_{\kappa}(\tau,x)=e^{-\tau}\bar{x}_{\kappa}+(1-e^{-\tau})\mu_{\kappa}(x)+\theta_{\kappa}(\tau,x)=\zeta_{\kappa}(\tau,x)-x.\label{definition-remainder-of-zeta}
		\end{align}
		With this definition of $\mathscr{R}_{\kappa}$ we have
		\begin{align}
		\dive\mathscr{G}_{\kappa}&=\delta^{\alpha}e^{-3\tau}\tilde{W}_{\kappa}^{\alpha}\njac_{\kappa}^{-1}+\ninv[\kappa]^{k}_{j}\dell_{i}\mathscr{R}_{\kappa}^{j}\dell_{k}\mathscr{G}_{\kappa}^{i},\label{self-interaction-term-estimates-6}\\
		[\Curl\mathscr{G}_{\kappa}]^{i}_{j}&=\ninv[\kappa]^{k}_{s}\dell_{j}\mathscr{R}_{\kappa}^{s}\dell_{k}\mathscr{G}_{\kappa}^{i}-\ninv[\kappa]^{k}_{s}\dell_{i}\mathscr{R}_{\kappa}^{s}\dell_{k}\mathscr{G}_{\kappa}^{j}.\label{self-interaction-term-estimates-7}
		\end{align}
		We also have
		\begin{align}
		y^{i}\dive \mathscr{G}_{\kappa}=\angdell_{ik}\mathscr{G}_{\kappa}^{k}+y^{k}\dell_{i}\mathscr{G}_{\kappa}^{k}=\angdell_{ik}\mathscr{G}_{\kappa}^{k}+y^{k}[\Curl\mathscr{G}_{\kappa}]^{k}_{i}+y^{k}\dell_{k}\mathscr{G}_{\kappa}^{i}=\angdell_{ik}\mathscr{G}_{\kappa}^{k}+y^{k}[\Curl\mathscr{G}_{\kappa}]^{k}_{i}+\rdell\mathscr{G}_{\kappa}^{i}.\label{self-interaction-term-estimates-8}
		\end{align}
		Thus
		\begin{align}
		\rdell\mathscr{G}_{\kappa}^{i}=y^{i}\dive\mathscr{G}_{\kappa}-y^{k}[\Curl\mathscr{G}_{\kappa}]^{k}_{i}-\angdell_{ik}\mathscr{G}_{\kappa}^{k}.\label{self-interaction-term-estimates-9}
		\end{align}
		Then, using $\eqref{self-interaction-term-estimates-9}$ and the commutativity of $\rdell$ and $\angdell$, we obtain
		\begin{align}
		&\sum_{0\leq|\underline{n}|\leq M-a-1}\int_{\Omega}\chi\tilde{W}_{\kappa}^{\alpha+a+1}\left|\Ndell{a+1}{n}\mathscr{G}_{\kappa}\right|^{2}dx= \sum_{0\leq|\underline{n}|\leq M-a-1}\int_{\Omega}\chi\tilde{W}_{\kappa}^{\alpha+a+1}\left|\Ndell{a}{n}\rdell\mathscr{G}_{\kappa}\right|^{2}dx\nonumber\\
		&\lesssim\sum_{\substack{0< b\leq a\\ 0\leq|\underline{c}|\leq M-a-1}}\int_{\Omega}\chi\tilde{W}_{\kappa}^{\alpha+a+1}\left(\left|\Ndell{b}{c}\dive\mathscr{G}_{\kappa}\right|^{2}+\left|\Ndell{b}{c}\Curl\mathscr{G}_{\kappa}\right|^{2}\right)dx+\sum_{0\leq|\underline{n}|\leq M-a}\int_{\Omega}\chi\tilde{W}_{\kappa}^{\alpha+a}\left|\Ndell{a}{n}\mathscr{G}_{\kappa}\right|^{2}dx\nonumber\\
		&\lesssim\sum_{\substack{0< b\leq a\\ 0\leq|\underline{c}|\leq M-a-1}}\int_{\Omega}\chi\tilde{W}_{\kappa}^{\alpha+a+1}\left(\left|\Ndell{b}{c}\dive\mathscr{G}_{\kappa}\right|^{2}+\left|\Ndell{b}{c}\Curl\mathscr{G}_{\kappa}\right|^{2}\right)dx+C(a)\left(\delta^{2\alpha}e^{-2\tau}+(\varepsilon_{1}+\varepsilon_{2})\left\|\mathscr{G}_{\kappa}\right\|_{\X{M}_{\kappa}}^{2}\right).\label{self-interaction-term-estimates-10}
		\end{align}
		The first bound is due to $\eqref{self-interaction-term-estimates-9}$ and the second is due to the induction hypothesis $\eqref{self-interaction-term-estimates-2}$. It is left to estimate the $\dive\mathscr{G}_{\kappa}$ and $\Curl\mathscr{G}_{\kappa}$ terms. For this we use $\eqref{self-interaction-term-estimates-6}$ and $\eqref{self-interaction-term-estimates-7}$. First we have
		\begin{align}
		&\sum_{\substack{0< b\leq a\\ 0\leq|\underline{c}|\leq M-a-1}}\int_{\Omega}\chi\tilde{W}_{\kappa}^{\alpha+a+1}\left|\Ndell{b}{c}\left(\delta^{\alpha}e^{-3\tau}\tilde{W}_{\kappa}^{\alpha}\njac_{\kappa}^{-1}\right)\right|^{2}dx\nonumber\\
		&\lesssim\delta^{2\alpha}e^{-6\tau}\sum_{\substack{0\leq a_{1}+|\underline{b}_{1}|\leq a_{2}+|\underline{b}_{2}|\\0< a_{1}+a_{2}\leq a\\ 0\leq |\underline{b}_{1}|+|\underline{b}_{2}|\leq M-a-1}}\left\|\tilde{W}_{\kappa}^{\frac{a_{1}}{2}}\rdell^{a_{1}}\angdell^{\underline{b}_{1}}\left(\tilde{W}_{\kappa}^{\alpha}\right)\right\|_{L^{\infty}(\Omega)}^{2}\int_{\Omega}\chi\tilde{W}_{\kappa}^{\alpha+a_{2}+1}\left|\rdell^{a_{2}}\angdell^{\underline{b}_{2}}\left(\njac_{\kappa}^{-1}\right)\right|^{2}dx\nonumber\\
		&+\delta^{2\alpha}e^{-6\tau}\sum_{\substack{0\leq a_{1}+|\underline{b}_{1}|\leq a_{2}+|\underline{b}_{2}|\\0< a_{1}+a_{2}\leq a\\ 0\leq |\underline{b}_{1}|+|\underline{b}_{2}|\leq M-a-1}}\left\|\tilde{W}_{\kappa}^{\frac{a_{1}}{2}}\rdell^{a_{1}}\angdell^{\underline{b}_{1}}\left(\njac_{\kappa}^{-1}\right)\right\|_{L^{\infty}(\Omega)}^{2}\int_{\Omega}\chi\tilde{W}_{\kappa}^{\alpha+a_{2}+1}\left|\rdell^{a_{2}}\angdell^{\underline{b}_{2}}\left(\tilde{W}_{\kappa}^{\alpha}\right)\right|^{2}dx\nonumber\\
		&\lesssim\delta^{2\alpha}e^{-6\tau}\left\|\nabla\tilde{W}_{\kappa}\right\|_{\X{M}_{\kappa}}^{2}\lesssim\delta^{2\alpha}e^{-2\tau}.\label{self-interaction-term-estimates-11}
		\end{align}
		If $a_{1}+|\underline{b}_{1}|=0$ then both $\tilde{W}^{\alpha}$, and $\njac_{\kappa}^{-1}$ can bounded in $L^{\infty}(\Omega)$, due to the assumption that $\nabla\tilde{W}_{\kappa}\in\X{M}_{\kappa}$ and the a priori assumption $\eqref{a-priori-bound-jacobian}$ respectively. Otherwise, we first use $\eqref{lower-order-estimates-statement-2}$ (resp. $\eqref{lower-order-estimates-statement-3}$) from Lemma $\eqref{lower-order-estimates-lemma}$ to bound the terms involving $L^{2}$ (resp. $L^{\infty}$) norms of derivatives of $\njac_{\kappa}^{-1}$. Then for the $\tilde{W}_{\kappa}^{\alpha}$ terms, we can use $\nabla\tilde{W}_{\kappa}\in\X{M}_{\kappa}$ (resp. $\eqref{L-infinity-energy-space-bound-weights-statement-2}$ in Lemma $\ref{L-infinity-energy-space-bound}$) to bound the $L^{2}$ (resp. $L^{\infty}$) terms after using the Leibniz rule on $\rdell^{a_{i}}\angdell^{\underline{b}_{i}}\left(\tilde{W}_{\kappa}^{\alpha}\right)$. Note here that our range of $\gamma$ is crucial in keeping the norms of the derivatives of $\tilde{W}_{\kappa}^{\alpha}$ finite. Next we have
		\begin{align}
		&\sum_{\substack{0< b\leq a\\ 0\leq|\underline{c}|\leq M-a-1}}\int_{\Omega}\chi\tilde{W}_{\kappa}^{\alpha+a+1}\left|\Ndell{b}{c}\left(\ninv[\kappa]\nabla\mathscr{R}_{\kappa}\nabla\mathscr{G}_{\kappa}\right)\right|^{2}dx\nonumber\\
		&\lesssim\sum_{\substack{0\leq a_{1}+|\underline{b}_{1}|\leq a_{2}+|\underline{b}_{2}|\leq a_{3}+|\underline{b}_{3}|\\0< a_{1}+a_{2}+a_{3}\leq a\\ 0\leq |\underline{b}_{1}|+|\underline{b}_{2}|+|\underline{b}_{3}|\leq M-a-1}}\left\|\tilde{W}_{\kappa}^{\frac{a_{1}}{2}}\rdell^{a_{1}}\angdell^{\underline{b}_{1}}\ninv[\kappa]\right\|_{L^{\infty}(\Omega)}^{2}\left\|\tilde{W}_{\kappa}^{\frac{a_{2}}{2}}\rdell^{a_{2}}\angdell^{\underline{b}_{2}}\left(\nabla\mathscr{R}_{\kappa}\right)\right\|_{L^{\infty}(\Omega)}^{2}\int_{\Omega}\chi\tilde{W}_{\kappa}^{\alpha+a_{3}+1}\left|\rdell^{a_{3}}\angdell^{\underline{b}_{3}}\left(\nabla\mathscr{G}_{\kappa}\right)\right|^{2}dx\nonumber\\
		&+\sum_{\substack{0\leq a_{1}+|\underline{b}_{1}|\leq a_{2}+|\underline{b}_{2}|\leq a_{3}+|\underline{b}_{3}|\\0< a_{1}+a_{2}+a_{3}\leq a\\ 0\leq |\underline{b}_{1}|+|\underline{b}_{2}|+|\underline{b}_{3}|\leq M-a-1}}\left\|\tilde{W}_{\kappa}^{\frac{a_{1}}{2}}\rdell^{a_{1}}\angdell^{\underline{b}_{1}}\left(\nabla\mathscr{R}_{\kappa}\right)\right\|_{L^{\infty}(\Omega)}^{2}\left\|\tilde{W}_{\kappa}^{\frac{a_{2}}{2}}\rdell^{a_{2}}\angdell^{\underline{b}_{2}}\left(\nabla\mathscr{G}_{\kappa}\right)\right\|_{L^{\infty}(\Omega)}^{2}\int_{\Omega}\chi\tilde{W}_{\kappa}^{\alpha+a_{3}+1}\left|\rdell^{a_{3}}\angdell^{\underline{b}_{3}}\ninv[\kappa]\right|^{2}dx\nonumber\\
		&+\sum_{\substack{0\leq a_{1}+|\underline{b}_{1}|\leq a_{2}+|\underline{b}_{2}|\leq a_{3}+|\underline{b}_{3}|\\0< a_{1}+a_{2}+a_{3}\leq a\\ 0\leq |\underline{b}_{1}|+|\underline{b}_{2}|+|\underline{b}_{3}|\leq M-a-1}}\left\|\tilde{W}_{\kappa}^{\frac{a_{1}}{2}}\rdell^{a_{1}}\angdell^{\underline{b}_{1}}\left(\nabla\mathscr{G}_{\kappa}\right)\right\|_{L^{\infty}(\Omega)}^{2}\left\|\tilde{W}_{\kappa}^{\frac{a_{2}}{2}}\rdell^{a_{2}}\angdell^{\underline{b}_{2}}\ninv[\kappa]\right\|_{L^{\infty}(\Omega)}^{2}\int_{\Omega}\chi\tilde{W}_{\kappa}^{\alpha+a_{3}+1}\left|\rdell^{a_{3}}\angdell^{\underline{b}_{3}}\left(\nabla\mathscr{R}_{\kappa}\right)\right|^{2}dx.\label{self-interaction-term-estimates-12}
		\end{align}
		The terms involving $L^{2}$ (resp. $L^{\infty}$) norms of derivatives of $\ninv[\kappa]$ can be bounded using $\eqref{lower-order-estimates-statement-2}$ (resp. $\eqref{lower-order-estimates-statement-3}$) in Lemma $\ref{lower-order-estimates-lemma}$. For the $\mathscr{G}_{\kappa}$ terms, we have for $i=1,2$:
		\begin{align}
		\int_{\Omega}\chi\tilde{W}_{\kappa}^{\alpha+a_{3}+1}\left|\rdell^{a_{3}}\angdell^{\underline{b}_{3}}\left(\nabla\mathscr{G}_{\kappa}\right)\right|^{2}dx&\lesssim\sum_{c+|\underline{d}|=1}^{a_{3}+|\underline{b}_{3}|+1}\int_{\Omega}\chi\tilde{W}_{\kappa}^{\alpha+c+1}\left|\rdell^{c}\angdell^{\underline{d}}\left(\mathscr{G}_{\kappa}\right)\right|^{2}dx\lesssim\left\|\mathscr{G}_{\kappa}\right\|_{\X{M}_{\kappa}}^{2},\label{self-interaction-term-estimates-12a}\\
		\left\|\tilde{W}_{\kappa}^{\frac{a_{i}}{2}}\rdell^{a_{i}}\angdell^{\underline{b}_{i}}\left(\nabla\mathscr{G}_{\kappa}\right)\right\|_{L^{\infty}(\Omega)}^{2}&\lesssim\sum_{c+|\underline{d}|=1}^{a_{i}+|\underline{b}_{i}|+1}\left\|\tilde{W}_{\kappa}^{\frac{c}{2}}\rdell^{c}\angdell^{\underline{d}}\left(\mathscr{G}_{\kappa}\right)\right\|_{L^{\infty}(\Omega)}^{2}\lesssim\left\|\mathscr{G}_{\kappa}\right\|_{\X{M}_{\kappa}}^{2}.\label{self-interaction-term-estimates-12b}
		\end{align}
		In $\eqref{self-interaction-term-estimates-12a}$ we use $\eqref{rectangular-as-ang-rad}$, and in $\eqref{self-interaction-term-estimates-12b}$ we use $\eqref{rectangular-as-ang-rad}$ and $\eqref{L-infinity-energy-space-bound-weights-statement-1}$ in Lemma $\ref{L-infinity-energy-space-bound}$. The $\mathscr{R}_{\kappa}$ terms can be bounded analogously to $\eqref{self-interaction-term-estimates-12a}$ and $\eqref{self-interaction-term-estimates-12b}$. Thus
		\begin{align}
		\sum_{\substack{0< b\leq a\\ 0\leq|\underline{c}|\leq M-a-1}}\int_{\Omega}\chi\tilde{W}_{\kappa}^{\alpha+a+1}\left|\Ndell{b}{c}\left(\ninv[\kappa]\nabla\mathscr{R}_{\kappa}\nabla\mathscr{G}_{\kappa}\right)\right|^{2}dx\lesssim\left(S_{M}(\tau)+\left\|\nabla\mu_{\kappa}\right\|_{\X{M}_{\kappa}}^{2}\right)\left\|\nabla\mathscr{G}_{\kappa}\right\|_{\X{M}_{\kappa}}^{2}\lesssim(\varepsilon_{1}+\varepsilon_{2})\left\|\mathscr{G}_{\kappa}\right\|_{\X{M}_{\kappa}}^{2}.\label{self-interaction-term-estimates-12c}
		\end{align}
		Then, from the identities for $\dive\mathscr{G}_{\kappa}$ and $\Curl\mathscr{G}_{\kappa}$ given in $\eqref{self-interaction-term-estimates-6}$ and $\eqref{self-interaction-term-estimates-7}$, as well as $\eqref{self-interaction-term-estimates-10}$--$\eqref{self-interaction-term-estimates-12c}$, we have
		\begin{align}
		\sum_{0\leq|\underline{n}|\leq M-a-1}\int_{\Omega}\chi\tilde{W}_{\kappa}^{\alpha+a+1}\left|\Ndell{a+1}{n}\mathscr{G}_{\kappa}\right|^{2}dx&=\sum_{0\leq|\underline{n}|\leq M-a-1}\int_{\Omega}\chi\tilde{W}_{\kappa}^{\alpha+a+1}\left|\Ndell{a}{n}\rdell\mathscr{G}_{\kappa}\right|^{2}dx\nonumber\\
		&\leq C(a+1)\left(\delta^{2\alpha}e^{-2\tau}+(\varepsilon_{1}+\varepsilon_{2})\left\|\mathscr{G}_{\kappa}\right\|_{\X{M}_{\kappa}}^{2}\right),\label{self-interaction-term-estimates-13}
		\end{align}
		for some constant $C(a+1)$. This proves the induction hypothesis.
		
		Clearly the sum all of the $C(a)$ can be bounded above by an absolute constant that we call $C_{1}$, and so summing these bounds for all $a$, we obtain
		\begin{align}
		\sum_{m+|\underline{n}|=0}^{M}\int_{\Omega}\chi\tilde{W}_{\kappa}^{\alpha+m}\left|\Ndell{m}{n}\mathscr{G}_{\kappa}\right|^{2}dx\leq C_{1}\left( \delta^{2\alpha}e^{-2\tau}+(\varepsilon_{1}+\varepsilon_{2})\left\|\mathscr{G}\right\|_{\X{M}_{\kappa}}^{2}\right).\label{self-interaction-term-estimates-14}
		\end{align}
		Combining $\eqref{self-interaction-term-estimates-14}$ and $\eqref{self-interaction-term-interior-estimates-2}$, we have, for some constant $C_{2}$, 
		\begin{align}
		\sum_{m+|\underline{n}|=0}^{M}\int_{\Omega}\chi\tilde{W}_{\kappa}^{\alpha+m}\left|\Ndell{m}{n}\mathscr{G}_{\kappa}\right|^{2}dx+\sum_{|\underline{k}|=0}^{M}\int_{\Omega}\bar{\chi}\tilde{W}_{\kappa}^{\alpha}\left|\nabla^{\underline{k}}\mathscr{G}_{\kappa}\right|^{2}dx\leq C_{2}\left( \delta^{2\alpha}e^{-2\tau}+(\varepsilon_{1}+\varepsilon_{2})\left\|\mathscr{G}\right\|_{\X{M}_{\kappa}}^{2}\right).\label{self-interaction-term-estimates-15}
		\end{align}
		Recall in Section $\ref{assumptions-section}$ we defined $\varepsilon_{1}$ and $\varepsilon_{2}$ as small constants that could be shrunk if necessary. Upon doing this, and recalling the definition of $\X{M}_{\kappa}$ norm from Definition $\ref{energy-function-spaces-definition}$, we have
		\begin{align}
		\sum_{m+|\underline{n}|=0}^{M}\int_{\Omega}\chi\tilde{W}_{\kappa}^{\alpha+m}\left|\Ndell{m}{n}\mathscr{G}_{\kappa}\right|^{2}dx+\sum_{|\underline{k}|=0}^{M}\int_{\Omega}\bar{\chi}\tilde{W}_{\kappa}^{\alpha}\left|\nabla^{\underline{k}}\mathscr{G}_{\kappa}\right|^{2}dx\lesssim\delta^{2\alpha}e^{-2\tau}.\label{self-interaction-term-estimates-16}
		\end{align}
		Summing over $\kappa$ for $\eqref{self-interaction-term-estimates-16}$ gives $\eqref{self-interaction-term-estimates-1}$, and completes the proof of Proposition $\ref{self-interaction-term-estimates-proposition}$.
	\end{proof}

	\subsection{Potential Term Estimates}\label{potential-term-estimates-section}
	
	We end this section with the main estimate needed to bound all potential terms when we perform our energy estimates.
	
	\begin{theorem}\label{potential-term-estimate-theorem}
			Assume $\gamma=1+\frac{1}{n}$ for some $n\in\mathbb{Z}_{\geq2}$, or $\gamma\in(1,14/13)$. Let $\{(\dell_{\tau}\theta_{\kappa},\theta_{\kappa})|\kappa=1,2,\dots,N\}$ be the solution to $(\ref{delta-rescaled-lagrangian-euler-1})-(\ref{delta-rescaled-lagrangian-euler-3})$ on $[0,T]$, in the sense of Theorem \ref{local-well-posedness-theorem}, for some $M\geq2\ceil*{\alpha}+12$. Suppose the assumptions $\eqref{a-priori-smallness-mu-norm}$--$\eqref{a-priori-sum-of-theta-L-infinity-norms}$ hold. Assume that $\nabla\tilde{W}_{\kappa}\in\X{M}_{\kappa}$ for $\kappa=1,2,\dots,N$. Then for all $[\tau_{1},\tau]\subset[0,T]$, we have:
			\begin{align}
			&\sum_{\kappa=1}^{N}\sum_{m+|\underline{n}|=0}^{M}\int_{\tau_{1}}^{\tau}\frac{1}{\delta}e^{\beta s}\int_{\Omega}\chi\tilde{W}_{\kappa}^{\alpha+m}\Ndell{m}{n}\left(\ninv[\kappa]^{j}_{i}\dell_{j}\psi_{\kappa}\right)\Ndell{m}{n}\dell_{\tau}\theta_{\kappa}^{i}\ dx\ ds\nonumber\\
			&+\sum_{\kappa=1}^{N}\sum_{|\underline{k}|=0}^{M}\int_{\tau_{1}}^{\tau}\frac{1}{\delta}e^{\beta s}\int_{\Omega}\bar{\chi}\tilde{W}_{\kappa}^{\alpha}\nabla^{\underline{k}}\left(\ninv[\kappa]^{j}_{i}\dell_{j}\psi_{\kappa}\right)\nabla^{\underline{k}}\dell_{\tau}\theta_{\kappa}^{i}\ dx\ ds\nonumber\\
			&\lesssim\delta^{\alpha-\frac{1}{2}}\int_{\tau_{1}}^{\tau}e^{(\frac{\beta}{2}-1)s}\sqrt{S_{M}(s)}\ ds.\label{potential-term-estimate-statement}
			\end{align}
	\end{theorem}

	\begin{proof}
		The proof is follows from the identity $\eqref{lagrangian-gradient-of-lagrangian-potential}$, and Propositions $\ref{tidal-term-estimates-proposition}$ and $\ref{self-interaction-term-estimates-proposition}$.
	\end{proof}

	\section{Curl Estimates}\label{curl-estimates-section}
	
	In this section we obtain sufficient bounds, for all $\kappa\in\{1,2,\dots,N\}$, of the quantities $\left\|\dell_{\tau}\theta_{\kappa}\right\|_{\Y{M}{\kappa}{\nCurl}}$ and $\left\|\theta_{\kappa}\right\|_{\Y{M}{\kappa}{\nCurl}}$, thereby giving us sufficient bounds for the cumulative $\nCurl$ energy function $C_{M}(\tau)$ defined in $\eqref{cumulative-curl-energy-function}$. Common to many works regarding Euler, or Euler-Poisson flows is the difficulty of controlling the $\nCurl$ norms mentioned above. This is due to the fact that they will appear in our main energy identity, at top order, with a bad sign. This means we need to find another way to control the $\nCurl$ terms. This is done by noticing extra structure in the equation $\eqref{rescaled-lagrangian-euler-1}$. First, note that we can rewrite this equation as
	
	\begin{align}
	\left(\dell_{\tau\tau}\theta_{\kappa}+\dell_{\tau}\theta_{\kappa}\right)+\ngrad\left((1+\alpha)e^{-\beta\tau}w_{\kappa}\pjac_{\kappa}+\psi_{\kappa}\right)=0,\ \ \kappa=1,2,\dots,N.\label{curl-equation}
	\end{align}
	From here, the key insight is to see that $\nCurl$ annihilates the $\ngrad$ term, leaving us with
	\begin{align}
	\nCurl\dell_{\tau\tau}\theta_{\kappa}+\nCurl\dell_{\tau}\theta_{\kappa}=0\ \ \kappa=1,2,\dots,N.
	\end{align}
	
	\begin{remark}\label{curl-equation-structure-remark}This underlying structure of the compressible Euler equations in Lagrangian variables has been used extensively in the study of the system, see $\cite{CouLindShk,CoSh2012,JaMa2015,HaJa3}$ for example.
	\end{remark}
	This is structurally very similar to the $\nCurl$ equation obtained in~\cite{PaHaJa}, and indeed the subsequent estimates follow an analogous strategy. One important difference is that in our case, derivatives of $\ninv[\kappa]$ and $\njac_{\kappa}$ produce derivatives of $\mu_{\kappa}$ as well as derivatives of $\theta_{\kappa}$. However, these terms can be dealt with using the assumption on $\nabla\mu_{\kappa}$, $\eqref{a-priori-smallness-mu-norm}$. For example, we can bound terms like so:
	\begin{align}
	\left|\int_{\Omega}\chi\tilde{W}_{\kappa}^{1+\alpha+m}\Ndell{m}{n}\mu_{\kappa}^{i}\Ndell{m}{n}\dell_{\tau}\theta_{\kappa}^{i}dx\right|^{2}\lesssim\delta e^{-\beta\tau}\left\|\mu_{\kappa}\right\|_{\X{M}_{\kappa}}^{2}S_{M}(\tau)\lesssim\delta e^{-\beta\tau}S_{M}(\tau),
	\end{align}
	where we use $\eqref{a-priori-smallness-mu-norm}$ to bound the $\mu_{\kappa}$ norm above by a constant. Hence, these extra terms are readily bounded without any further techincal difficulties. Therefore, using the methods in~\cite{PaHaJa}, we obtain the following theorem.
	\begin{theorem}[Higher Order Curl Estimates]\label{curl-estimates-theorem}Assume $\gamma=1+\frac{1}{n}$ for some $n\in\mathbb{Z}_{\geq2}$, or $\gamma\in(1,14/13)$. Let $\{(\dell_{\tau}\theta_{\kappa},\theta_{\kappa})|\kappa=1,2,\dots,N\}$ be the solution to $(\ref{delta-rescaled-lagrangian-euler-1})-(\ref{delta-rescaled-lagrangian-euler-3})$ on $[0,T]$, in the sense of Theorem \ref{local-well-posedness-theorem}, for some $M\geq2\ceil*{\alpha}+12$. Suppose the assumptions $\eqref{a-priori-smallness-mu-norm}$--$\eqref{a-priori-sum-of-theta-L-infinity-norms}$ hold. Assume that $\nabla\tilde{W}_{\kappa}\in\X{M}_{\kappa}$ for $\kappa=1,2,\dots,N$. Then for all $\tau\in[0,T]$, we have:
		\begin{align}
		\sum_{\kappa=1}^{N}\left(\left\|\theta_{\kappa}\right\|_{\Y{M}{\kappa}{\nCurl}}^{2}+\left\|\dell_{\tau}\theta_{\kappa}\right\|_{\Y{M}{\kappa}{\nCurl}}^{2}\right)\lesssim S_{M}(0)+C_{M}(0)+\delta S_{M}(\tau).\label{curl-estimates-statement}
		\end{align}
	\end{theorem}

	\section{Energy Estimates}\label{energy-estimates-section}
	In this section, we prove the main energy identity, and subsequent energy estimates that will form the bulk of the proof of Theorem $\ref{main-theorem}$. First we need some definitions.
	
	\begin{definition}[Damping Functional]\label{damping-functional-definition} Assume $\gamma=1+\frac{1}{n}$ for some $n\in\mathbb{Z}_{\geq2}$, or $\gamma\in(1,14/13)$. Let $\{(\dell_{\tau}\theta_{\kappa},\theta_{\kappa})|\kappa=1,2,\dots,N\}$ be the solution to $(\ref{delta-rescaled-lagrangian-euler-1})-(\ref{delta-rescaled-lagrangian-euler-3})$ on $[0,T]$, in the sense of Theorem \ref{local-well-posedness-theorem}, for some $M\geq2\ceil*{\alpha}+12$. Define the damping functional $\mathbb{D}(\tau)$ on $[0,T]$ by
		\begin{align}
		\mathbb{D}(\tau)&=\frac{1}{\delta}\left(2-\beta\right)\sum_{\kappa=1}^{N}\left(\sum_{m+|\underline{n}|=0}^{M}\spaceI e^{\beta\tau}\chi \tilde{W}_{\kappa}^{\alpha+m}\left|\Ndell{m}{n}\dell_{\tau}\theta_{\kappa}\right|^{2}dx+\sum_{|\underline{k}|=0}^{M}\spaceI e^{\beta\tau}\bar{\chi} \tilde{W}_{\kappa}^{\alpha}\left|\nabla^{\underline{k}}\dell_{\tau}\theta_{\kappa}\right|^{2}dx\right)\nonumber\\
		&=(2-\beta)\sum_{\kappa=1}^{N}\frac{1}{\delta}e^{\beta\tau}\left\|\dell_{\tau}\theta_{\kappa}(\tau)\right\|_{\X{M}_{\kappa}}^{2}.\label{damping-functional}
		\end{align}
	\end{definition}

	\begin{remark}
	Just as in $\cite{HaJa3, PaHaJa}$, this damping functional does not play an essential role in our analysis. We only require they have correct sign so that our energy estimates are sufficient. This is guaranteed by the fact that $\beta=3(\gamma-1)\leq3/2$, implied by our range of $\gamma$.
	\end{remark}
	Similarly to $\cite{HaJa3}$, we define a truncated-in-time higher order energy function which we use to prove our main theorem.

	\begin{definition}[Truncated Higher Order Function]\label{truncated-higher-order-function}Assume $\gamma=1+\frac{1}{n}$ for some $n\in\mathbb{Z}_{\geq2}$, or $\gamma\in(1,14/13)$. Let $\{(\dell_{\tau}\theta_{\kappa},\theta_{\kappa})|\kappa=1,2,\dots,N\}$ be the solution to $(\ref{delta-rescaled-lagrangian-euler-1})-(\ref{delta-rescaled-lagrangian-euler-3})$ on $[0,T]$, in the sense of Theorem \ref{local-well-posedness-theorem}, for some $M\geq2\ceil*{\alpha}+12$. For $\tau_{2}\geq\tau_{1}$, define the truncated higher order energy function for each $\kappa$, $S_{\kappa}^{M}(\tau_{1},\tau_{2})$, on $[0,T]$ by
		\begin{align}
		S_{\kappa}^{b}(\tau_{1},\tau_{2})&=S^{b}(\dell_{\tau}\theta_{\kappa},\theta_{\kappa},\tau_{1},\tau_{2})\nonumber\\
		&\coloneqq\sup_{\tau_{1}\leq \tau'\leq \tau_{2}}\left(\frac{1}{\delta}e^{\beta\tau'}\left\|\dell_{\tau}\theta_{\kappa}(\tau')\right\|_{\X{b}_{\kappa}}^{2}+\left\|\theta_{\kappa}(\tau')\right\|_{\X{b}_{\kappa}}^{2}+\left\|\theta_{\kappa}(\tau')\right\|_{\Y{b}{\kappa}{\ngrad}}^{2}+\frac{1}{\alpha}\left\|\theta_{\kappa}(\tau')\right\|_{\Y{b}{\kappa}{\ndiv}}^{2}\right).\label{individual-truncated-energy-functions}
		\end{align}
	The cumulative truncated energy function is then given by
		\begin{align}
		S_{M}(\tau_{1},\tau_{2})=\sum_{\kappa=1}^{N}S_{\kappa}^{b}(\tau_{1},\tau_{2}).\label{cumulative-truncated-energy-function}
		\end{align}
	\end{definition}
	Note that we have the identity
	\begin{align}
	S_{M}(0,\tau)=S_{M}(\tau),\label{truncated-energy-function-at-0}
	\end{align}
	and, for $\tau_{1}\leq\tau_{2}$, the inequalities
	\begin{align}
	S_{M}(\tau_{1},\tau_{2})\leq S_{M}(\tau_{2})\leq\max\{S_{M}(\tau_{1},\tau_{2}),S_{M}(\tau_{1})\}\leq S_{M}(\tau_{1},\tau_{2})+S_{M}(\tau_{1}).\label{truncated-and-original-energy-function-comparison}
	\end{align}
	With the damping functional from Definition $\ref{damping-functional-definition}$, we can state our main energy identity.
	
	\begin{theorem}[Higher Order Energy Identity]\label{energy-identity-theorem}Assume $\gamma=1+\frac{1}{n}$ for some $n\in\mathbb{Z}_{\geq2}$, or $\gamma\in(1,14/13)$. Let $\{(\dell_{\tau}\theta_{\kappa},\theta_{\kappa})|\kappa=1,2,\dots,N\}$ be the solution to $(\ref{delta-rescaled-lagrangian-euler-1})-(\ref{delta-rescaled-lagrangian-euler-3})$ on $[0,T]$, in the sense of Theorem \ref{local-well-posedness-theorem}, for some $M\geq2\ceil*{\alpha}+12$. Suppose the assumptions $\eqref{a-priori-smallness-mu-norm}$--$\eqref{a-priori-sum-of-theta-L-infinity-norms}$ hold. Assume that $\nabla\tilde{W}_{\kappa}\in\X{M}_{\kappa}$ for $\kappa=1,2,\dots,N$. Then for all $\tau\in[0,T]$, we have:
		\begin{align}
		&\frac{1}{2}\dell_{\tau}\left(\sum_{\kappa=1}^{N}\left(\frac{1}{\delta}e^{\beta\tau}\left\|\dell_{\tau}\theta_{\kappa}\right\|_{\X{M}_{\kappa}}^{2}+\left\|\theta_{\kappa}\right\|_{\Y{M}{\kappa}{\ngrad}}^{2}+\frac{1}{\alpha}\left\|\theta_{\kappa}\right\|_{\Y{M}{\kappa}{\ndiv}}^{2}-\left\|\theta_{\kappa}\right\|_{\Y{M}{\kappa}{\nCurl}}^{2}\right)\right)+\frac{1}{2}\mathbb{D}(\tau)\nonumber\\
		&=\sum_{\kappa=1}^{N}\left(\sum_{m+|\underline{n}|=0}^{M}\int_{\Omega}\chi\mathcal{R}_{\kappa}(m,\underline{n})^{i}\Ndell{m}{n}\dell_{\tau}\theta_{\kappa}^{i}dx+\sum_{|\underline{k}|=0}^{M}\int_{\Omega}\bar{\chi}\mathcal{R}_{\kappa}(\underline{k})^{i}\nabla^{\underline{k}}\dell_{\tau}\theta_{\kappa}^{i}dx\right)\nonumber\\
		&-\sum_{\kappa=1}^{N}\sum_{m+|\underline{n}|=0}^{M}\frac{1}{\delta}e^{\beta \tau}\int_{\Omega}\chi\tilde{W}_{\kappa}^{\alpha+m}\Ndell{m}{n}\left(\ninv[\kappa]^{j}_{i}\dell_{j}\psi_{\kappa}\right)\Ndell{m}{n}\dell_{\tau}\theta_{\kappa}^{i}\ dx\ ds\nonumber\\
		&-\sum_{\kappa=1}^{N}\sum_{|\underline{k}|=0}^{M}\frac{1}{\delta}e^{\beta \tau}\int_{\Omega}\bar{\chi}\tilde{W}_{\kappa}^{\alpha}\nabla^{\underline{k}}\left(\ninv[\kappa]^{j}_{i}\dell_{j}\psi_{\kappa}\right)\nabla^{\underline{k}}\dell_{\tau}\theta_{\kappa}^{i}\ dx\ ds,\label{energy-identity}
		\end{align}
		with $\mathcal{R}_{\kappa}(m,\underline{n})$ and $\mathcal{R}_{\kappa}(\underline{k})$ being remainder terms that we can bound effectively.
	\end{theorem}

	\begin{proof}[\textbf{Proof}]The proof of Theorem $\ref{energy-identity-theorem}$ can be adapted from the methods used in~\cite{PaHaJa}. For each fixed $\kappa$, on $\supp{\chi}$ we first divide by $\tilde{W}_{\kappa}^{\alpha}$ in $\eqref{delta-rescaled-lagrangian-euler-1}$, and then act on the result with $\tilde{W}_{\kappa}^{\alpha+m}\Ndell{m}{n}$, some $0\leq m+|\underline{n}|\leq M$, resulting in the identity
		\begin{align*}
		\frac{1}{\delta}e^{\beta\tau}\tilde{W}_{\kappa}^{\alpha+m}\Ndell{m}{n}\left(\dell_{\tau\tau}\theta_{\kappa}+\dell_{\tau}\theta_{\kappa}\right)&+\tilde{W}_{\kappa}^{\alpha+m}\Ndell{m}{n}\left(\frac{1}{\tilde{W}_{\kappa}^{\alpha}}\dell_{k}\left(\tilde{W}_{\kappa}^{1+\alpha}\ninv[\kappa]^{k}_{*}\njac_{\kappa}^{-1/\alpha}\right)\right)\nonumber\\
		&+\frac{1}{\delta}e^{\beta\tau}\tilde{W}_{\kappa}^{\alpha+m}\Ndell{m}{n}\left(\ninv[\kappa]^{k}_{*}\dell_{k}\psi_{\kappa}\right)=0.
		\end{align*}
		From here, we take the $L^{2}(\Omega)$ inner product of this equality with $\chi\Ndell{m}{n}\dell_{\tau}\theta_{\kappa}$. We can follow a similar procedure on $\supp{\bar{\chi}}$ with $\tilde{W}_{\kappa}^{\alpha}\nabla^{\underline{k}}$ and $\bar{\chi}\nabla^{\underline{k}}\dell_{\tau}\theta_{\kappa}$. From here we can extract the left hand side in $\eqref{energy-identity}$, and move everything else to the right hand side. Doing this for each $\kappa=1,2,\dots,N$ and then summing gives us the energy identity $\eqref{energy-identity}$, once we have labelled the remainders of this procedure from the pressure term $\mathcal{R}_{\kappa}(m,\underline{n})$ and $\mathcal{R}_{\kappa}(\underline{k})$ respectively.
	\end{proof}
	
	\begin{theorem}[Higer Order Energy Inequality]\label{energy-inequality-theorem}Assume $\gamma=1+\frac{1}{n}$ for some $n\in\mathbb{Z}_{\geq2}$, or $\gamma\in(1,14/13)$. Let $\{(\dell_{\tau}\theta_{\kappa},\theta_{\kappa})|\kappa=1,2,\dots,N\}$ be the solution to $(\ref{delta-rescaled-lagrangian-euler-1})-(\ref{delta-rescaled-lagrangian-euler-3})$ on $[0,T]$, in the sense of Theorem \ref{local-well-posedness-theorem}, for some $M\geq2\ceil*{\alpha}+12$. Suppose the assumptions $\eqref{a-priori-smallness-mu-norm}$--$\eqref{a-priori-sum-of-theta-L-infinity-norms}$ hold. Assume that $\nabla\tilde{W}_{\kappa}\in\X{M}_{\kappa}$ for $\kappa=1,2,\dots,N$. Then for all $[\tau_{1},\tau]\subset[0,T]$, we have:
		\begin{align}
		S_{M}(\tau_{1},\tau)\leq C_{1}S_{M}(\tau_{1})+C_{2}\left(S_{M}(0)+C_{M}(0)\right)+C_{3}\sqrt{\delta}+C_{4}\sqrt{\delta}S_{M}(\tau_{1},\tau)+C_{5}\sqrt{\delta}\int_{\tau_{1}}^{\tau}\mathcal{G}(\tau')S_{M}(\tau_{1},\tau')\ d\tau',\label{energy-inequality}
		\end{align}
	for some constants $C_{1},\dots,C_{5}\in[1,\infty)$, and $\mathcal{G}:[0,\infty)\rightarrow[0,\infty)$ integrable.
	\end{theorem}

	\begin{proof}Let $0\leq\tau_{1}\leq s\leq\tau\leq T$. Integrating the identity $\eqref{energy-identity}$ over $[\tau_{1},s]$, as well as utilising Lemma $\ref{theta-estimates-lemma}$, Theorem $\ref{potential-term-estimate-theorem}$, and Theorem $\ref{curl-estimates-theorem}$ gives us
		\begin{align}
		&\sum_{\kappa=1}^{N}\left(\frac{1}{\delta}e^{\beta s}\left\|\dell_{\tau}\theta_{\kappa}(s)\right\|_{\X{M}_{\kappa}}^{2}+\left\|\theta_{\kappa}(s)\right\|_{\X{M}_{\kappa}}^{2}+\left\|\theta_{\kappa}(s)\right\|_{\Y{M}{\kappa}{\ngrad}}^{2}+\frac{1}{\alpha}\left\|\theta_{\kappa}(s)\right\|_{\Y{M}{\kappa}{\ndiv}}^{2}\right)+\int_{\tau_{1}}^{s}\mathbb{D}(\tau')\ d\tau'\nonumber\\
		&\lesssim\sum_{\kappa=1}^{N}\left(\frac{1}{\delta}e^{\beta\tau_{1}}\left\|\dell_{\tau}\theta_{\kappa}(\tau_{1})\right\|_{\X{M}_{\kappa}}^{2}+\left\|\theta_{\kappa}(\tau_{1})\right\|_{\Y{M}{\kappa}{\ngrad}}^{2}+\left\|\theta_{\kappa}(\tau_{1})\right\|_{\Y{M}{\kappa}{\ndiv}}^{2}\right)+S_{M}(0)+C_{M}(0)+\delta S_{M}(s)\nonumber\\
		&+\sum_{\kappa=1}^{N}\left(\sum_{m+|\underline{n}|=0}^{M}\int_{\tau_{1}}^{s}\int_{\Omega}\chi\mathcal{R}_{\kappa}(m,\underline{n})^{i}\Ndell{m}{n}\dell_{\tau}\theta_{\kappa}^{i}dxd\tau'+\sum_{|\underline{k}|=0}^{M}\int_{\tau_{1}}^{s}\int_{\Omega}\bar{\chi}\mathcal{R}_{\kappa}(\underline{k})^{i}\nabla^{\underline{k}}\dell_{\tau}\theta_{\kappa}^{i}dxd\tau'\right)\nonumber\\
		&+\delta^{\alpha-\frac{1}{2}}\int_{\tau_{1}}^{s}e^{\left(\frac{\beta}{2}-1\right)\tau'}\sqrt{S_{M}(\tau')}\ d\tau'.\label{energy-inequality-proof-1}
		\end{align}
		\noindent For each $\kappa$ most of the terms coming from $\mathcal{R}_{\kappa}(m,\underline{n})$ and $\mathcal{R}_{\kappa}(\underline{k})$ will be analogous to the ones obtained in \cite{PaHaJa}. However, just as with the curl estimates in Section $\ref{curl-estimates-section}$, we will also have remainder terms that look, for example, like a weighted $L^{2}(\Omega)$ inner product of $\Ndell{m}{n}\mu_{\kappa}$ and $\Ndell{m}{n}\dell_{\tau}\theta_{\kappa}$, coming from the derivatives of $\ninv[\kappa]$ and $\pjac_{\kappa}$. These terms can be bounded, for example like
		\begin{align}
		\left|\int_{\Omega}\chi \tilde{W}_{\kappa}^{\alpha+m}\Ndell{m}{n}\mu_{\kappa}^{i}\Ndell{m}{n}\dell_{\tau}\theta_{\kappa}dx\right|\lesssim\sqrt{\delta}e^{-\frac{\beta}{2}\tau}\sqrt{S_{M}(\tau)},\label{energy-inequality-proof-1a}
		\end{align}
		where we use $\eqref{a-priori-smallness-mu-norm}$ to bound $\left\|\nabla\mu_{\kappa}\right\|_{\X{M}_{\kappa}}$ by a constant.
			
		Moreover, as we are assuming $\delta$ is small, in particular $\delta<1$, we have the bound $\delta+\delta^{\alpha-\frac{1}{2}}<2\sqrt{\delta}$, as $\alpha\geq2$. Thus we have
		\begin{align}
		&\sum_{\kappa=1}^{N}\left(\frac{1}{\delta}e^{\beta s}\left\|\dell_{\tau}\theta_{\kappa}(s)\right\|_{\X{M}_{\kappa}}^{2}+\left\|\theta_{\kappa}(s)\right\|_{\X{M}_{\kappa}}^{2}+\left\|\theta_{\kappa}(s)\right\|_{\Y{M}{\kappa}{\ngrad}}^{2}+\frac{1}{\alpha}\left\|\theta_{\kappa}(s)\right\|_{\Y{M}{\kappa}{\ndiv}}^{2}\right)\nonumber\\
		&\lesssim S_{M}(\tau_{1})+S_{M}(0)+C_{M}(0)+\sqrt{\delta} S_{M}(s)+\sqrt{\delta}\int_{\tau_{1}}^{s}e^{\left(\frac{\beta}{2}-1\right)\tau'}\sqrt{S_{M}(\tau')}\ d\tau'+\sqrt{\delta}\int_{\tau_{1}}^{s}e^{-\frac{\beta}{2}\tau'}\left(\sqrt{S_{M}(\tau')}+S_{M}(\tau')\right)\ d\tau'\nonumber\\
		&\lesssim S_{M}(\tau_{1})+S_{M}(0)+C_{M}(0)+\sqrt{\delta} S_{M}(s)+\sqrt{\delta}+\sqrt{\delta}\int_{\tau_{1}}^{s}\left(e^{\left(\frac{\beta}{2}-1\right)\tau'}+e^{-\frac{\beta}{2}\tau'}\right)S_{M}(\tau')\ d\tau'\nonumber\\
		&\lesssim S_{M}(\tau_{1})+S_{M}(0)+C_{M}(0)+\sqrt{\delta} S_{M}(\tau_{1},s)+\sqrt{\delta}+\sqrt{\delta}\int_{\tau_{1}}^{s}\left(e^{\left(\frac{\beta}{2}-1\right)\tau'}+e^{-\frac{\beta}{2}\tau'}\right)S_{M}(\tau_{1},\tau')\ d\tau'.\label{energy-inequality-proof-2}
		\end{align}
		The second bound is obtained by using Young's inequality which gives us $\sqrt{\delta}\sqrt{S_{M}(\tau')}\lesssim\sqrt{\delta}+\sqrt{\delta}S_{M}(\tau')$, and also the fact that $e^{\left(\frac{\beta}{2}-1\right)\tau}$ and $e^{-\frac{\beta}{2}\tau}$ are integrable, as $0<\beta\leq3/2$. The last bound is obtain by using $\eqref{truncated-and-original-energy-function-comparison}$, as well as the integrability of $e^{\left(\frac{\beta}{2}-1\right)\tau}$ and $e^{-\frac{\beta}{2}\tau}$ once again.
		
		As $\eqref{energy-inequality-proof-2}$ holds for all $s\in[\tau_{1},\tau]$, we can take the supremum over this interval, and obtain
		\begin{align}
		S_{M}(\tau_{1},\tau)\leq C_{1}S_{M}(\tau_{1})+C_{2}\left(S_{M}(0)+C_{M}(0)\right)+C_{3}\sqrt{\delta}+C_{4}\sqrt{\delta}S_{M}(\tau_{1},\tau)+C_{5}\sqrt{\delta}\int_{\tau_{1}}^{\tau}\mathcal{G}(\tau')S_{M}(\tau_{1},\tau')\ d\tau',\label{energy-inequality-proof-3}
		\end{align}
		where $\mathcal{G}(\tau)=e^{\left(\frac{\beta}{2}-1\right)\tau}+e^{-\frac{\beta}{2}\tau}$. The statement $\eqref{energy-inequality}$ follows by noting that we can adjust constants on the right hand side of $\eqref{energy-inequality-proof-2}$ to be $\geq1$ if necessary.
	\end{proof}

\section{Proof of Theorem $\ref{main-theorem}$.}\label{proof-of-main-result-section}
Let $T$ be such that
\begin{align}
\sup_{0\leq\tau\leq T}S_{M}(\tau)=S_{M}(T)\leq \bar{C}(S_{M}(0)+C_{M}(0)+\sqrt\delta),\label{main-proof-inequality-1}
\end{align}
for some constant $\bar{C}$, whose existence is guaranteed by the Local Well-Posedness theory set out in Theorem $\ref{local-well-posedness-theorem}$. Let $C^{*}$ be defined by
\begin{align}
C^{*}=3(C_{1}\bar{C}+C_{2}+C_{3}),\label{C-star-definition}
\end{align}
with $C_{i}$, $i=1,2,3,4,5$ given in Theorem \ref{energy-inequality-theorem}. Note that $\bar{C}<C^{*}$, so given
\begin{align}
T^{*}=\sup_{\tau\geq0}\left\{\tau\ |\text{\ Solution to $\eqref{delta-rescaled-lagrangian-euler-1}$--$\eqref{delta-rescaled-lagrangian-euler-3}$ exists on $[0,\tau]$\ }, S_{M}(\tau)\leq C^{*}(S_{M}(0)+C_{M}(0)+\sqrt\delta)\right\},\label{T-star-definition}
\end{align}
we have $T\leq T^{*}$. Now letting $\tau_{1}=T/2$, Theorem \ref{energy-inequality-theorem} tells us that for any $\tau\in[\frac{T}{2},T^{*})$, we have
\begin{align}
S_{M}\left(\frac{T}{2},\tau\right)&\leq C_{1}S_{M}\left(\frac{T}{2}\right)+C_{2}\sqrt\delta+C_{3}\left(C_{M}(0)+S_{M}(0)\right)+C_{4}\sqrt{\delta}S_{M}\left(\frac{T}{2},\tau\right)+C_{5}\sqrt{\delta}\int_{\frac{T}{2}}^{\tau}\mathcal{G}(\tau')S_{M}\left(\frac{T}{2},\tau'\right)\ d\tau'\nonumber\\
&\leq C_{1}S_{M}\left(\frac{T}{2}\right)+C_{2}\sqrt\delta+C_{3}\left(C_{M}(0)+S_{M}(0)\right)+C_{4}\sqrt{\delta}S_{M}\left(\frac{T}{2},\tau\right)+C_{5}\sqrt{\delta}\left(\int_{0}^{\infty}\mathcal{G}(\tau')\ d\tau'\right)S_{M}\left(\frac{T}{2},\tau\right).\label{main-proof-inequality-2}
\end{align}
Let $\delta$ satisfy
\begin{align}
\delta\leq\min{\left(\left(4C_{4}\right)^{-2},\left(4C_{5}\int_{0}^{\infty}\mathcal{G}(\tau')d\tau'\right)^{-2}\right)}.\label{delta-smallness-requirements}
\end{align}
Then we have 
\begin{align}
S_{M}\left(\frac{T}{2},\tau\right)\leq 2C_{1}S_{M}\left(\frac{T}{2}\right)+2C_{2}\sqrt\delta+2C_{3}\left(C_{M}(0)+S_{M}(0)\right).\label{main-proof-inequality-3}
\end{align}
Using $(\ref{main-proof-inequality-1})$ to bound $S_{M}(T/2)$ by $\bar{C}\left(C_{M}(0)+S_{M}(0)+\sqrt\delta\right)$, we obtain
\begin{align}
S_{M}\left(\frac{T}{2},\tau\right)&\leq2C_{1}\bar{C}\left(C_{M}(0)+S_{M}(0)+\sqrt\delta\right)+2C_{2}\sqrt\delta+2C_{3}\left(C_{M}(0)+S_{M}(0)\right)\nonumber\\
&\leq2(C_{1}\bar{C}+C_{2}+C_{3})\left(C_{M}(0)+S_{M}(0)+\sqrt \delta\right)<C^{*}\left(C_{M}(0)+S_{M}(0)+\sqrt\delta\right).\label{main-proof-inequality-4}
\end{align}
Combining $(\ref{main-proof-inequality-4})$ with $(\ref{main-proof-inequality-1})$, and using $\eqref{truncated-and-original-energy-function-comparison}$ we obtain
\begin{align}
S_{M}(\tau)\leq\max\{\bar{C},2(C_{1}\bar{C}+C_{2}+C_{3})\}\left(C_{M}(0)+S_{M}(0)+\sqrt\delta\right)<C^{*}\left(C_{M}(0)+S_{M}(0)+\sqrt\delta\right)
\end{align}
for all $\tau\in[0,T^{*})$. Shrinking $\delta$ further if necessary, we also improve our a priori assumptions $\eqref{a-priori-bound-cumulative-energy}$--$\eqref{a-priori-sum-of-theta-L-infinity-norms}$. For example, for all $\tau\in[0,T^{*})$,
\begin{align}
\left\|\ninv[\kappa]-\mathbb{I}\right\|_{L^{\infty}}=\left\|\int_{0}^{\tau}\dell_{\tau}\ninv[\kappa] d\tau'\right\|_{L^{\infty}}\leq C\sqrt{\delta}\int_{0}^{\tau}e^{-\frac{\beta}{2}\tau'}S_{M}(\tau')d\tau'<\varepsilon_{2},
\end{align}
for small enough $\delta$. Similarly for $\njac_{\kappa}$, and $\theta_{\kappa}$. Then, by continuity of $S_{M}(\tau)$ as a function of $\tau$, we must have that $T^{*}=\infty$. Therefore, the bound $\eqref{global-energy-estimate}$ follows. It is left to prove $\eqref{theta-infinity-statement}$. Fix $\kappa\in\{1,2,\dots,N\}$. Let $\tau_{1}>\tau_{2}$. For any $m+|\underline{n}|\leq M$, we have the estimate
\begin{align}
&\spaceI\chi \tilde{W}_{\kappa}^{\alpha+m}\left|\Ndell{m}{n}\left(\theta_{\kappa}(\tau_{1})-\theta_{\kappa}(\tau_{2})\right)\right|^{2}dx=\spaceI\chi \tilde{W}_{\kappa}^{\alpha+m}\left|\int_{\tau_{2}}^{\tau_{1}}\Ndell{m}{n}\dell_{\tau}\theta_{\kappa}(\tau')d\tau'\right|^{2}dx\nonumber\\
&\lesssim\left(\int_{\tau_{2}}^{\tau_{1}}e^{-\frac{\beta}{2}\tau'}d\tau'\right)\int_{\tau_{2}}^{\tau_{1}}e^{\frac{\beta}{2}\tau'}\spaceI\chi W^{\alpha+m}\left|\Ndell{m}{n}\dell_{\tau}\theta_{\kappa}\right|^{2}dx\ d\tau'\nonumber\\
&\lesssim\delta S_{M}(\tau_{1})\int_{\tau_{2}}^{\tau_{1}}e^{-\frac{\beta}{2}\tau'}d\tau' \nonumber \\
& \lesssim \delta \left(e^{-\frac{\sigma_{1}}{2}\tau_{2}}-e^{-\frac{\sigma_{1}}{2}\tau_{1}}\right)\rightarrow0\ \ \tau_{1},\tau_{2}\rightarrow\infty
\end{align}
The bound follows from Cauchy-Schwarz in $\tau$ and the global energy estimate $\eqref{global-energy-estimate}$. A similar bound holds for any $|\underline{k}|\leq N$ on $\supp{\bar{\chi}}$. Thus $\theta_{\kappa}(\tau_{n})$ is Cauchy, for any strictly increasing sequence $\tau_{n}$. As $\X{M}_{\kappa}$ is a Banach space, $\lim_{\tau\to\infty}\theta_{\kappa}(\tau)$ exists in $\X{M}_{\kappa}$; we call this limit $\theta_{\kappa}^{(\infty)}$. This gives $(\ref{theta-infinity-statement})$.

\section*{Acknowledgments}
The author acknowledges the support of the EPSRC studentship grant EP/N509498/1.
		
		\begin{appendices}
		\section{Local Well-Posedness}\label{local-well-posedness}
		The local-in-time theory for the system $\eqref{delta-rescaled-lagrangian-euler-1}$--$\eqref{delta-rescaled-lagrangian-euler-3}$ can be adapted from \cite{HaJa4}, using the framework set out in \cite{JaMa2015}. The tidal terms of the form $\mathscr{I}_{[\kappa,\kappa']}$, defined in $\eqref{lagrangian-gradient-of-lagrangian-potential}$, are present because $N\geq2$, since these terms are exactly the ones which encode how the stars in our system interact with each other via gravity. In \cite{HaJa4}, $N=1$, meaning they only have to deal with the self-interaction terms $\mathscr{G}_{\kappa}$, also defined in $\eqref{lagrangian-gradient-of-lagrangian-potential}$. As mentioned in the introduction, in Eulerian coordinates, the gravitational potential term $\rho_{\kappa}\nabla\phi$ in $\eqref{pre-euler-2}$ is lower order, from the point of view of derivatives, to $\nabla p_{\kappa}$. As discussed in \cite{HaJa4}, this means that it is enough to show sufficient a priori estimates for the gravitational potential terms in order to construct a local-in-time theory. For the tidal terms $\mathscr{I}_{[\kappa,\kappa]}$, these can be shown using similar methods to Proposition $\ref{tidal-term-estimates-proposition}$. For details on the self-interaction terms and other aspects of constructing the local-in-time theory, see $\cite{JaMa2015,HaJa4}$.
		\section{Useful Estimates}\label{useful-estimates}
			
		In this appendix we state some estimates that will be used in the proof of global-in-time existence of our solutions.
			
			\begin{lemma}\label{theta-estimates-lemma}
				Assume $\gamma=1+\frac{1}{n}$ for some $n\in\mathbb{Z}_{\geq2}$, or $\gamma\in(1,14/13)$. Let $\{(\dell_{\tau}\theta_{\kappa},\theta_{\kappa})|\kappa=1,2,\dots,N\}$ be the solution to $(\ref{delta-rescaled-lagrangian-euler-1})-(\ref{delta-rescaled-lagrangian-euler-3})$ on $[0,T]$, in the sense of Theorem \ref{local-well-posedness-theorem}, for some $M\geq2\ceil*{\alpha}+12$. Suppose the assumptions $\eqref{a-priori-smallness-mu-norm}$--$\eqref{a-priori-sum-of-theta-L-infinity-norms}$ hold. Then for all $\tau\in[0,T]$, we have:
				\begin{align}
				\sum_{\kappa=1}^{N}\left\|\theta_{\kappa}\right\|_{\X{M}_{\kappa}}^{2}\lesssim\delta S_{M}(\tau).\label{theta-estimates-lemma-statement}
				\end{align}
			\end{lemma}
			
			\begin{proof}[\textbf{Proof}]
				The proof is readily adadpted from \cite{PaHaJa}. It is an application of the Fundamental Theorem of Calculus.
			\end{proof}
			
			
			\begin{lemma}\label{lower-order-estimates-lemma}Assume $\gamma=1+\frac{1}{n}$ for some $n\in\mathbb{Z}_{\geq2}$, or $\gamma\in(1,14/13)$. Let $\{(\dell_{\tau}\theta_{\kappa},\theta_{\kappa})|\kappa=1,2,\dots,N\}$ be the solution to $(\ref{delta-rescaled-lagrangian-euler-1})-(\ref{delta-rescaled-lagrangian-euler-3})$ on $[0,T]$, in the sense of Theorem \ref{local-well-posedness-theorem}, for some $M\geq2\ceil*{\alpha}+12$. Suppose the assumptions $\eqref{a-priori-smallness-mu-norm}$--$\eqref{a-priori-lower-bound-convex-combination-mu-centres}$ hold. Assume that $\nabla\tilde{W}_{\kappa}\in\X{M}_{\kappa}$ for $\kappa=1,2,\dots,N$. Let $(m,\underline{n},\underline{k})\in\mathbb{Z}_{\geq0}\times\mathbb{Z}_{\geq0}^{3}\times\mathbb{Z}_{\geq0}^{3}$ be such that $1\leq\max{\{m+|\underline{n}|,|\underline{k}|\}}\leq M$. Then for each $\kappa\in\{1,\dots,N\}$ and all $\tau\in[0,T]$, we have:
				\begin{align}
				\spaceI\chi \tilde{W}_{\kappa}^{1+\alpha+m}\left|\Ndell{m}{n}\left(\ninv[\kappa]\right)\right|^{2}dx+\spaceI\bar{\chi}\tilde{W}_{\kappa}^{1+\alpha}\left|\nabla^{\underline{k}}\left(\njac_{\kappa}^{-1}\right)\right|^{2}dx\lesssim\left\|\nabla\mu_{\kappa}\right\|_{\X{M}_{\kappa}}^{2}+ S_{M}(\tau),\label{lower-order-estimates-statement-2}
				\end{align}
				If instead we have $1\leq\max{\{m+|\underline{n}|,|\underline{k}|\}}\leq M/2$, then for each $\kappa\in\{1,\dots,N\}$ and all $\tau\in[0,T]$, we have:
				\begin{align}
				\left\|\tilde{W}_{\kappa}^{\frac{m}{2}}\Ndell{m}{n}\ninv[\kappa]\right\|_{L^{\infty}(\Omega)}^{2}+\left\|\tilde{W}_{\kappa}^{\frac{m}{2}}\Ndell{m}{n}\left(\njac_{\kappa}^{-1}\right)\right\|_{L^{\infty}(\Omega)}^{2}\lesssim\left\|\nabla\mu_{\kappa}\right\|_{\X{M}_{\kappa}}^{2}+ S_{M}(\tau).\label{lower-order-estimates-statement-3}
				\end{align}
			\end{lemma}
			
			\begin{proof}[\textbf{Proof}]
				The full proof of $\eqref{lower-order-estimates-statement-2}$ can be adapted very easily from \cite{PaHaJa}. The main technical difficulties are for the integrals localised near the boundary, on $\supp{\chi}$, with regards to ensuring each term resulting from the application of the Leibniz rule to either $\Ndell{m}{n}\left(\ninv[\kappa]\right)$ or $\Ndell{m}{n}\left(\njac_{\kappa}^{-1}\right)$ appears with a high enough power of $\tilde{W}_{\kappa}$ to be bounded in either the $\X{M}_{\kappa}$ or $\Y{M}{\kappa}{\ngrad}$ norms.
				
				For the $L^{\infty}$ estimates in $\eqref{lower-order-estimates-statement-3}$, we first concentrate on the $\ninv[\kappa]$ term. Note that $(\ref{rescaled-inverse-differentiation-formula})$ gives
				\begin{align}
				\angdell\ninv[\kappa]^{k}_{i}=-\ninv[\kappa]^{k}_{j}\ninv[\kappa]^{s}_{i}\angdell\dell_{s}\zeta_{\kappa}^{j},\label{inverse-differentiation-rule-tangential-derivatives}
				\end{align}
				and similarly for $\rdell\ninv[\kappa]$. We apply $\eqref{inverse-differentiation-rule-tangential-derivatives}$ repeatedly to obtain
				\begin{align}
				\left\|\tilde{W}_{\kappa}^{m}\Ndell{m}{n}\ninv[\kappa]\right\|_{L^{\infty}(\Omega)}\lesssim\sum_{a+|\underline{b}|=1}^{m+|\underline{n}|}\left\|\ninv[\kappa]\right\|_{L^{\infty}(\Omega)}^{l_{a+\underline{b}}}\prod_{\substack{\nu_{1}+\dots+\nu_{a}=a\\\underline{\lambda}_{1}+\dots+\underline{\lambda}_{|\underline{b}|}=\underline{b}}}\left\|\tilde{W}_{\kappa}^{\nu_{j}}\rdell^{\nu_{j}}\angdell^{\underline{\lambda}_{i}}\left(\nabla\zeta_{\kappa}\right)\right\|_{L^{\infty}(\Omega)}^{l_{\nu_{j}+\underline{\lambda}_{i}}}\label{lower-order-estimates-statement-3-1}
				\end{align}
				where $l_{a+\underline{b}},l_{\nu_{j}+\underline{\lambda}_{i}}$ are nonnegative integers. By $(\ref{a-priori-bound-grad-zeta})$, we have $\left\|\ninv[\kappa]\right\|_{L^{\infty}(\Omega)}\lesssim1$. It is left to bound terms of the form $\left\|\tilde{W}_{\kappa}^{\nu}\rdell^{\nu}\angdell^{\underline{\lambda}}\left(\nabla\zeta_{\kappa}\right)\right\|_{L^{\infty}(\Omega)}$, for any $\nu\in\mathbb{Z}_{\geq0}$, $\underline{\lambda}\in\mathbb{Z}_{\geq0}^{3}$ such that $1\leq\nu+|\underline{\lambda}|\leq m+|\underline{n}|$. First we have
				\begin{align}
				\left\|\tilde{W}_{\kappa}^{\nu}\rdell^{\nu}\angdell^{\underline{\lambda}}\left(\nabla\zeta_{\kappa}\right)\right\|_{L^{\infty}(\Omega)}\lesssim\left\|\tilde{W}_{\kappa}^{\nu}\rdell^{\nu}\angdell^{\underline{\lambda}}\left(\nabla\theta_{\kappa}\right)\right\|_{L^{\infty}(\Omega)}+\left\|\tilde{W}_{\kappa}^{\nu}\rdell^{\nu}\angdell^{\underline{\lambda}}\left(\nabla\mu_{\kappa}\right)\right\|_{L^{\infty}(\Omega)}.\label{lower-order-estimates-statement-3-2}
				\end{align}
				For the $\theta_{\kappa}$ term, we write $\rdell^{\nu}\angdell^{\underline{\lambda}}\nabla$ as a linear combination of operators of the form $\nabla\rdell^{\nu'}\angdell^{\underline{\lambda}'}$, $\nu'+|\underline{\lambda}'|\leq\nu+|\underline{\lambda}|$ with smooth coefficients using $(\ref{angular-derivative-def}), (\ref{radial-derivative-def})$, and $(\ref{rectangular-as-ang-rad})$. Finally, employing $(\ref{L-infinity-energy-space-bound-weights-statement-2})$ and $\eqref{L-infinity-energy-space-bound-weights-statement-1}$ from Lemma $\ref{L-infinity-energy-space-bound}$ to the $\theta_{\kappa}$ and $\mu_{\kappa}$ terms respectively gives us
				\begin{align}
				\left\|\tilde{W}_{\kappa}^{\nu}\rdell^{\nu}\angdell^{\underline{\lambda}}\left(\nabla\zeta_{\kappa}\right)\right\|_{L^{\infty}(\Omega)}&\lesssim\left(\sum_{\nu'+|\underline{\lambda}'|\leq\nu+|\underline{\lambda}|}\left\|\tilde{W}_{\kappa}^{\nu}\nabla\rdell^{\nu'}\angdell^{\underline{\lambda}'}\left(\theta_{\kappa}\right)\right\|_{L^{\infty}(\Omega)}\right)+\left\|\tilde{W}_{\kappa}^{\nu}\rdell^{\nu}\angdell^{\underline{\lambda}}\left(\nabla\mu_{\kappa}\right)\right\|_{L^{\infty}(\Omega)}\nonumber\\
				&\lesssim\left\|\nabla\mu_{\kappa}\right\|_{\X{M}_{\kappa}}+\sqrt{S_{M}(\tau)}.\label{lower-order-estimates-statement-3-3}
				\end{align}
				The proof of the same bound for the $\njac_{\kappa}$ term is similar, where we utilise the differentiation formula $\eqref{rescaled-jacobian-differentiation-formula}$ instead.
			\end{proof}			
		
			\section{Commutators and Transformations of Derivatives}\label{commutators-transformation-of-derivatives}
			We recall that in Section $\ref{derivatives}$, we introduced our angular and radial derivatives, $\angdell$ and $\rdell$ and the various commutator identities $(\ref{commutator-ang-rad})-(\ref{commutator-rectangular-ang})$. In this appendix, we state how these commutators can be written as radial and angular derivatives, following \cite{HaJa3}. We also record how $\Ndell{m}{n}$ can be written as a sum of $\nabla^{\underline{k}}$ on an appropriate sub-domain of the unit ball $B_{1}$, and vice versa.
			
			\begin{lemma}[Higher order commutator identities]\label{higher-order-commutator-identities} Let $\dell_{s}$ be a rectangular derivative. Then for $\underline{n}\in\mathbb{Z}^{3}_{\geq0}$ with $|\underline{n}|>0$, we have
				\begin{align}
				\comm{\Ndell{m}{n}}{\dell_{s}}=\sum_{i=1}^{m+|\underline{n}|}\sum_{\substack{a+|\underline{b}|=i\\a\leq m+1}}\mathcal{K}_{s,i,a,\underline{b}}\Ndell{a}{b},\label{expansion-commutator-Ndell-rectangular}
				\end{align}
				where $\mathcal{K}_{s,i,a,\underline{b}}$ are smooth functions away from the origin. If $|\underline{n}|=0$, then we instead have
				\begin{align}
				\comm{\rdell^{m}}{\dell_{s}}=\sum_{i=1}^{m+|\underline{n}|}\sum_{\substack{a+|\underline{b}|=i\\a\leq m}}\mathcal{F}_{s,i,a,\underline{b}}\Ndell{a}{b},\label{expansion-commutator-rdell-rectangular}
				\end{align}
				where, again, $\mathcal{F}_{s,i,a,\underline{b}}$ are smooth functions away from the origin.
			\end{lemma}
			
			\begin{lemma}\label{Ndell-to-Dk-lemma}Let $B_{1}$ be the closed unit ball in $\mathbb{R}^{3}$, and let $\Ndell{m}{n}$ and $\nabla^{\underline{k}}$ be defined as in $(\ref{Ndell-def})$ and $(\ref{Dk-def})$. Then we have
				\begin{align}
				\Ndell{m}{n}=\sum_{|\underline{k}|=0}^{m+|\underline{n}|}\mathcal{Q}_{\underline{k}}\nabla^{\underline{k}},\label{Ndell-to-Dk-statement}
				\end{align}
				for some $\mathcal{Q}_{\underline{k}}$, smooth on $B_{1}$.
			\end{lemma}
			There is also a partial converse to Lemma $\ref{Ndell-to-Dk-lemma}$.
			
			\begin{lemma}\label{Dk-to-Ndell-lemma}Let $B_{1}$ be the closed unit ball in $\mathbb{R}^{3}$, and let $\Ndell{m}{n}$ and $\nabla^{\underline{k}}$ be defined as in $(\ref{Ndell-def})$ and $(\ref{Dk-def})$. Then we have
				\begin{align}
				\nabla^{\underline{k}}=\sum_{m+|\underline{n}|=0}^{|\underline{k}|}\mathcal{Z}_{(m,\underline{n})}\Ndell{m}{n},\label{Dk-to-Ndell-statement}
				\end{align}
				for some functions $\mathcal{Z}_{(m,\underline{n})}$, smooth on any region removed from the origin in $B_{1}$.
			\end{lemma}
			The proofs of all three lemmas in this appendix can be found in \cite{PaHaJa}, for example.
			
			\section{Hardy-Type Inequality and Hardy-Sobolev Embeddings}\label{hardy-sobolev-embeddings}
			One of the main tools we use is a higher order Hardy-type embedding which tells us that a weighted Sobolev space on a domain can be realised in a Sobolev space of lower regularity, in essence sacrificing regularity to remove degeneracy near the boundary. 
			
			\begin{definition}\label{sobolev-norm}For a bounded domain $\mathcal{O}\subset \mathbb R^3$, and $s\in\mathbb{Z}_{\geq0}$, define the Sobolev space $H^{s}(\mathcal{O})$ by
				\begin{align*}
				H^{s}(\mathcal{O})=\left\{F\in L^{2}(\mathcal{O}): \nabla^{\underline{k}}F \ \text{ is weakly in } \ L^2(\mathcal O)
				, \ \underline{k}\in\mathbb Z_{\ge0}^3, \ 0\leq |\underline{k}|\leq b\right\},
				\end{align*}
				with norm given by
				\begin{align*}
				\left\|F\right\|_{H^{s}(\mathcal{O})}^{2}=\sum_{|\underline{k}|=0}^{b}\int_{\mathcal{O}} \left|\nabla^{\underline{k}}F\right|^{2}dx.
				\end{align*}
				The definition of $H^{s}(\mathcal{O})$ can be extended to $s\in\mathbb{R}_{\geq0}$ by interpolation.
			\end{definition}
			
			\begin{definition}\label{weighted-sobolev-norm}For a bounded domain $\mathcal{O}\subset \mathbb R^3$, $\alpha>0$ and $b\in\mathbb{Z}_{\geq0}$, define the weighted Sobolev space $H^{\alpha,b}$ by
				\begin{align*}
				H^{\alpha,b}(\mathcal{O})=\left\{d^{\frac{\alpha}{2}}F\in L^{2}(\mathcal{O}): 
				\nabla^{\underline{k}}F \ \text{ is weakly in } \ L^2(\mathcal O, d^\alpha dx),
				\ \underline{k}\in\mathbb Z_{\ge0}^3, \ 0\leq |\underline{k}|\leq b\right\},
				\end{align*}
				with norm given by
				\begin{align*}
				\left\|F\right\|_{H^{\alpha,b}(\mathcal{O})}^{2}=\sum_{|\underline{k}|=0}^{b}\int_{\mathcal{O}} d^{\alpha}\left|\nabla^{\underline{k}}F\right|^{2}dx,
				\end{align*}
				where $d=d(x,\dell\mathcal{O})$ is the distance function to the boundary on $\mathcal{O}$.
			\end{definition}
			
			
			\noindent Given these definitions, we have the following embedding.
			
			\begin{lemma}{\label{hardy-type-embedding}} Let $\mathcal O$ be as above and let $b\in\mathbb Z_{>0}$ and $0<\alpha\le 2b$. Then the Banach space $H^{\alpha,b}(\mathcal{O})$ embeds continuously in to $H^{b-\frac{\alpha}{2}}(\mathcal{O})$.
			\end{lemma}
			
			\noindent The proof of Lemma~\ref{hardy-type-embedding} is standard and can be found in~\cite{KMP}.
			
			
			Finally, we state the Hardy-Sobolev bounds for $L^\infty$-norm in terms of our energy norms.
			Recall the spaces $\X{b}_{\kappa}$ and $\Y{b}{\kappa}{\ngrad}$ with their associated norm and semi-norm respectively, given in Definition $\ref{energy-function-spaces-definition}$. We also recall the definitions of $\chi$ and $\bar{\chi}$ given in~$\eqref{cutoff-function}$ and~$\eqref{cutoff-function-conjugate}$.
			
			\begin{lemma}\label{L-infinity-energy-space-bound} Let $a\in\mathbb{Z}_{\geq0}$ and $\underline{b}\in\mathbb{Z}_{\geq0}^{3}$, and let $a+|\underline{b}|=j$. Then for each $\kappa\in\{1,2,\dots,N\}$,
				\begin{align}
				\left\|\Ndell{a}{b}F\right\|_{L^{\infty}(\Omega)}\lesssim \left\|F\right\|_{\X{\ceil*{\alpha}+2j+4}_{\kappa}}.\label{L-infinity-energy-space-bound-no-weights-statement-1}
				\end{align}

				\begin{align}
				\left\|\tilde{W}_{\kappa}^{\frac{a}{2}}\Ndell{a}{b}F\right\|_{L^{\infty}(\Omega)}\lesssim \left\|F\right\|_{\X{\ceil*{\alpha}+j+6}_{\kappa}}.\label{L-infinity-energy-space-bound-weights-statement-1}
				\end{align}
				
				\begin{align}
				\left\|\tilde{W}_{\kappa}^{\frac{a}{2}}\nabla\Ndell{a}{b}F\right\|_{L^{\infty}(\Omega)}\lesssim \left\|F\right\|_{\Y{\ceil*{\alpha}+j+6}{\kappa}{\ngrad}}.\label{L-infinity-energy-space-bound-weights-statement-2}
				\end{align}
			\end{lemma}
		
			\begin{proof}[\textbf{Proof}]
				Fix $\kappa\in\{1,2,\dots,N\}$. We begin with $\eqref{L-infinity-energy-space-bound-no-weights-statement-1}$. First, using Lemma $\ref{Ndell-to-Dk-lemma}$ and the Sobolev embedding $H^{2}\hookrightarrow L^{\infty}$,we have
				\begin{align}
				\left\|\Ndell{a}{b}F\right\|_{L^{\infty}(\Omega)}\lesssim\left\|F\right\|_{H^{2+j}(\Omega)}.\label{L-infinity-energy-space-bound-1}
				\end{align}
				Then, we use Lemma $\ref{hardy-type-embedding}$ to obtain
				\begin{align}
				\left\|F\right\|_{H^{2+j}(\Omega)}\lesssim\left\|F\right\|_{H^{\alpha+\ceil*{\alpha}+2j+4,\ceil*{\alpha}+2j+4}(\Omega)}\label{L-infinity-energy-space-bound-2}
				\end{align}
				having noted that
				\begin{align*}
				2+j\leq\ceil*{\alpha}+2j+4-\left(\frac{\alpha+\ceil*{\alpha}+2j+4}{2}\right).
				\end{align*}
				Using $\eqref{translated-W-and-d-equivalence}$, we have
				\begin{align}
				\left\|F\right\|_{H^{\alpha+\ceil*{\alpha}+2j+4,\ceil*{\alpha}+2j+4}(\Omega)}^{2}&=\sum_{|\underline{k}|=0}^{\ceil*{\alpha}+2j+4}\int_{\Omega} d_{\Omega}^{\alpha+\ceil*{\alpha}+2j+4}\left|\nabla^{\underline{k}}F\right|^{2}dx\nonumber\\
				&\lesssim\sum_{|\underline{k}|=0}^{\ceil*{\alpha}+2j+4}\int_{\Omega} \tilde{W}_{\kappa}^{\alpha+\ceil*{\alpha}+2j+4}\left|\nabla^{\underline{k}}F\right|^{2}dx,\label{L-infinity-energy-space-bound-3}
				\end{align}
				where $d_{\Omega}(x)=d(x,\dell\Omega)$, and $\tilde{W}_{\kappa}$ is given in $\eqref{definition-of-translated-W}$. Using $\chi+\bar{\chi}=1$, and Lemma $\ref{Dk-to-Ndell-lemma}$, we have
				\begin{align}
				\sum_{|\underline{k}|=0}^{\ceil*{\alpha}+2j+4}\int_{\Omega} \tilde{W}_{\kappa}^{\alpha+\ceil*{\alpha}+2j+4}\left|\nabla^{\underline{k}}F\right|^{2}dx&\lesssim\sum_{m+|\underline{n}|=0}^{\ceil*{\alpha}+2j+4}\int_{\Omega} \chi\tilde{W}_{\kappa}^{\alpha+\ceil*{\alpha}+2j+4}\left|\Ndell{m}{n}F\right|^{2}dx+\sum_{|\underline{k}|=0}^{\ceil*{\alpha}+2j+4}\int_{\Omega} \bar{\chi}\tilde{W}_{\kappa}^{\alpha+\ceil*{\alpha}+2j+4}\left|\nabla^{\underline{k}}F\right|^{2}dx\nonumber\\
				&\lesssim\sum_{m+|\underline{n}|=0}^{\ceil*{\alpha}+2j+4}\int_{\Omega} \chi\tilde{W}_{\kappa}^{\alpha+m}\left|\Ndell{m}{n}F\right|^{2}dx+\sum_{|\underline{k}|=0}^{\ceil*{\alpha}+2j+4}\int_{\Omega} \bar{\chi}\tilde{W}_{\kappa}^{\alpha}\left|\nabla^{\underline{k}}F\right|^{2}dx\nonumber\\
				&=\left\|F\right\|_{\X{\ceil*{\alpha}+2j+4}}^{2},\label{L-infinity-energy-space-bound-4}
				\end{align}
				where the second bound is because $\tilde{W}_{\kappa}$ is bounded on $\Omega$, and the last line is by definition. Bounds $\eqref{L-infinity-energy-space-bound-1}$--$\eqref{L-infinity-energy-space-bound-4}$ give $\eqref{L-infinity-energy-space-bound-no-weights-statement-1}$.
				
				For $\eqref{L-infinity-energy-space-bound-weights-statement-1}$ we first have
				\begin{align}
				\left\|\tilde{W}_{\kappa}^{\frac{a}{2}}\Ndell{a}{b}F\right\|_{L^{\infty}(\Omega)}\lesssim\left\|\tilde{W}_{\kappa}^{\frac{a}{2}}\Ndell{a}{b}F\right\|_{L^{\infty}(\supp{\chi})}+\left\|\tilde{W}_{\kappa}^{\frac{a}{2}}\Ndell{a}{b}F\right\|_{L^{\infty}(\supp{\bar{\chi}})}.\label{L-infinity-energy-space-bound-5}
				\end{align}
				The estimate
				\begin{align}
				\left\|\tilde{W}_{\kappa}^{\frac{a}{2}}\Ndell{a}{b}F\right\|_{L^{\infty}(\supp{\chi})}\lesssim\left\|F\right\|_{\X{\ceil*{\alpha}+j+6}_{\kappa}}\label{L-infinity-energy-space-bound-6}
				\end{align}
				is standard and can be found in~\cite{KMP,Jang2014}. For the estimate on $\supp{\bar{\chi}}$, we have 
				\begin{align}
				\left\|\tilde{W}_{\kappa}^{\frac{a}{2}}\Ndell{a}{b}F\right\|_{L^{\infty}(\supp{\bar{\chi}})}\lesssim\left\|F\right\|_{H^{2+j}(\supp{\bar{\chi}})}\label{L-infinity-energy-space-bound-7}
				\end{align}
				by first bounding $\tilde{W}_{\kappa}^{\alpha}$ in $L^{\infty}$, and then using Lemma $\ref{Ndell-to-Dk-lemma}$ and the $H^{2}\hookrightarrow L^{\infty}$. embedding. Then we have
				\begin{align}
				\left\|F\right\|_{H^{2+j}(\supp{\bar{\chi}})}^{2}= \sum_{|\underline{k}|=0}^{2+j}\int_{\supp{\bar{\chi}}\cap\supp{\chi}}\chi \left|\nabla^{\underline{k}}F\right|^{2}dx+\sum_{|\underline{k}|=0}^{2+j}\int_{\supp{\bar{\chi}}}\bar{\chi}\left|\nabla^{\underline{k}}F\right|^{2}dx.\label{L-infinity-energy-space-bound-8}
				\end{align}
				On both $\supp{\bar{\chi}}$ and $\supp{\bar{\chi}}\cap\supp{\chi}$, $\tilde{W}_{\kappa}\sim1$. Moreover, since $\supp{\bar{\chi}}\cap\supp{\chi}$ is removed from the origin, we can apply Lemma $\ref{Dk-to-Ndell-lemma}$. Therefore, we have the bound
				\begin{align}
				\left\|F\right\|_{H^{2+j}(\supp{\bar{\chi}})}^{2}&\lesssim\sum_{m+|\underline{n}|=0}^{2+j}\int_{\supp{\bar{\chi}}\cap\supp{\chi}}\chi\tilde{W}_{\kappa}^{\alpha+m} \left|\Ndell{m}{n}F\right|^{2}dx+\sum_{|\underline{k}|=0}^{2+j}\int_{\supp{\bar{\chi}}}\bar{\chi}\tilde{W}_{\kappa}^{\alpha}\left|\nabla^{\underline{k}}F\right|^{2}dx\nonumber\\
				&=\left\|F\right\|_{\X{2+j}_{\kappa}}^{2}\lesssim\left\|F\right\|_{\X{\ceil*{\alpha}+j+6}_{\kappa}}^{2}.\label{L-infinity-energy-space-bound-9}
				\end{align}
				Bounds $\eqref{L-infinity-energy-space-bound-5}$--$\eqref{L-infinity-energy-space-bound-9}$ give $\eqref{L-infinity-energy-space-bound-weights-statement-1}$. The proof for $\eqref{L-infinity-energy-space-bound-weights-statement-2}$ requires similar estimates as for $\eqref{L-infinity-energy-space-bound-weights-statement-1}$ on $\supp{\chi}$ and $\supp{\bar{\chi}}$. The $\supp{\chi}$ bound can be found in~\cite{KMP,Jang2014} as for $\eqref{L-infinity-energy-space-bound-weights-statement-1}$, and the bound for $\supp{\bar{\chi}}$ follows the same strategy as $\eqref{L-infinity-energy-space-bound-7}$--$\eqref{L-infinity-energy-space-bound-9}$, requiring the additional fact that $\njac_{\kappa}\sim1$ on $\Omega$ due to $\eqref{a-priori-bound-jacobian}$. This completes the proof of the lemma.
			\end{proof}
			
		\end{appendices}


\begin{thebibliography}{99}
		
		\bibitem{Auch}
		\textsc{Auchmuty, J. F. G.}: Existence of equilibrium figures. Arch. Rational Mech.
		Anal. 65, 249--261 (1977)
				
		\bibitem{AuchBeals1}
		\textsc{Auchmuty, J. F. G., Beals, R.}: Variational solutions of some nonlinear free
		boundary problems. Arch. Rational Mech. Anal. 43, 255--271 (1971)
		
		\bibitem{AuchBeals2}
		\textsc{Auchmuty, J. F. G., Beals, R.}: Models of rotating stars. Astrophysical J. 165, 79--82 (1971)
		
		\bibitem{BiTr}
		\textsc{Binney, J., Tremaine, R.}: Galactic Dynamics Second Edition. Princeton University Press, Princeton (2008)
		
		
		\bibitem{CafFrie}
		\textsc{Caffarelli, L. A., Friedman, A.}: The shape of axisymmetric rotating fluid.	J. Funct. Anal. 35, 100--142 (1980)
		
		\bibitem{Ch1}
		\textsc{Chandrasekhar, S.}: The equilibrium of distorted polytropes (I). Mon. Not.
		R. Astron. Soc. 93, 390--405 (1933)
		
		\bibitem{Ch2}
		\textsc{Chandrasekhar, S.}: An Introduction to the Study of Stellar Structures. University of Chicago Press, Chicago (1938)
		
		\bibitem{Ch3}
		\textsc{Chandrasekhar, S.}: Ellipsoidal Figures in Equilibrium. Yale University Press, New Haven (1969)
		
		\bibitem{ChanLi}
		\textsc{Chanillo, S., Li, Y.-Y.}: On diameters of uniformly rotating stars. Comm. Math. Phys. 166, no. 2, 417--430 (1994)
		
		\bibitem{ChanWeiss}
		\textsc{Chanillo, S., Weiss, G. S.}: A Remark on the Geometry of Uniformly Rotating Stars. 
		J. Differential Equations. 253, 553-562 (2012)
		
		
		
		
		
		\bibitem{CouLindShk}
		\textsc{Coutand, D., Lindblad, H., Shkoller, S.}: A priori estimates for the free-boundary 3D compressible Euler Equations in physical vacuum. Comm. Math. Phys. 296, no. 2, 559--587 (2010)
		
		\bibitem{CoSh2011} \textsc{Coutand, D., Shkoller, S.}: Well-posedness in smooth function spaces for the moving-boundary 1-D compressible Euler equations in physical vacuum. Comm. Pure Appl. Math. 64, no. 3, 328--366 (2011)
		
		\bibitem{CoSh2012}
		\textsc{Coutand, D., Shkoller, S.}:
		Well-posedness in smooth function spaces for the moving boundary three-dimensional compressible Euler equations in physical vacuum. Arch. Ration. Mech. Anal. 206 no. 2, 515--616 (2012) 
		
		
		
		
		\bibitem{DengLiuYangYao}
		\textsc{Deng, Y., Liu, T.-P., Yang, T., Yao, Z.}: Solutions of Euler-Poisson Equations
		for Gaseous Stars. Arch. Ration. Mech. Anal. 164, no. 3, 261--285 (2002)
		
		\bibitem{DengXiangYang}
		\textsc{Deng, Y., Xiang, J., Yang, T.}: Blowup phenomena of solutions to Euler-Poisson equations. J. Math. Anal. Appl. 286, 295--306 (2003)
		
		
		\bibitem{Dyson1968}
		\textsc{Dyson F. J.}: 
		Dynamics of a Spinning Gas Cloud. J. Math. Mech. 18, no. 1, 91--101 (1968)
		
		\bibitem{EgFo}
		\textsc{Eggers J., Fontelos, A. M.}: The role of self-similarity in singularities of partial differential equations. Nonlinearity. 22, no. 1, R1--R44, (2009)
		
		\bibitem{FedLuoSmol}
		\textsc{Federbush, P., Luo, T., Smoller, J.}: Existence of Magnetic Compressible Fluid Stars. Arch. Ration. Mech. Anal. 215, no. 2, 611--631 (2015)
		
		\bibitem{FrieTurk1}
		\textsc{Friedman, A., Turkington, B.}: Asymptotic estimates for an axisymmetric rotating fluid. J. Funct. Anal. 37, no. 2, 136--163 (1980)
		
		\bibitem{FrieTurk2}
		\textsc{Friedman, A., Turkington, B.}: The oblateness of an axisymmetric rotating fluid. Indiana Univ. Math. J. 29, no. 5, 777--792 (1980)
		
		\bibitem{FrieTurk3}
		\textsc{Friedman, A., Turkington, B.}: Existence and dimensions of a rotating white dwarf. J. Differential Equations. 42, no. 3, 414-437 (1981)
		
		
		\bibitem{FuLin}
		\textsc{Fu, C.-C., Lin, S.-S.}: On the critical mass of the collapse of a gaseous star in spherically symmetric and isentropic motion. Japan J. Indust. Appl. Math. 15, no. 3, 461--469 (1998)
		
		\bibitem{GolWeb}
		\textsc{Goldreich, P., Weber, S.}: Homologously collapsing stellar cores. Astrophys. J. 238, 991 (1980)
		
		\bibitem{Grassin98}
		\textsc{Grassin, M.}: Global smooth solutions to {E}uler equations for a perfect gas. Indiana Univ. Math. J. 47 1397--1432 (1998)
		
		\bibitem{GuLei}
		\textsc{Gu, X., Lei, Z.}: Local Well-posedness of the three dimensional compressible Euler–Poisson equations with physical vacuum. Journal de Mathmatiques Pures et Appliques. 105, no. 5, 662–723 (2016)
		
		\bibitem{Guo1}		
		\textsc{Guo, Y.}: Smooth Irrotational Flows in the Large to the Euler–Poisson System in $R^{3+1}$. Comm. Math. Phys. 195, no. 2, 249--265 (1998)
		
		\bibitem{GuoPau}
		\textsc{Guo, Y., Pausader, B.}: Global Smooth Ion Dynamics in the Euler-Poisson System. Comm. Math. Phys. 303, no. 1, 89--125 (2011)
		
		\bibitem{HaJa1}
		\textsc{Had\v zi\'c, M.,  Jang, J.}: Dynamics of Expanding Gases. Research Institute for Mathematical Science, Kyoto, Kokyuroku, No. 2038, Mathematical Analysis in Fluid and Gas Dynamics (2017)
		
		\bibitem{HaJa2}
		\textsc{Had\v zi\'c, M.,  Jang, J.}:
		Nonlinear stability of expanding star solutions in the radially-symmetric mass-critical Euler-Poisson system. Comm. Pure Appl. Math. 71, no. 5, 827--891 (2018)
		
		\bibitem{HaJa3}
		\textsc{Had\v zi\'c, M., Jang, J.}:
		Expanding large global solutions of the equations of compressible fluid mechanics. Inventiones Math. 214, no. 3, 1205--1266 (2018)
		
		\bibitem{HaJa4}
		\textsc{Had\v zi\'c, M., Jang, J.}:
		A class of global solutions to the Euler-Poisson system. Comm. Math. Phys. 370, no. 2, 475--505 (2019)
		
		
		\bibitem{Heggie}
		\textsc{Heggie, D. C.}: The Classical Gravitational $N$-Body Problem. Encyclopedia of Mathematical Physics. Academic Press, Cambridge, 575--582 (2006)
		
		\bibitem{Hei}
		\textsc{Heilig, U.}: On Lichtenstein's analysis of rotating newtonian stars. Annales de l'I.H.P. Physique th\'{e}orique. 60, no. 4, 457--487 (1994)
		
		
		\bibitem{Jang2008}
		\textsc{Jang, J.}: Nonlinear Instability in Gravitational Euler-Poisson system for $\gamma=6/5$. Arch. Ration. Mech. Anal. 188, 265-307 (2008)
		
		\bibitem{Jang2014}
		\textsc{Jang, J.}: Nonlinear Instability Theory of Lane-Emden stars. Comm. Pure Appl. Math. 67 no. 9,1418--1465 (2014)
		
		\bibitem{JaMa2009} \textsc{Jang, J., Masmoudi, N.}:
		Well-posedness for compressible Euler equations with physical vacuum singularity. Comm. Pure Appl. Math. 62, 1327--1385 (2009) 
		
		\bibitem{JaMa2011} \textsc{Jang, J., Masmoudi, N.}:
		Vacuum in Gas and Fluid dynamics. Proceedings of the IMA summer school on Nonlinear Conservation Laws and Applications, Springer. 315--329 (2011) 
		
		\bibitem{JaMa2012}\textsc{Jang, J., Masmoudi, N.}: Well and ill-posedness for compressible Euler equations with vacuum. J. Math. Phys. 53, no. 11, 115625 (2012) 
		
		\bibitem{JaMa2015}
		\textsc{Jang, J., Masmoudi, N.}:
		Well-posedness of compressible Euler equations in a physical vacuum. Comm. Pure Appl. Math. 68 no. 1, 61--111 (2015)
		
		\bibitem{JangMak1}
		\textsc{Jang, J., Makino, T.}: On slowly rotating axisymmetric solutions of the
		Euler-Poisson equations. Arch. Ration. Mech. Anal. 225, no. 2, 873--900 (2017)
		
		\bibitem{JangMak2}
		\textsc{Jang, J., Makino, T.}: On rotating axisymmetric solutions of the
		Euler-Poisson equations. J. Differential Equations. 266, no. 7, 3942--3972 (2019)
		
		\bibitem{JangStraussWu}
		\textsc{Jang, J., Strauss, W. A., Wu, Y.}: Existence of rotating magnetic stars. Physica D: Nonlinear Phenomena. 397, 65--74 (2019)
		
		\bibitem{KMP} \textsc{Kufner, A., Malgranda, L.,  Persson, L.-E.}: The Hardy inequality.
		Vydavatelsk\' y Servis, Plzen (2007)
		
		\bibitem{Li}
		\textsc{Li, Y.-Y.}: On uniformly rotating stars. Arch. Rational Mech. Anal. 115, 367--393 (1991)
		
		\bibitem{Lich}
		\textsc{Lichtenstein, L.}: Untersuchungen \"{u}ber die Gleichgewichtsfiguren rotierender Fl\"{u}ssigkeiten, deren
		Teilchen einander nach dem Newtonschen Gesetze anziehen. Mathematische Zeitschrift. 36, no. 1,
		481–562 (1933)
		
		\bibitem{Lin} 
		\textsc{Lin, S.-S.}: Stability of gaseous stars in spherically symmetric motion. SIAM J. Math. Anal. 28, no. 3, 539--569 (1997)
		
		
		\bibitem{Lion}
		\textsc{Lions, P. L.}: Minimization Problems in $L^{1}(R^{3})$. J. Funct. Anal. 41, no. 2, 236--275 (1981)
		
		
		
		\bibitem{Liu} 
		\textsc{Liu, T.-P.}: Compressible flow with damping and vacuum. Japan J. Appl. Math. 13, 25-32 (1996)
		
		\bibitem{LiuSmol} \textsc{Liu, T.-P., Smoller, J.}:
		On the vacuum state for isentropic gas dynamics equations. Advances in Math. 1, 345--359 (1980)
		
		\bibitem{LiuYang1} \textsc{Liu, T.-P., Yang, T.}:
		Compressible Euler equations with vacuum. J. Differential Equations. 140, 223-237 (1997)
		
		\bibitem{LiuYang2} \textsc{Liu, T.-P., Yang, T.}:
		Compressible flow with vacuum and physical singularity. Methods Appl. Anal. 7, 495--509 (2000)
		
		\bibitem{LuoSmol1}
		\textsc{Luo, T., Smoller, J.}: Rotating fluids with self-gravitation in bounded domains. Arch. Ration. Mech. Anal. 173, no. 3, 345--377 (2004)
		
		\bibitem{LuoSmol2}
		\textsc{Luo, T., Smoller, J.}: Existence and nonlinear stability of rotating star solutions of the compressible Euler-Poisson equations. Arch. Ration. Mech. Anal. 191, no. 3, 447--496 (2009)
		
		
		
		
		
		\bibitem{LuoXinZeng} \textsc{Luo, T., Xin, Z., Zeng, H.}: Well-posedness for the motion of physical vacuum of the three-dimensional compressible Euler equations with or without self-gravitation. Arch. Ration. Mech. Anal. 213, no. 3, 763--831 (2014)
		
		
		
		
		
		\bibitem{Mak2}
		\textsc{Makino, T.}: Blowing up solutions of the Euler-Poisson equation for the evolution of gaseous stars. Transport Theory Statist. Phys. 21, 615-624 (1992)
		
		\bibitem{MakPer}
		\textsc{Makino, T., Perthame, B.}: Sur les Solution \'{a} Sym\'{e}trie Sph\'{e}rique de l’Equation d’Euler-Poisson pour l’Evolution d’Etoiles Gazeuses. Japan J. Appl. Math. 7, 165--170, (1990)	
		
		\bibitem{MakUkai}
		\textsc{Makino, T. Ukai, S.}: Sur l’existence des solutions locales de l’equation d’Euler-Poisson pour l’évolution d’étoiles gazeuses. J. Math. Kyoto Univ. 27, no. 3, 387--399 (1987)
		
		\bibitem{MakUkaiKawa}
		\textsc{Makino, T., Ukai, S., Kawashima, S.}: Sur la solution \`a support compact de l'\'equations d'{E}uler compressible. Japan J. Appl. Math. 3, 249-257 (1986)
		
		\bibitem{McCann}
		\textsc{McCann, R. J.}: Stable rotating binary stars and fluid in a tube. Houston J. Math. 32, no. 2, 603--631 (2006)
		
		
		\bibitem{MiaoShah}
		\textsc{Miao, S., Shahshahani, S.}: On tidal energy in Newtonian two-body motion. Available on ArXiv at: https://arxiv.org/abs/1708.04307 (2017)
		
		\bibitem{Milne}
		\textsc{Milne, E. A.}: The equilibrium of a rotating star. Mon. Not.
		R. Astron. Soc. 83, 118--147 (1923)
		
		
		
		
		
		
		
		\bibitem{Ov1956}
		\textsc{Ovsiannikov, L. V.}:
		New solution of hydrodynamic equations. Dokl. Akad. Nauk SSSR Vol lll, N l, 47--49 (1956) 
		
		\bibitem{PaHaJa}
		\textsc{Parmeshwar, S., Had\v{z}i\'{c}, M., Jang, J.}: Global expanding solutions of compressible Euler equations with small initial densities. Available on ArXiv at: https://arxiv.org/abs/1904.01122 (2019)
		
		\bibitem{Rein1}		
		\textsc{Rein, G.}: Nonlinear instability of gaseous stars. Arch. Ration. Mech. Anal. 168, no. 2, 261–285 (2003)
		
		
		\bibitem{RicHaJa}
		\textsc{Rickard, C., Had\v{z}i\'{c}, M., Jang, J.}: Global existence of the nonisentropic compressible Euler equations with vacuum boundary surrounding a variable entropy state. Available on ArXiv at: https://arxiv.org/abs/1907.01065 (2019)
		
		%
		
		\bibitem{Roz} 
		\textsc{Rozanova, O.}: 
		Solutions with linear profile of velocity to the Euler equations in several dimensions. Hyperbolic problems: theory, numerics, applications. Springer, Berlin, 861--870 (2003) 
		
		\bibitem{ShSi2017}
		\textsc{Shkoller, S., Sideris, T. C.}:
		Global existence of near-affine solutions to the compressible Euler equations. Arch. Ration. Mech. Anal. 234 no. 1, 115--180 
		
		
		\bibitem{Serre1997}
		\textsc{Serre, D.}:
		Solutions classiques globales des \'equations d'Euler pour un fluide parfait compressible. Annales de l'Institut Fourier. 47, 139--153 (1997)
		
		\bibitem{Serre2}
		\textsc{Serre D.}: Expansion of a compressible gas in vacuum. Bull. Inst. Math. Acad. Sin. (N.S.) 10, no. 4, 695--716 (2015) 
		
		
		
		
		
		\bibitem{Sid1}
		\textsc{Sideris, T. C.}: Formation of singularities in three-dimensional compressible fluids. Comm. Math. Phys. 101, no. 4, 475--485 (1985)
		
		\bibitem{Sid2}
		\textsc{Sideris, T. C.}: Spreading of the free boundary of an ideal fluid in a vacuum. J. Differential Equations 257. no. 1, 1--14 (2014)  
		
		\bibitem{Sid3}
		\textsc{Sideris, T., C.}: Global existence and asymptotic behavior of affine motion of 3D ideal fluids surrounded by vacuum. Arch. Ration. Mech. Anal. 225, no. 1, 141--176 (2017)
		
		\bibitem{SieMos}
		\textsc{Siegel, C. L., Moser, J. K.}: Lectures on Celestial Mechanics. Springer, Berlin (1971)
		
		\bibitem{StraussWu1}
		\textsc{Strauss, W. A., Wu, Y.}: Steady states of rotating stars and galaxies. SIAM J. Math. Anal. 49, no. 6, 4865--4914 (2017)
		
		\bibitem{StraussWu2}
		\textsc{Strauss, W. A. Wu, Y.}: Rapidly rotating stars. Comm. Math. Phys. 368, no. 2, 701--721 (2019)
		
		\bibitem{Von}
		\textsc{von Zeipel, H.}: The radiative equilibrium of a slightly oblate rotating star.
		Mon. Not. R. Astron. Soc. 84, 684--702 (1924)
		
		\bibitem{Wu1}
		\textsc{Wu, Y.}: Existence of rotating planet solutions to the Euler-Poisson equations with an inner hard core. Arch. Ration. Mech. Anal. 219, no. 1, 1--26 (2016)
		
		\bibitem{Wu2}
		\textsc{Wu, Y.}: On rotating star solutions to the non-isentropic Euler-Poisson equations. J. Differential Equations. 259, no. 12, 7161--7198 (2015)
				
		\bibitem{ZelNov}
		\textsc{Zel'dovich, Ya. B., Novikov, I. D.}: Relativistic Astrophysics Vol. 1: Stars and Relativity. University of Chicago Press, Chicago (1971)
		
		
	\end{thebibliography}
\end{document}